\documentclass{amsart}
\usepackage[colorlinks=true,linkcolor=blue,citecolor=green!65!black]{hyperref}
\usepackage{cite}
\usepackage{mathtools,bm,comment,amssymb,enumitem,xcolor,array,float,subcaption,float,tikz,pgfplots,siunitx,booktabs,accents,multirow,epsfig}
\usepackage[capitalise,nameinlink]{cleveref}
\usepackage[english]{babel}
\usepackage[margin=1in]{geometry}
\usepackage[foot]{amsaddr}
\usepackage[linesnumbered,ruled]{algorithm2e}
\usepackage[noend]{algpseudocode}
\usepackage[title]{appendix}

\numberwithin{equation}{section}
\crefname{subsection}{Subsection}{Subsections}
\crefformat{equation}{(#2#1#3)}
\crefformat{align}{(#2#1#3)}
\crefformat{enumi}{(#2#1#3)}
\crefname{figure}{Figure}{Figures}

\newtheorem{remark}{Remark}

\newtheorem{lemma}{Lemma}
\newtheorem{theorem}{Theorem}
\newtheorem{definition}{Definition}

\makeatletter
\@namedef{subjclassname@2020}{\textup{2020} Mathematics Subject Classification}
\makeatother

\newcommand{\revision}[1]{#1}

\makeatletter
\renewcommand{\email}[2][]{%
	\ifx\emails\@empty\relax\else{\g@addto@macro\emails{,\space}}\fi%
	\@ifnotempty{#1}{\g@addto@macro\emails{\textrm{(#1)}\space}}%
	\g@addto@macro\emails{#2}%
}
\makeatother

\DeclareMathOperator*{\esssup}{ess\,sup}

\newcommand{\tr}{\operatorname{tr}}
\newcommand{\doublehookrightarrow}{%
    \mathrel{\mathrlap{{\mspace{4mu}\lhook}}{\hookrightarrow}}}

\renewcommand{\vec}[1]{\pmb{#1}}

\newcommand{\dd}{\mathop{}\!\mathrm{d}}

\newcommand{\dt}{\frac{\dd}{\dd\mathrm{t}}}
\newcommand{\p}{\partial}
\newcommand{\symgrad}{\bm{\varepsilon}}

\newcommand{\R}{\mathbb{R}} 
\newcommand{\N}{\mathbb{N}} 
\newcommand{\PY}{\mathbb{Y}} 
\newcommand{\PZ}{\mathbb{Z}} 
\newcommand{\PW}{\mathbb{W}} 
\newcommand{\Cphi}{\vartheta}
\newcommand{\Cmu}{\varrho}
\newcommand{\Cu}{\varsigma}
\newcommand{\Cpsi}{\varkappa}
\newcommand{\Cchi}{\varpi}
\newcommand{\Func}{H}
\renewcommand{\L}{\mathcal{L}} 
\newcommand{\W}{\mathcal{W}} 
\renewcommand{\H}{\mathcal{H}} 
\newcommand{\C}{\mathcal{C}} 
\newcommand{\Ba}{\mathcal{U}} 
\newcommand{\Diru}{\Sigma_{1}} 
\newcommand{\Dirpsi}{\Sigma_{2}} 
\newcommand{\HS}{\mathcal{H}} 
\newcommand{\HSV}{\mathcal{V}} 
\newcommand{\HSVD}{\mathcal{V}'} 

\newcommand{\FDphiM}{\p_t^{\alpha}(\phi^m-\phi^m_0)}
\newcommand{\FDmuM}{\p_t^{\alpha}(\mu^m-\mu^m_0)}
\newcommand{\FDdivuM}{\p_t^{\alpha}(\nabla\cdot \vec{u}^m-\nabla\cdot\vec{u}^m_0)}
\newcommand{\FDsymuM}{\p_t^{\alpha}\left(\symgrad(\vec{u}^m)-\symgrad(\vec{u}^m_0)\right)}

\allowdisplaybreaks

\begin{document}
\title[On a Subdiffusive Tumour Growth Model with Fractional Time Derivative]{On a Subdiffusive Tumour Growth Model \\ with Fractional Time Derivative}

\author{Marvin Fritz$^1$, Christina Kuttler$^1$, Mabel L. Rajendran$^{1,*}$, \\ Barbara Wohlmuth$^1$, Laura Scarabosio$^2$}
\address{${}^1$Department of Mathematics, Technical University of Munich}
\address{${}^2$Institute for Mathematics, Astrophysics and Particle Physics, Radboud University}

\email{\{fritzm, kuttler, rajendrm, wohlmuth\}@ma.tum.de; scarabosio@science.ru.nl}
\thanks{${}^*$Corresponding author}


\maketitle

\begin{abstract}
In this work, we present and analyse a system of coupled partial differential equations, which models tumour growth under the influence of subdiffusion, mechanical effects, nutrient supply, and chemotherapy. The subdiffusion of the system is modelled by a time fractional derivative in the equation governing the volume fraction of the tumour cells. The mass densities of the nutrients and the chemotherapeutic agents are modelled by reaction diffusion equations. We prove the existence and uniqueness of a weak solution to the model via the Faedo--Galerkin method and the application of appropriate compactness theorems. Lastly, we propose a fully discretised system and illustrate the effects of the fractional derivative and the influence of the fractional parameter in numerical examples.
\end{abstract}

\section{Introduction}
\par Mathematical modelling to understand the development of tumour cells and their dynamics is of great importance as it in turn helps in devising appropriate treatment methods. In this study, we introduce fractional time derivatives in a tumour growth model with mechanical coupling. The fractional derivatives have the role of accounting for anomalous diffusion, more precisely subdiffusion, seen in tumour growth. 

The tumour microenvironment has a strong influence on tumour cell proliferation and migration \cite{Balkwill12,wang17,yuan16}. Depending on the environment of the surrounding host tissue, tumour not only migrates using typical Fickian diffusion, but it also migrates more generally using subdiffusion, superdiffusion, and even ballistic diffusion. Haptotaxis and chemotaxis, which are initiated by extracellular matrix and nutrient supply respectively, and cell-cell adhesion all drastically affect a tumour's diffusion mode when a tumour invades its surrounding host tissue and proliferates. In particular, experimental results by \cite{jiang2014anomalous} both from \textsl{in vitro} and \textsl{in vivo} show evidence of anomalous diffusion in cancer growth. Taking the average radius of the tumour to be an indicator of the root-mean-squared displacement of the cells, they observed anomalous diffusion in \textsl{in vitro} experiments of growing cultured cells from the breast line, and in the clinical data from patients with adrenal tumour and liver tumour. 

Anomalous diffusion is a diffusion process with a nonlinear relation between mean squared displacement and time, unlike the normal diffusion process where the relation is linear. In the microscopic setting, the diffusion processes are presented by the continuous time random walk (CTRW) model \cite{tahir1983correlated}, wherein the particle jumps in random directions, and the waiting time before the next jump and jump lengths are given by random processes. We have the following three relevant examples of CTRW. When the mean of the probability density function (PDF) of the waiting time (first moment) and the variance of the PDF of the jump length (second moment) are finite, in the long-time limit we have a behaviour described by an integer-order diffusion equation. In this case, the solution for a point initial condition is a Gaussian PDF, and the mean square displacement (MSD) has a linear dependence on time. A PDF of the waiting time $\sim t^{-1-\alpha}$ as $t\rightarrow \infty$ with $0<\alpha<1$, results, in the continuum limit, in a time fractional diffusion equation, represented by a power-law dependence of MSD on time of the form $<x^2(t)>\sim t^{\alpha}$ leading to subdiffusive behaviour. A PDF of jump length $\sim |\vec{x}|^{-1-\beta}$ as $|\vec{x}|\rightarrow\infty$ with $0<\beta<2$ gives us in the long-time limit a behaviour described by fractional diffusion equations in space, leading to superdiffusive behaviour. Fractional differential equations were obtained from the CTRW formulation in \cite{compte1996stochastic} and \cite{metzler2000random}. The description of reactions that take place in systems with anomalous diffusion is discussed in \cite{henry2006anomalous,seki2003recombination, yuste2004reaction,nepomnyashchy2016mathematical}. These examples of CTRW discussed above are adapted to cancer modelling by the migration proliferation dichotomy observed in the development of cancer cells in \cite{iomin2005superdiffusion,iomin2005fractional,iomin2007fractional,fedotov2007migration}. We consider a subdiffusion limited reaction equation for the density of tumour cells, in contrast to the normal reaction diffusion for nutrients and chemotherapeutic density, to take into account the memory effects of cells. This involves the introduction of the Riemann--Lioville fractional derivative, which has the memory kernel in its definition, in the flux and reaction terms in the equation concerning tumour density, as seen in \cite{iomin2015continuous}, resulting in a multi-order system of fractional differential equations. The model can be modified to the one with Caputo fractional derivative assuming sufficient regularity as seen in \cite{yuste2004reaction}.

\par It is important to incorporate mechanical effects in tumour growth model since the growth of the tumour increases mechanical stress due to the surrounding host tissues, which in turn impede the further growth of the tumour. Experimental evidence can be seen in \cite{helmlinger1997solid}, where multi-cellular spheroids were grown in agar gel concentrations ranging from 0\% to 1\% and increasing the agar concentration resulted in the inhibited expansion of the spheroid as the substrate stiffness increased. In the literature, reaction-diffusion models with mechanical coupling are seen in \cite{lima2016selection,lima2017selection,faghihi2020coupled,hormuth2018mechanically} for modelling tumour growth. In our model, we incorporate the mechanical effects in a similar way to the aforementioned papers.

\par After having introduced the mathematical model, we proceed with analysing existence and uniqueness of a weak solution. We remark that, while the mathematical analysis of Cahn--Hilliard equations with mechanical effects is well addressed in the literature see, e.g. \cite{miranville2001long,carrive1999cahn, garcke2003cahn, garcke2005mechanical, garcke2019phase}, the analysis of reaction-diffusion equations with mechanical coupling is not straightforward. The traditional Caputo derivative, which is valid for absolutely continuous functions, is extended to a wider class of functions through various generalisations in the study of weak solutions to fractional differential equations. For instance, some generalisations of the Caputo derivative in the literature are given in \cite{kilbas2006theory,allen2016parabolic,gorenflo2015time,li2018generalized,akilandeeswari2017solvability}, and they all reduce to the traditional one under the assumption of sufficient regularity of the function. In the analysis, we use the one in \cite{kilbas2006theory,diethelm2010analysis}, which relies on Riemann--Liouville derivatives and is, in contrast to the traditional one, also valid for some functions that do not necessarily have the first derivative. Using Galerkin methods for showing the existence of weak solution to partial fractional differential equations is quite popular and it is seen for instance in \cite{djilali2018galerkin,ouedjedi2019galerkin,zacher2009weak,zacher2019time,mclean2019well,lean2020regularity,li2018some,manimaran2019numerical} and for tumour growth models in \cite{garcke2017well,fritz2018unsteady,fritz2019local}. The key variations from that of integer-order in the analysis are the fractional Gronwall Lemma \cite{lean2020regularity}, some special estimates due to the lack of chain rule for fractional derivatives \cite{vergara2008lyapunov}, and compactness theorems similar to the Aubin--Lions theorem  \cite{li2018some}. The multi-order ordinary fractional differential system is well addressed \cite{diethelm2010analysis}, however there is not much in the literature on the multi-order partial differential system.

\par The main novelties of our work can be summarised as follows: (a) we consider a nonlinear reaction-diffusion system with a fractional time derivative, capable of modelling subdiffusion in tumour progression, and we illustrate, by numerical simulations, its flexibility in describing the tumour dynamics by varying the fractional exponent; (b) we provide a rigorous mathematical analysis of the existence and uniqueness of a weak solution to this model, the original aspects consisting in the treatment of the fractional time derivative and of the mechanical coupling.

\par This exposition has the following structure:
Section \ref{sec:mathematical_modelling} gives the mathematical modelling of the tumour growth with mechanical effects and fractional derivative. In Section \ref{sec:preliminaries}, we introduce the notations and preliminary results needed in the later sections. The mathematical analysis of the model giving existence and uniqueness of the weak solution using Galerkin methods is worked out in Section \ref{sec:mathematical_analysis}. The numerical discretisation of the model using finite element method for space and finite differences in time with a convolution quadrature formula for the fractional derivative is given in Section \ref{sec:numerical_discretization}. The results of the numerical experiments in Section \ref{sec:numerical_simulation} show the effects of the fractional derivative and the mechanical coupling in the model.
\section{Mathematical modelling}\label{sec:mathematical_modelling}
We consider a material body $\mathcal{B}$ composed of two constituents, tumour cells and healthy cells, which occupy a common portion of a bounded Lipschitz domain $\Omega\subset\R^d$, $d=2$, during the time $t\in[0,T]$. Nutrients such as oxygen and glucose in $\Omega$ nourish both the healthy and tumour cells. The increasing number of tumour cells by the intake of nutrients and interaction with the surrounding healthy cells increases the mechanical stress which in turn affects the mobility of the tumour. 
Treatment for cancer is given by chemotherapy in which the drug diffuses through the region $\Omega$ and kills the fast growing cancerous cells.
\par The quantities of interest to us are as follows: 
the mass density of the tumour cells per unit volume $\rho\phi$, where $\phi:\Omega\times[0,T]\rightarrow[0,1]$ is the volume fraction of the tumour cells in $\mathcal{B}$ and $\rho$ is the mass density of the tumour cells, 
the displacement field $\vec{u}:\Omega\times[0,T]\rightarrow\R^d$, the mass density of the nutrients $\psi:\Omega\times[0,T]\rightarrow\R$  and 
 the mass density of the chemotherapeutic agents $\chi:\Omega\times[0,T]\rightarrow\R$.
\subsection{Evolution of tumour}
Time evolution of the physical system must obey the laws of conservation of mass, linear and angular momentum, energy and the second law of thermodynamics. We ignore the temperature and thermal effects and proceed as done in \cite{lima2016selection}.
\begin{itemize}
  \item \textit{Conservation of mass:}
\begin{eqnarray}\label{law:phi_mass}
    \p_t(\rho \phi) + \nabla \cdot (\rho\phi\vec{v}) = \rho (S-\nabla\cdot \vec{J}),
\end{eqnarray}
where $\vec{v}$ is the velocity field, $\rho S$ is the mass density supplied by other constituents, which encompasses proliferation of tumour cells and their death due to chemotherapy treatment, and $\rho \vec{J}$ is the mass flux over the boundary of $\Omega$ which we denote by $\p\Omega$.
  \item \textit{Conservation of linear and angular momentum:}
    \begin{equation}\label{law:linear_angular_momentum}
    \begin{aligned}
    \p_t (\rho \phi\vec{v}) + \nabla \cdot (\rho \phi\vec{v}\otimes \vec{v}) = \nabla \cdot \vec{T} + \rho\phi\vec{b} + \vec{p}, \\
    \vec{T} - \vec{T}^{t} = \vec{m},
    \end{aligned}
    \end{equation}
where $\vec{T}$ is the Cauchy stress tensor for the tumour, $\vec{b}$ is the body force, $\vec{p}$ momentum supplied by other constituents, $\vec{m}$ is the intrinsic moment of momentum, and $\vec{T}^t$ denotes the transpose of $\vec{T}$.
\end{itemize}
\par The total energy of the system $\tilde{\Psi}$ consists of the Ginzburg--Landau component $\Psi(\phi,\nabla\phi)$ 
depending only on $\phi$ and its gradient $\nabla\phi$, and the stored energy potential $W(\phi,\symgrad(\vec{u}))$ depending on $\phi$ and the strain measure $\symgrad(u)$. Assuming small deformations, we consider the potentials 
\begin{gather*}
    \tilde{\Psi} = \int_{\Omega} \Psi(\phi,\nabla\phi) + W(\phi,\symgrad(\vec{u}))\,\dd \vec{x}, \\
    \Psi(\phi,\nabla\phi) = \frac{c}{2}\phi^2,\quad W(\phi,\symgrad(\vec{u})) = \frac{1}{2} \symgrad:\vec{C}(\phi)\symgrad + \symgrad:\overline{\vec{T}}(\phi),
\end{gather*}
where $c>0$ is a constant, $\symgrad(\vec{u}) = \frac{1}{2}(\nabla\vec{u}+\nabla\vec{u}^t)$,
 $\overline{\vec{T}}(\phi)=\lambda \phi \mathbb{I}$ is the symmetric compositional stress tensor, $\lambda>0$ depending on the tumour growth rate and $\mathbb{I}$ being the identity matrix, $\vec{C}(\phi)$ is the linear elastic inhomogeneous material tensor, and the operator $:$ denotes the inner product for second-order tensors.  

\par The first variations of the energy functional with respect to $\phi$ and $\symgrad$ define the chemical potential, and the stress tensor respectively
\begin{equation}\label{eqn:chempot_stress}
\begin{aligned}
\mu = \frac{\delta \Psi}{\delta \phi} + \frac{\delta W}{\delta \phi}, \quad
\vec{T} = \frac{\delta W}{\delta \symgrad}.  
\end{aligned}
\end{equation}
The effects of the elastic deformation on the movement of tumour cells is prescribed by the term $\lambda\nabla\cdot\vec{u}$ in the chemical potential.
\par To incorporate subdiffusion, we introduce fractional derivatives in the mass flux and mass sources. 
Modelling the subdiffusion and proliferation of cancer cells can lead to a linear fractional partial differential equation through a comb model with proliferation in one dimension \cite{iomin2015continuous}. On a microscopic level, subdiffusion-limited reaction is modelled in \cite{seki2003recombination, yuste2004reaction} by having fractional derivatives in flux and reaction terms.

\par Motivated by the previous models, the subdiffusion limited reaction for tumour mass density takes the form, 
\begin{equation}\label{frac_flux_source}
\begin{aligned}
    \vec{J}=-M_\phi(\vec{x})\p_t^{1-\alpha} \nabla\mu,\quad
    S = N_\phi \p_t^{1-\alpha}f(\phi,\psi) - P_\phi \p_t^{1-\alpha} g(\phi,\chi),
\end{aligned}
\end{equation}
for $\alpha\in(0,1)$, where $M_\phi:\Omega\rightarrow\R^+$ is such that $cM_\phi$ is the mobility of tumour cells, $f,g:\Omega\times[0,T]\rightarrow\R$ model the uptake of nutrient and chemotherapic by the tumour cells, $N_\phi>0$ is the rate at which the tumour cells proliferate by using the nutrients, and $P_\phi>0$ is the rate at which the tumour cells die due to the chemotherapy treatment. The operator $\p^{1-\alpha}_t$ is the Riemann--Liouville fractional derivative and is defined for a function $\varphi:\Omega\times[0,T]\rightarrow\R$, as
\begin{equation}\label{Def:frac_Riemann}
    \begin{aligned}
    \p^{\alpha}_t \varphi(t)& = \partial_t (g_{1-\alpha}*\varphi)(t),
    \end{aligned}
\end{equation}
where the kernel is defined by
$$g_{\alpha}(t):=\begin{cases}t^{\alpha-1}/\Gamma(\alpha),&\alpha>0,\\\delta(t),&\alpha=0,\end{cases}$$ where $\delta(t)$ is the Dirac delta distribution and $*$ denotes the convolution on the positive halfline with respect to the time variable, i.e., $(g_\alpha*\varphi)(t)=\int_0^t g_\alpha(t-s)\varphi(s)\dd s$. If $\varphi$ is sufficiently smooth, then we have 
\begin{equation}\label{Def:frac_Caputo}
    \begin{aligned}
    \p^{\alpha}_t(\varphi(t)-\varphi_0)=g_{1-\alpha}*\p_t\varphi(t),
    \end{aligned}
\end{equation}
where \revision{$\varphi_0$ is a given initial value}. The right-hand side is the classical Caputo fractional derivative. The formulation on the left hand side, which expresses the Caputo fractional derivative in terms of Riemann--Liouville fractional derivative, has the advantage that it requires less regularity of $\varphi$ than the classical definition.

\par We reduce the complexity of the system by using the common simplifying assumptions as in \cite{lima2016selection}: the tumour and healthy cells have constant mass density $\rho=\rho_0$, $\vec{m}=\vec{0}$, i.e., the material is monopolar, no body force, i.e., $\vec{b}=\vec{0}$, 
\revision{we neglect the terms with $\vec{v}\otimes \vec{v}$ and $\vec{p}$ by not considering the inertial effects, and we further assume that the mechanical equilibrium is attained on a much faster time scale than diffusion takes place, i.e., the term $\rho_0 \vec{v} \p_t \phi$ on the left hand side in the linear momentum equation vanishes}. For ease of technical difficulties, we assume that the tumour is an isotropic and homogeneous $\vec{C}(\phi)\equiv \vec{C}$ material, and so $\vec{C}$ takes the form $\vec{C}\symgrad=2G\symgrad+\frac{2G\nu}{1-2\nu} \tr\symgrad \mathbb{I}$, where $G>0$ denotes the shear modulus, while $\nu<\frac{1}{2}$ is the Poisson's ratio. This assumption assures that the energy functional $W(\phi,\symgrad)$ is convex in both its variables, which is required in using Lemma \ref{Lem_basic_inequality_fractional_chain_rule}, which is an analogous result to the chain rule in fractional derivatives for providing estimates for $\phi$. A more general energy functional is considered in \cite{garcke2003cahn} with integer-order derivatives.

\par 
\revision{Along with these assumptions, we integrate \eqref{law:phi_mass}, take $\p_t^\alpha$ on both sides of equation \eqref{law:phi_mass} and use the semigroup property of the kernel $g_{1-\alpha}*g_{\alpha}=g_1=1$ in the following way
    \begin{equation*}
        g_{1-\alpha}*\p_t(g_{\alpha}*\varphi) = \p_t(g_{1-\alpha}*g_{\alpha}*\varphi) - g_{1-\alpha}(t)(g_{\alpha}*\varphi)(0) = \p_t(1*\varphi) = \varphi,
    \end{equation*}
    assuming sufficient smoothness on the functions, to obtain from \eqref{law:phi_mass}--\eqref{frac_flux_source} the system}
\begin{subequations}\label{model:tumour_mechcoup_frac}
    \begin{align}
     \p_t^{\alpha}(\phi-\phi_0) &= \nabla\cdot\left( M_\phi(\vec{x}) \nabla\mu\right) + N_\phi f(\phi,\psi) - P_\phi g(\phi,\chi),\label{model:tumour}\\
     \mu &= c\phi + \lambda \nabla\cdot \vec{u},\label{model:chemicalpot}\\
    \vec{0}&=\nabla \cdot \left( 2G \symgrad(\vec{u}) + \frac{2G\nu}{1-2\nu}\text{tr}(\symgrad(\vec{u}))\mathbb{I} + \lambda \phi\mathbb{I}\right),\label{model:displacement}
    \end{align}
\end{subequations}
where $\phi_0$ is a given data, playing the role of initial condition. 
\begin{remark}\label{rem:decoupling}
If we assume smoothness on all involved variables, we can formally take the divergence in the deformation equation \eqref{model:displacement} and conclude that
$$\left(G + \frac{G}{1-2\nu} \right)\Delta (\nabla\cdot \vec{u}) = -\lambda \Delta \phi.$$ 
If $M_\phi$ is a constant, then, by substitution into \eqref{model:tumour}--\eqref{model:chemicalpot}, we obtain
$$\p_t^{\alpha}(\phi-\phi_0) = M_\phi  \left( c  - \frac{\lambda^2(1-2\nu)}{2G(1-\nu)}\right)  \Delta \phi + N_\phi f(\phi,\psi) - P_\phi g(\phi,\chi).$$
We note that in this case the equation for $\phi$ is independent of $\vec{u}$. Further this also suggests that we require at least $c > \frac{\lambda^2(1-2\nu)}{2G(1-\nu)}$ to conclude the existence of a solution. We indeed see in Section \ref{sec:mathematical_analysis} that we require a slightly stronger condition on $c$.
\end{remark}

\subsection{Evolution of nutrient}
The nutrient mass density is assumed to obey a reaction-diffusion equation, as standard \cite[Ch. 5 and 10]{preziosi}
\begin{equation}\label{model:nutrient}
    \frac{\partial \psi}{\partial t}=\nabla\cdot (M_\psi(\vec{x}) \nabla \psi) + S_\psi(\vec{x},t) - N_\psi f(\phi,\psi),
\end{equation}
$M_\psi:\Omega\rightarrow\R^+$ is the mobility of the nutrients, $S_\psi:\Omega\times[0,T]\rightarrow\R$ denotes the external source of nutrients over the volume, $N_\psi>0$ denotes the rate at which nutrients are consumed by the tumour cells, $f(\phi,\psi) =  \frac{\phi (1-\phi) \psi}{K_\psi+\psi}$ is a monod equation combined with the term $(1-\phi)$ that ensures that $\phi$, which is a volume fraction, does not take values greater than 1.
The parameter $K_\psi>0$ is the monod half saturation constant, corresponding to that nutrient mass density, where the nutrient-dependent growth takes its half maximum value.

\subsection{Evolution of chemotherapy}
The mass density of chemotherapy is assumed to be governed by a reaction-diffusion equation 
\begin{equation}\label{model:chemo}
\frac{\partial \chi}{\partial t}=\nabla\cdot (M_\chi(\vec{x}) \nabla \chi) - N_\chi \chi + S_\chi(\vec{x},t) - P_\chi g(\phi,\chi),
\end{equation}
where $M_\chi:\Omega\rightarrow\R^+$ is the mobility of chemotherapeutic agents, $S_\chi:\Omega\times[0,T]\rightarrow\R$ is the external supply of chemotherapeutic over the domain, $N_\chi>0$ is the rate at which the chemotherapeutic agents are degraded, $P_\chi>0$ denotes the rate at which chemotherapeutic agents act and are blocked later by killing tumour cells. The term $g(\phi,\chi) = \frac{\phi (1-\phi) \chi}{K_\chi+\chi}$ includes, analogously to the nutrient uptake, a saturation effect, including also that mainly cells in a certain growth phase are sensible to the chemotherapy. The parameter $K_\chi>0$ is the density of chemotherapeutic agents when they reach their half maximum value. 

\par Finally, collecting \eqref{model:tumour_mechcoup_frac}--
\eqref{model:chemo}, the tumour evolution is governed by the system
\begin{subequations}\label{System}
    \begin{align}
     \p_t^{\alpha}(\phi-\phi_0) &= \nabla\cdot\left( M_\phi(\vec{x}) \nabla\mu\right) + N_\phi f(\phi,\psi) - P_\phi g(\phi,\chi),\label{eqn:phieq}\\
     \mu &= c\phi + \lambda \nabla\cdot \vec{u},\label{eqn:mueq}\\
    \vec{0}&=\nabla \cdot \left( 2G \symgrad(\vec{u}) + \frac{2G\nu}{1-2\nu}\text{tr}(\symgrad(\vec{u}))\mathbb{I} + \lambda \phi\mathbb{I}\right),\label{eqn:ueq}\\
     \frac{\partial \psi}{\partial t}&=\nabla\cdot (M_\psi(\vec{x}) \nabla \psi) + S_\psi(\vec{x},t) - N_\psi f(\phi,\psi),\label{eqn:psieq}\\
     \frac{\partial \chi}{\partial t}&=\nabla\cdot (M_\chi(\vec{x}) \nabla \chi) - N_\chi \chi + S_\chi(\vec{x},t) - P_\chi g(\phi,\chi), \label{eqn:chieq}
    \end{align}
\end{subequations}
in $\Omega$. We add to this system the following initial and boundary conditions
 \begin{subequations}  \label{eq:data}
 \begin{align}
 (\phi,\psi,\chi) = (\phi_0,\psi_0,\chi_0)&\text{ on } \Omega\times\{t=0\},\label{eq:IC}\\
 \nabla \mu\cdot\vec{n}=0&\text{ on } \revision{\p\Omega}\times (0,T), \label{eq:Neumannbc_mu}\\
 \vec{u}=0&\text{ on } \Diru\times (0,T),|\Diru|>0,\label{eq:Dirichletbc_u}\\
 \left(2G \symgrad(\vec{u}) + \frac{2G\nu}{1-2\nu}\text{tr}(\symgrad(\vec{u}))\mathbb{I} + \lambda \phi\mathbb{I}\right)\cdot\vec{n}= \vec{0} &\text{ on } \p\Omega\backslash\Diru\times (0,T),\label{eq:Neumannbc_u}\\
\psi=\tilde{\psi_b},\  \chi=\tilde{\chi_b}&\text{ on } \Dirpsi\times (0,T),\label{eq:Dirichletbc_psi_chi}\\
 M_\psi\nabla\psi\cdot \vec{n}=\psi_{b},\ M_\chi\nabla\chi\cdot \vec{n}=\chi_{b} &\text{ on } \p\Omega\backslash\Dirpsi\times (0,T),\label{eq:Neumannbc_psi_chi}
 \end{align}
 \end{subequations}
 where $\vec{n}$ denotes the outer normal to $\Omega$ and $\Diru,\Dirpsi\subset\p\Omega$ are parts of the boundary $\p\Omega$ with non-zero measures. 
We assume no-flux boundary conditions for the chemical potential and Dirichlet--Neumann mixed boundary condition for the displacement, density of nutrient and chemotherapy. The homogeneous Dirichlet condition on the part of the boundary $\Diru$ for displacement accounts for the presence of a rigid part of the body such as bone, which prevents the variations of the displacement. The non-homogeneous Dirichlet condition on part of the boundary $\Dirpsi$ for density of nutrient and chemotherapy accounts for the concentration supply from blood vessels. In the rest of the boundary, the more natural Neumann boundary condition is applied. \revision{The problem with non-homogeneous Dirichlet condition for density of nutrient and chemotherapy can be converted to a problem to have homogeneous Dirichlet boundary conditions by taking $\tilde{\psi}=\psi-\tilde{\psi}_b$ and similarly for $\chi$. Therefore we assume $\tilde{\psi}_b=\tilde{\chi}_b=0$.}

 \begin{remark}
We assumed that the displacement $\vec{u}$ vanishes on $\Diru \subset \partial \Omega$, that means we imposed a homogeneous Dirichlet boundary on $\Diru$, see also  \cite{garcke2019phase, faghihi2020coupled,bartkowiak2005cahn} for the same choice of boundary behaviour. This allows us to apply the well-known Korn inequality directly in the proof of existence. In the case of pure Neumann boundary conditions, one has to consider a different solution space for $\vec{u}$, see \cite{miranville2001long, miranville2003generalized,garcke2005cahn} for more details. 
 \end{remark}

\section{Preliminaries}\label{sec:preliminaries}
\revision{
In this section, we introduce the spaces along with the embedding results, and the useful inequalities and auxiliary results which are used in the Section \ref{sec:mathematical_analysis}.
\subsection{Notation and embedding results}}
Let $\W^{k}_{p}(\Omega;\R^d)$ denote the Sobolev space of order $k$ with weak derivatives in the space  $\L_p(\Omega;\R^d)$ of $p$-integrable functions having value in $\R^d$. Shortly, we write $\W^k_p(\Omega;\R)=\W^k_p(\Omega)$, $\H^1(\Omega)=\W^{1}_{2}(\Omega)$ and $\H_{0,\Sigma}^1(\Omega;\R^d)$ denotes the space of $\H^1(\Omega;\R^d)$ functions with vanishing trace on $\Sigma \subset \p\Omega$, see \cite{brezis2010functional} for more details. 
For notational simplicity, we denote $\|\cdot\|_{\L_2(\Omega;\R^d)}$ and $\|\cdot\|_{\L_2(\Sigma;\R^d)}$ by $\|\cdot\|$ and $\|\cdot\|_{\Sigma}$ respectively, $(\cdot,\cdot)_{\L_2(\Omega;\R^d)}$ and $(\cdot,\cdot)_{\L_2(\Sigma;\R^d)}$ by $(\cdot,\cdot)$ and $(\cdot,\cdot)_{\Sigma}$ respectively, and the brackets $\left<\cdot,\cdot\right>$ denote the duality pairing on $\H^{-1}(\Omega)\times\H^1(\Omega)$. The symbol $\C^k(\cdot)$ denotes the space of $k$-times continuously differentiable functions and $\C_b(\cdot)$ denotes the space of bounded continuous functions. Let $\HS$ be a real separable Hilbert space with norm $\|\cdot\|_{\HS}$ and $\HSV$ \revision{be a Hilbert space} such that $\HSV\doublehookrightarrow\HS\hookrightarrow\HSVD$ is a Gelfand triple. We define the Bochner space
$$\L_p(0,T;\HS):=\left\{ \!\varphi: (0,T) \to X : \varphi \text{ Bochner measurable and } \|\varphi\|_{\L_p(0,T;\HS)}^p\!:=\int_0^T\!\!\! \|\varphi(t)\|_{\HS}^p \dd t < \infty\right\},$$
where $p \in [1,\infty)$. For $p=\infty$ we modify it in the usual sense with the Bochner norm $$\|\varphi\|_{\L_\infty(0,T;\HS)}:=\esssup_{t\in (0,T)} \|\varphi(t)\|_{\HS}.$$ 
We introduce the Sobolev--Bochner space
$$\W^{1}_{p,q}(0,T;\HSV,\HSVD):=\{ \varphi \in \L_p(0,T;\HSV) : \partial_t \varphi \in \L_q(0,T;\HSVD) \},$$
and its fractional counter-part 
\begin{equation*} 
    \W^{\alpha}_{p,q}(0,T;\varphi_0,\HSV,\HSVD):=\{\varphi\in\L_p(0,T;\HSV): g_{1-\alpha}*(\varphi-\varphi_0)\in {}_{0}\W^{1}_{p,q}(0,T;\HSV,\HSVD)\},
\end{equation*}
where $\varphi_0\in \HS$ and ${}_{0}\W^{1}_{p,q}$ denotes functions in $\W^{1}_{p,q}$ with vanishing trace at $t=0$. This definition of the fractional Sobolev--Bochner space indeed corresponds to the space of integrable fractional time-derivatives.

In the existence proof we typically apply compactness results. A special case of the Aubin--Lions compactness theorem  \cite[Corollary 4]{simon1986compact},  states the following compact embedding
\begin{equation} \label{Lem_Aubin}
p\in[1,\infty), \quad \W^{1}_{p,1}(0,T;\HSV,\HSVD)\doublehookrightarrow \L_p(0,T;\HS).
\end{equation}
In the fractional setting, we have the following analogous result
\begin{equation} \label{Lem_FractionalAubin}
\begin{aligned}
p\in[1,\infty), r\in\left(\frac{p}{1+p\alpha},\infty\right)\cap[1,\infty),\quad&
 \W^{\alpha}_{p,r}(0,T;\varphi_0,\HSV,\HSVD)\doublehookrightarrow \L_p(0,T;\HS),
\end{aligned}
\end{equation}
where $\varphi_0 \in \HS$ is given. The proof follows the lines of \cite[Theorem 4.2]{li2018some} using the estimates from \cite[Proposition 3.4]{li2018some}.

We also employ the following continuous embedding into the time-continuous function space to establish additional regularity of the solutions of the partial differential equations,
\begin{equation} \label{Lem_InterpolationEmbedding}
\W^1_{2,2}(0,T;\HSV,\HSVD) \hookrightarrow \C([0,T];[\HSV,\HSVD]_{1/2}),
\end{equation}
where $[\HSV,\HSVD]_{1/2}$ denotes the interpolation space of order $1/2$ of $\HSV$ and $\HSVD$, see \cite[Theorem 3.1, Chapter 1]{lions2012non}, e.g. $[\H_0^1(\Omega),\H^{-1}(\Omega)]_{1/2}=\L_2(\Omega)$. In the fractional setting, we have a continuous embedding analogous to the one above. Indeed,
\begin{equation}\label{Lem_ContinousEmbedding_fractional}
    \varphi_0\in \HS,\, \varphi\in \W^{\alpha}_{2,2}(0,T;\varphi_0,\HSV,\HSVD)\implies g_{1-\alpha}*(\varphi-\varphi_0)\in \C([0,T];\HS),
\end{equation}
after possibly being redefined on a set of measure zero, see \cite[Theorem 2.1]{zacher2009weak}. 

Throughout the whole paper, we denote by $C$ a generic positive constant which is independent of the unknowns $\phi, \mu, \vec{u}, \psi$ and $\chi$.

\subsection{Useful inequalities and auxiliary results}
We recall the Poincar\'e--Wirtinger, Poincar\'e, Korn and Sobolev inequalities, see \cite{brezis2010functional,ciarlet2013linear},
\begin{align}\label{inequality:Poincare_Korn_Sobolev}
    \begin{aligned}
	\|\varphi-\overline{\varphi}\| & \leq C \|\nabla \varphi\| && \text{for all } \varphi \in \H^1(\Omega),   \\
	\|\varphi\|              & \leq C  \|\nabla \varphi\| && \text{for all } \varphi \in \H^1_{0,\Sigma}(\Omega), \\
	\|\varphi\|_{\H^1(\Omega;\R^d)}              & \leq C  \|\symgrad(\varphi)\| && \text{for all } \varphi \in \H^1_{0,\Sigma}(\Omega;\mathbb{R}^d), 
	\\[-.18cm]
	 \|\varphi\|_{\W^{m}_{q}(\Omega;\R^d)} & \leq C \|\varphi\|_{\W^{k}_{p}(\Omega;\R^d)} && \text{for all } \varphi \in \W^{k}_{p}(\Omega;\R^d), \quad k-\frac{d}{p} \geq m-\frac{d}{q}, \quad k \geq m, 
\end{aligned}
\end{align}
where $\overline{\varphi}=\frac{1}{|\Omega|}\int_\Omega \varphi(\vec{x}) \,\textup{d} x$ denotes the mean of $\varphi$. Also, we often make use of the $\epsilon$-Young and the Young convolution inequalities, given by
\begin{align}\label{inequality:Youngs}
    \begin{aligned}
    ab &\leq \epsilon a^p + \frac{b^q}{q (\epsilon p)^{q/p}} && \text{for all } a,b \geq 0, \quad  \frac{1}{p} + \frac{1}{q} = 1, \quad\epsilon>0,\\
    \| \varphi_1 * \varphi_2 \|_{\L_r(\Omega)} &\leq \| \varphi_1  \|_{\L_p(\Omega)}  \| \varphi_2 \|_{\L_q(\Omega)} && \text{for all } \varphi_1 \in \L_p(\Omega), \varphi_2 \in \L_q(\Omega), \quad \frac{1}{p} + \frac{1}{q} = \frac{1}{r}+1,
    \end{aligned}
\end{align}
and H\"older's inequality, given by
\begin{align}\label{inequality:Holder}
    \begin{aligned}
    \| \varphi_1 \varphi_2 \|_{\L_1(\Omega)} &\leq \| \varphi_1  \|_{\L_p(\Omega)}  \| \varphi_2 \|_{\L_q(\Omega)} && \text{for all } \varphi_1 \in \L_p(\Omega), \varphi_2 \in \L_q(\Omega), \quad \frac{1}{p} + \frac{1}{q} = 1,
    \end{aligned}
\end{align}
see \cite[Appendix B]{evans2010partial}. 

The following inequality, which is analogous to the chain rule, is required to obtain a priori estimates to prove the existence of weak solutions of a time-fractional partial differential equation. 
\begin{lemma} Suppose $\varphi\in\L_2(0,T;\L_2(\Omega,\R^d))$, and there exists $\varphi_0\in \L_2(\Omega,\R^d)$ such that $g_{1-\alpha}*(\varphi-\varphi_0)\in{}_{0}\W^1_{2,2}(0,T;\L_2(\Omega,\R^d),\L_2(\Omega,\R^d))$. Let $\Func\in\C^1(\R^d)$ be a convex function such that $g_{1-\alpha}*\int_\Omega \Func(\varphi)\dd x\in {}_{0}\W^1_1(0,T)$.
Then for almost all $t\in(0,T)$, we have 
\begin{equation}\label{basic_inequality_fractional_chain_rule}
\begin{aligned}
\left( \Func'(\varphi(t)),\partial_t(g_{1-\alpha}*\varphi)(t)\right) &\geq \partial_t\left(g_{1-\alpha}*\int_\Omega \Func(\varphi)\dd x\right) (t)\\
&+ \left(-\int_\Omega \Func(\varphi(t))\dd x
+\left(\Func'(\varphi(t)),\varphi(t)\right)\right)g_{1-\alpha}(t).
\end{aligned}
\end{equation}
        \label{Lem_basic_inequality_fractional_chain_rule}
\end{lemma}
\begin{proof}
Let $k\in\W^1_1(0,T)$. Then from a straightforward computation, we have the following identity for almost all $t\in[0,T]$
\begin{equation}
\begin{aligned}
\int_\Omega \Func'(\varphi(t)):\partial_t(k*\varphi)(t)\dd x = \partial_t\left(k*\int_\Omega \Func(\varphi)\dd x\right)(t) +  k(t) \left( \int_\Omega -\Func(\varphi(t))+\Func'(\varphi(t)):\varphi(t)\dd x\right)\\
- \int_0^t \frac{\dd}{\dd \rm{s}} k(s) \left(\int_\Omega  \Func(\varphi(t-s))-\Func(\varphi(t)) - \Func'(\varphi(t)):(\varphi(t-s)-\varphi(t)) \dd x\right)  \mathrm{d}s.
\end{aligned}
\end{equation}
We remark that the identity for functions with values in $\R$ can be seen in \cite[Lemma 6.1]{kemppainen2016decay} and the integrated form for functions in $\R^d$ can be seen in \cite[Lemma 18.4.4]{gripenberg1990volterra}. We note that if $k$ is non-negative and non-increasing then the last term is positive, since $\Func$ is a convex functional. The inequality \eqref{basic_inequality_fractional_chain_rule} follows as in \cite[Theorem 2.1]{vergara2008lyapunov} by approximating $g_{1-\alpha}$ with a more regular kernel $k_n\in\W^{1,1}(0,T)$, using the above identity and taking the limit. 
\end{proof}

A particular form of Lemma \ref{Lem_basic_inequality_fractional_chain_rule} with $\Func(\varphi)=\frac{1}{2}\varphi^2$ is proved in \cite[Theorem 2.1]{vergara2008lyapunov} with functionals having value in any Hilbert space, and we have for almost all $t\in[0,T]$
\begin{equation}\label{basic_inequality_fractional}
\frac{1}{2}\dt (g_{1-\alpha}*\|\varphi\|^2)(t) + \frac{1}{2} g_{1-\alpha}(t)\|\varphi(t)\|^2 \leq \left(\varphi(t),\partial_t (g_{1-\alpha}*\varphi)(t)\right).
\end{equation}


\begin{remark}\label{rem:kernel_property1}
The first term in \eqref{basic_inequality_fractional} is well-posed for $\varphi\in\W^{\alpha}_{2,2}(0,T;\varphi_0,\L_2(\Omega,\R^d),\L_2(\Omega,\R^d))$ because of the following implication, which indeed holds true for a wide class of kernels and is proved in \cite[Proposition 2.1]{vergara2008lyapunov},
\begin{equation}\label{eqn:kernel_property}
\varphi\in\L_2(0,T;\H),\, g_{1-\alpha}*\varphi\in{}_{0}\H^1(0,T;\H)\implies g_{1-\alpha}*\|\varphi\|_{\H}^2\in{}_{0}\W^{1}_{1}(0,T).
\end{equation}
\end{remark}

The following are the Gronwall--Bellman and generalised Gronwall--Bellman with singularity, used for providing explicit bounds on solutions.
\begin{lemma}[Gronwall--Bellman, cf. {\cite[Appendix B]{evans2010partial}}] Assume $C_1,C_2\geq 0$ are constants. If $\varphi(t)$ is a non-negative, integrable function on $[0,T]$, satisfying 
		$$\varphi(t) \leq C_1 + C_2 \int_0^t \varphi(s) \, \dd s,$$ 
for almost all $t\in [0,T]$, then it holds
$\varphi(t) \leq C_1 e^{C_2 T}$ 
for almost all $t \in [0,T]$.
		\label{Lem_Gronwall}
\end{lemma}

\begin{lemma}[Generalized Gronwall--Bellman, cf. {\cite[Corollary 1]{ye2007generalized}}] 
Assume that $a(t)$ is a non-negative integrable function on the interval $[0,T]$, and $b>0$ is a constant. If $\varphi(t)$ is a non-negative, integrable function satisfying 
\begin{equation*}
\varphi(t)\leq a(t)+\frac{b}{\Gamma(\alpha)}\int_0^t (t-s)^{\alpha-1}\varphi(s)\, \textup{d}s,
\end{equation*}
for almost all $t\in[0,T]$, then for almost all $t\in[0,T]$ the following holds true,
\begin{equation*}
    \varphi(t)\leq a(t) + \int_0^t b\Gamma(\alpha) (t-s)^{\alpha-1} E_{\alpha,\alpha}(b\Gamma(\alpha)(t-s)^{\alpha}) a(s)\dd s,
\end{equation*}
where $E_{\alpha,\alpha}(x)=\sum_{k=0}^{\infty}\frac{x^k}{\Gamma(\alpha k +\alpha)}$ is the two-paramater Mittag--Leffler function.
	\label{Lem_Gronwall_fractional}
\end{lemma}


\begin{lemma}[Fractional integration by parts, cf. {\cite[Proposition 3.1]{djilali2018galerkin}}]
Let $\varphi_1\in\L_2(0,T;\H)$ and $\varphi_2\in\H^1(0,T;\H)$. Then
\begin{equation*}
\begin{aligned}
\int_0^T\left(\p_t(g_{1-\alpha}*\varphi_1)(t),\varphi_2\right)\dd t  = -\int_0^T\left(\varphi_1,(g_{1-\alpha}*'\p_t\varphi_2)(t)\right)\dd t+\left((g_{1-\alpha}*\varphi_1)(t),\varphi_2\right)|^{t=T}_{t=0},
\end{aligned}
\end{equation*}
where the convolution $*'$ is defined by $(g_\alpha*'\varphi)(t)=\int_t^T g_\alpha(t-s)\varphi(s)\dd s$.
		\label{Lem_frac_integration_by_parts}
\end{lemma}

\section{Mathematical analysis} \label{sec:mathematical_analysis}
In this section, we provide the existence and uniqueness \revision{of the weak solution} to the model \eqref{System} using the Galerkin approximation approach.
\begin{definition} \label{Def:WeakSolution}
	We say that $(\phi, \mu, \vec{u}, \psi, \chi)$ satisfying
	\begin{equation*}
	    \begin{aligned}
	    \phi &\in \W^{\alpha}_{2,2}(0,T;\phi_0,\L_2(\Omega),\L_2(\Omega)),\quad &
		\mu &\in \L_2(0,T;\H^1(\Omega)),\\
		\vec{u}&\in \L_2(0,T;\revision{\H^1_{0,\Diru}(\Omega;\R^d)}),\quad&
		\psi,\chi &\in \W^{1}_{2,2}(0,T;\revision{\H^1_{0,\Dirpsi}(\Omega)},\H^{-1}(\Omega)),\\
	    \end{aligned}
	\end{equation*}
	is a weak solution to the system \eqref{System} with data \eqref{eq:data}, if the initial conditions $g_{1-\alpha}*(\phi-\phi_0)(0)=0,\ \psi(0)=\psi_0,\ \chi(0)=\chi_0$ holds in the weak sense and the solution satisfies the variational form
	\begin{equation}\label{eqn:variational_form}
	\begin{aligned}
		\left(\p_t^{\alpha}(\phi-\phi_0),\xi_1\right) &+ (M_\phi\nabla\mu,\nabla \xi_1)  =  N_\phi\big(f(\phi,\psi),\xi_1\big) - P_\phi\big(g(\phi,\chi),\xi_1\big),\\
		\big(\mu,\xi_2\big)&=c\big(\phi,\xi_2\big) + \lambda\big(\nabla\cdot \vec{u},\xi_2\big),\\
		2G\left(\symgrad(\vec{u}),\symgrad(\vec{\xi}_3)\right)&+\frac{2G\nu}{1-2\nu}\left(\nabla\cdot\vec{u},\nabla\cdot\vec{\xi}_3\right)=-\lambda\left(\phi,\nabla\cdot\vec{\xi}_3\right), \\
		\left<\partial_t \psi,\xi_4\right> &+ \big(M_\psi \nabla \psi,\nabla \xi_4\big)= \big(S_\psi,\xi_4\big) - N_\psi \big(f(\phi,\psi),\xi_4\big)+(\psi_b,\xi_4)_{\p\Omega\backslash\Dirpsi}, \\
		\left<\partial_t \chi,\xi_4\right> &+ \big(M_\chi \nabla \chi,\nabla \xi_4\big)=  \big(S_\chi,\xi_4\big)- N_\chi(\chi,\xi_4) - P_\chi \big(g(\phi,\chi),\xi_4\big) + (\chi_b,\xi_4)_{\p\Omega\backslash\Dirpsi}, \\
	\end{aligned}
	\end{equation}
	for all $\xi_1\in \H^1(\Omega), \xi_2\in\L_2(\Omega), \vec{\xi}_3\in \H^1_{0,\Diru}(\Omega;\R^d)$ and $\xi_4\in \H^1_{0,\Dirpsi}(\Omega)$.
	\end{definition}
\begin{theorem}[Well-posedness of global weak solutions]
 Let the following assumptions hold: 
\begin{enumerate}[wide,label=\textup{(A\arabic*)},labelwidth=!,labelindent=0pt] \itemsep0em
 \item \label{Ass:initial_boundary} $\phi_0,\psi_0, \chi_0 \in \L_2(\Omega)$, $\psi_b,\chi_b \in\L_2(0,T;\L_2(\p\Omega\backslash\Dirpsi))$,
 \item \label{Ass:f_g} $f,g \in \C_b(\R^2)$ such that $0\leq f\leq C_f$ and $0\leq g\leq C_g$, for positive constants $C_f,C_g$,
 \item \label{Ass:mobility} $M_\phi,M_\psi,M_\chi \in \C_b(\Omega)$ such that $M_0\leq M_\phi(\vec{x}),M_\psi(\vec{x}),M_\chi(\vec{x}) \leq M_\infty$ for positive constants $M_0$, $M_\infty$,
 \item \label{Ass:source} $S_\psi,S_\chi \in \L_2(0,T;\L_2(\Omega))$,
  \item \label{Ass:paramters} $c>\frac{\lambda^2(1-2\nu)}{2G\nu}$,
\end{enumerate}
then there exists a weak solution $(\phi,\mu,\vec{u},\psi,\chi)$ in the sense of Definition \ref{Def:WeakSolution}. Additionally, the solution satisfies the estimate 
\begin{equation}\label{est:final}
\begin{aligned}\begin{multlined}[t]
\|\phi\|_{\L_2(0,T;\L_2(\Omega))}^2+\|\mu\|^2_{\L_2(0,T;\H^1(\Omega))}+\|\vec{u}\|_{\L_2(0,T;\H^1(\Omega;\R^d))}^2+\| \psi\|_{\L_2(0,T;\H^1(\Omega))}^2+\| \chi\|_{\L_2(0,T;\H^1(\Omega))}^2\end{multlined}\\
\leq \begin{multlined}[t]
 C  \big(\textup{IC}+C_f+C_g+\|S_{\psi}\|_{\L_2(0,T;\L_2(\Omega))}^2+\|S_{\chi}\|_{\L_2(0,T;\L_2(\Omega))}^2
 +\|\psi_b\|_{\L_2(0,T;\L_2(\p\Omega\backslash\Dirpsi))}^2\\+\|\chi_b\|_{\L_2(0,T;\L_2(\p\Omega\backslash\Dirpsi))}^2\big),
\end{multlined}
\end{aligned}
\end{equation}
where $\textup{IC}=\|\phi_0\|^2+\|\psi_0\|^2+\|\chi_0\|^2$. Furthermore, the solution is unique if 
the nonlinear functions $f,g$ are Lipschitz continuous with Lipschitz constants $L_f,L_g>0$, respectively.
\end{theorem}

\noindent \textit{Proof.} \quad
	In order to prove the existence of weak solution, we first use the Faedo-Galerkin method and semi-discretise the original problem in space in Section \ref{subsec:faedo_galerkin}.
	The discretised model can be formulated as a system of nonlinear mixed-order fractional differential equations in a finite dimensional space whose existence of a solution is then obtained by fixed point theorem in Appendix \ref{appendix:existence_finitedim}.
	We obtain the required energy estimates in Section \ref{subsec:energy_estimates}. In Section \ref{subsec:existence_weak_solution}, we deduce from the Banach--Alaoglu theorem and compactness theorems, \revision{the existence of limit functions that yield a weak solution} to the nonlinear system \eqref{System} in the sense of Definition \ref{Def:WeakSolution} and show the weak solution satisfies the estimate \eqref{est:final}.
	Finally, we show in Section \ref{subsec:uniqueness} that the Lipschitz continuity assumption on the nonlinear functions $f$ and $g$ gives uniqueness of the solution.

\subsection{Faedo--Galerkin approximation}\label{subsec:faedo_galerkin}
We first choose discrete spaces $\PY^m$, $\PZ^m$ and $\PW^m$ such that their unions over $m\in\mathbb{N}$ are dense in $\H^1(\Omega)$, $\H^1_{0,\Dirpsi}(\Omega)$ and $\H^1_{0,\Diru}(\Omega;\R^d)$, respectively. We construct approximate solutions in these discrete spaces. We see that the semi-discretised model of Problem \eqref{eqn:variational_form} can be formulated as a system of multi-order fractional ordinary differential equations.\\ 

\noindent \textbf{Discrete spaces.} 
We introduce the discrete spaces
\begin{equation*}
    \PY^m=\mathrm{span}\{y_1,\ldots,y_m\},\quad  \PZ^m= \mathrm{span}\{z_1,\ldots,z_m\},\quad
    \PW^m= \mathrm{span}\{\vec{w}^1,\ldots,\vec{w}^m\},
\end{equation*}
where $y_k,z_k:\Omega\rightarrow \R$, $\vec{w}_k:\Omega\rightarrow \R^d$  for $k=1,\ldots,m$ are eigenfunctions to the eigenvalues $\lambda_k^y,\lambda_k^z,\lambda_k^{\vec{w}}$ of the following respective problems
 $$ \hspace{-2.2cm}  \begin{multlined}[t]
    \begin{aligned}
     -\Delta y_k&=\lambda_k^y y_k &\text{in $\Omega$}, \\ 	
     \nabla y_k\cdot \vec{n}&=0 &\mbox{on $\p\Omega$},
    \end{aligned}
    \end{multlined} 
    \hspace{-2.2cm}
    \begin{multlined}[t]
    \begin{aligned}
     -\Delta z_k&=\lambda_k^z z_k &&\mbox{in $\Omega$}, \\ 	
     z_k&=0 &&\mbox{on $\p\Dirpsi$},
     \\ 	
     \nabla z_k\cdot \vec{n}&=0 &&\mbox{on $\p\Omega \backslash \p\Dirpsi$},
     \end{aligned}
    \end{multlined} 
    \hspace{-3.2cm}
    \begin{multlined}[t]
    \begin{aligned}
     -\Delta \vec{w}_k&=\lambda_k^{\vec{w}} \vec{w}_k &&\mbox{in $\Omega$}, \\	\vec{w}_k&=0 &&\mbox{on $\p\Diru$},
     \\ 	
     \nabla \vec{w}_k\cdot \vec{n}&=0 &&\mbox{on $\p \Omega \backslash \p\Diru$}.
    \end{aligned}
    \end{multlined} $$
Since the Laplace operator is a compact, self-adjoint, injective operator, we conclude by the spectral theorem \cite{boyer2013navier, robinson2001infinite, brezis2010functional}, that 
\begin{center}
     $\{y_k\}_{k=1}^{\infty}$, $\{z_k\}_{k=1}^{\infty}$ are orthonormal bases in $\L_2(\Omega)$ and orthogonal bases in $\H^1(\Omega)$,\\
    $\{\vec{w}_k\}_{k=1}^{\infty}$ is an orthonormal basis of $\L_{2}(\Omega;\R^d)$ and orthogonal basis in $\H^1(\Omega;\R^d)$. 
\end{center}
Exploiting the orthonormality of the eigenfunctions, we deduce that $\PY^m$, $\PZ^m$ are dense in $\L_2(\Omega)$, and $\PW^m$ is dense in $\L_2(\Omega;\R^d)$. We introduce the orthogonal projection, $\mathrm{\Pi}_{\PY^m}:\L_2(\Omega) \to \PY^m$, which can be written as 
\begin{equation*}
    \Pi_{\PY^m} \varphi =\sum_{k=0}^m(\varphi,y_k)y_k,
\end{equation*}
and by the properties of orthogonal projections, we have $\|\Pi_{\PY^m} \varphi\| \leq \|\varphi\|$.
Analogously, we can define $\Pi_{\PZ^m}$ and $\Pi_{\PW^m}$. \\

\noindent\textbf{Faedo--Galerkin system:} Fix $m>0$ and consider the Faedo--Galerkin approximations $\phi^m,\mu^m:[0,T]\rightarrow \PY^m$, $\vec{u}^m:[0,T]\rightarrow \PW^m$ and $\psi^m,\chi^m:[0,T]\rightarrow \PZ^m$ with the representations
\begin{equation}\label{approx} 
\begin{aligned}
\phi^m(t):=\sum_{k=1}^m \Cphi^m_k(t)y_k,\ \  \mu^m(t):= \sum_{k=1}^m \Cmu^m_k(t)y_k,\ \ \vec{u}^m(t):=\sum_{k=1}^m \Cu^m_k(t)\vec{w}_k, \\
\psi^m(t):=\sum_{k=1}^m \Cpsi^m_k(t)z_k,\ \ \chi^m(t):=\sum_{k=1}^m \Cchi^m_k(t)z_k,\ \   
\end{aligned}
\end{equation}
where $\Cphi^m_k, \Cmu_k^m, \Cu_k^m, \Cpsi_k^m, \Cchi_k^m:(0,T)\rightarrow\R$ are coefficient functions for $k=1,\ldots,m$. To simplify notations, we set
\begin{equation*}
\begin{aligned}
f^m=f(\phi^m,\psi^m),\ g^m=g(\phi^m,\chi^m),\ \psi_b^m=\Pi_{\PZ^m}\psi_b,\ \chi_b^m=\Pi_{\PZ^m}\chi_b,\\
\phi_0^m=\Pi_{\PY^m}\phi_0,\ \psi^m_0=\Pi_{\PZ^m} \psi_0,\ \chi^m_0=\Pi_{\PZ^m} \chi_0. 
\end{aligned}
\end{equation*}
The Faedo--Galerkin system of the model reads
\begin{subequations} \label{Faedo--Galerkin}
 \begin{align}
\left(\p_t^{\alpha}(\phi^m-\phi_0^m),y_k\right) &+ (M_\phi\nabla\mu^m,\nabla y_k)  =  N_\phi\big(f^m,y_k\big)- P_\phi\big(g^m,y_k\big), \label{Faedo--Galerkin:phi}\\
(\mu^m,y_k) &=c(\phi^m,y_k) + \lambda(\nabla\cdot\vec{u}^m,y_k),\label{Faedo--Galerkin:mu}\\
2G\big(\symgrad(\vec{u}^m),\symgrad(\vec{w}_k)\big)&+ \frac{2G\nu}{1-2\nu}\big(\nabla\cdot \vec{u}^m ,\nabla \cdot \vec{w}_k\big) =- \lambda\big(\phi^m,\nabla\cdot\vec{w}_k\big), \label{Faedo--Galerkin:u}\\
\left(\partial_t \psi^m,z_k\right) &+ \big(M_\psi \nabla \psi^m,\nabla z_k\big)=  \big(S_\psi,z_k\big) - N_\psi \big(f^m,z_k\big) + (\psi_b^m,z_k)_{\p\Omega\backslash\Dirpsi}, \label{Faedo--Galerkin:psi} \\
\left(\partial_t \chi^m,z_k\right) &+ \big(M_\chi \nabla \chi^m,\nabla z_k\big)=  \big(S_\chi,z_k\big) - N_\chi(\chi^m,z_k) - P_\chi \big(g^m,z_k\big)+ (\chi_b^m,z_k)_{\p\Omega\backslash\Dirpsi},  \label{Faedo--Galerkin:chi}
\end{align}
\end{subequations}
for all $k=1,\ldots,m$, along with the initial conditions
\begin{align*}
	g_{1-\alpha}*(\phi^m-\phi_0^m)(0)&=0, \ \ \ \psi^m(0)= \psi_0^m, \ \ \ \chi^m(0)= \chi_0^m.
\end{align*}
After inserting the Galerkin ansatz functions \eqref{approx} into the system \eqref{Faedo--Galerkin} 
and introducing the following notations
\begin{equation*}
\begin{alignedat}{5}
(\vec{A}_\mu^m)_{kl} &:=(M_\phi\nabla y_l,\nabla y_k), \quad &&(\vec{A}_\psi^m)_{kl} &&:=(M_\psi\nabla z_l,\nabla z_k), \quad &&(\vec{A}_\chi^m)_{kl} &&:=(M_\chi\nabla z_l,\nabla z_k), \\ 
(\vec{A}_{\vec{u}}^m)_{kl} &:=(\symgrad(\vec{w}_l),\symgrad(\vec{w}_k)),\quad &&(\vec{B}^m)_{kl} &&:=(\nabla\cdot \vec{w}_l,\nabla\cdot \vec{w}_k), \quad &&(\vec{C}^m)_{kl} &&:=(y_l,\nabla \cdot \vec{w}_k), \\ 
\vec{\Cphi}^m(t) &:=(\Cphi^m_1,\ldots,\Cphi^m_m)^T, \quad  &&\vec{\Cmu}^m(t)&&:=(\Cmu^m_1,\ldots,\Cmu^m_m)^T, \quad &&\vec{\Cu}^m(t) &&:=(\Cu^m_1,\ldots,\Cu^m_m)^T,\\  
\vec{\Cpsi}^m(t) &:=(\Cpsi^m_1,\ldots,\Cpsi^m_m)^T, \quad  &&\vec{\Cchi}^m(t) &&:=(\Cchi^m_1,\ldots,\Cchi^m_m)^T, \quad && \vec{\Cphi}^m_0 &&:=((\phi_0,y_1),\ldots,(\phi_0,y_m))^T,
\end{alignedat}
\end{equation*}
\begin{equation*}
\begin{alignedat}{3}
\vec{S}_\psi^m(t)&:=((S_\psi,z_1),\ldots,(S_\psi,z_m))^T, 
\quad &&\vec{S}_\chi^m(t) &&:=((S_\chi,z_1),\ldots,(S_\chi,z_m))^T, \\
\vec{\psi}_b^m &:=((\psi_b^m,z_1),\ldots,(\psi_b^m,z_m))^T, 
\quad &&\vec{\chi}_b^m(t)&&:=((\chi_b^m,z_1),\ldots,(\chi_b^m,z_m))^T,\\
 \vec{f}_y^m(t)&:=((f^m,y_1),\ldots,(f^m,y_m))^T, \quad &&\vec{g}_y^m(t)&&:=((g^m,y_1),\ldots,(g^m,y_m))^T,\\
 \vec{f}_z^m(t)&:=((f^m,z_1),\ldots,(f^m,z_m))^T, \quad &&\vec{g}_z^m(t)&&:=((g^m,z_1),\ldots,(g^m,z_m))^T,
\end{alignedat}
\end{equation*}
we obtain a more compact form of the Faedo--Galerkin system
\begin{subequations}\label{ODE}
    \begin{align}
	\dt\left(g_{1-\alpha}*(\vec{\Cphi}^m(t)-\vec{\Cphi}^m_0)\right)(t)+ \vec{A}_\mu^m \vec{\Cmu}^m(t)&= N_\phi\vec{f}_y^m(t) - P_\phi \vec{g}_y^m(t),\label{ODE:phi}\\
	\vec{\Cmu}^m(t) &= c\vec{\Cphi}^m(t) +\lambda(\vec{C}^m)^t\vec{\Cu}^m(t),\label{ODE:mu}\\
	\big( 2G \vec{A}_{\vec{u}}^m + \tfrac{2G\nu}{1-2\nu} \vec{B}^m\big)\vec{\Cu}^m(t) &= -\lambda \vec{C}^m \vec{\Cphi}^m(t),\label{ODE:u}\\
	\dt \vec{\Cpsi}^m(t) + \vec{A}_\psi^m \vec{\Cpsi}^m(t) &=\vec{S}_\psi^m(t)-N_\psi\vec{f}_z^m(t)+\vec{\psi}_b^m, \label{ODE:psi}\\
	\dt \vec{\Cchi}^m(t) + \vec{A}_\chi^m \vec{\Cchi}^m(t) &= \vec{S}_\chi^m(t) -N_\chi\vec{\Cchi}^m(t) -P_\chi\vec{g}_z^m(t)+\vec{\chi}_b^m. \label{ODE:chi}
\end{align}
\end{subequations}
The matrix $\vec{F}^m := 2G \vec{A}_{\vec{u}}^m + \tfrac{2G\nu}{1-2\nu} \vec{B}^m$ is positive definite by Korn\rq s inequality \eqref{inequality:Poincare_Korn_Sobolev}
and hence, it is invertible. From \eqref{ODE:mu} and \eqref{ODE:u}, we can write $\vec{\Cmu}^m(t),\vec{\Cu}^m(t)$ in terms of $\vec{\Cphi}^m(t)$.
\begin{subequations}
    \begin{align}
       \vec{\Cmu}^m(t) &= (c\mathbb{I}-\lambda^2(\vec{C}^m)^t(\vec{F}^m)^{-1} \vec{C}^m ) \vec{\Cphi}^m(t),\label{ODE_Soln:mu} \\
	\vec{\Cu}^m(t) &= -\lambda (\vec{F}^m)^{-1} \vec{C}^m \vec{\Cphi}^m(t).\label{ODE_Soln:u}
    \end{align}
\end{subequations}
Then we obtain a system of nonlinear multi-order fractional differential equations in the $3m$ unknowns $\{\Cphi_k,\Cpsi_k,\Cchi_k\}_{1\leq k\leq m}$ 
\begin{equation*}
\begin{aligned}
	\dt\left(g_{1-\alpha}*(\vec{\Cphi}^m(t)-\vec{\Cphi}^m_0)\right)(t)+\vec{A}_\mu^m (c\mathbb{I}-\lambda^2(\vec{C}^m)^t(\vec{F}^m)^{-1} \vec{C}^m ) \vec{\Cphi}^m(t)&= N_\phi\vec{f}_y^m(t) - P_\phi \vec{g}_y^m(t),\\
	\dt\vec{\Cpsi}^m(t) + \vec{A}_\psi^m \vec{\Cpsi}^m(t) &= \vec{S}^m_\psi(t)-N_\psi\vec{f}^m(t)+\vec{\psi}_b^m,\\
	\dt \vec{\Cchi}^m(t) + \vec{A}_\chi^m \vec{\Cchi}^m(t) &= \vec{S}_\chi^m(t) -N_\chi\vec{\Cchi}^m(t) -P_\chi\vec{g}^m(t)+\vec{\chi}^m_b,
	\end{aligned}
\end{equation*}
along with the initial conditions, for $k=1,\ldots,m$, 
\begin{align*}
	\left(g_{1-\alpha}*((\vec{\Cphi}^m)_k-(\phi_0,y_k))\right)(0) =0, \ \ \ (\vec{\Cpsi}^m)_k(0) =(\psi_0,z_k), \ \ \ (\vec{\Cchi}^m)_k(0) =(\chi_0,z_k).
\end{align*}
The theory of ordinary fractional differential equations in Appendix \ref{appendix:existence_finitedim} ensures the existence of solution to the nonlinear multi-order fractional differential system, and we obtain 
\begin{center}
	$\vec{\Cphi}^m,\vec{\Cmu}^m,\vec{\Cu}^m\in \W^{\alpha}_{2,2}(0,T;\vec{\Cphi}^m_0,\R^m, \R^m),\quad \vec{\Cpsi}^m,\vec{\Cchi}^m\in \W^1_{2,2}(0,T;\R^m,\R^m)$.
\end{center} 
 We further see from \eqref{ODE_Soln:mu} and \eqref{ODE_Soln:u} that
\begin{equation*}
    \begin{aligned}
     \Cmu^m_k(t)- \sum_{l=1}^m \left(c\mathbb{I}-\lambda^2(\vec{C}^m)^t(\vec{F}^m)^{-1} \vec{C}^m \right)_{kl} (\phi_0,y_l) &= \sum_{l=1}^m \left(c\mathbb{I}-\lambda^2(\vec{C}^m)^t(\vec{F}^m)^{-1} \vec{C}^m \right)_{kl} \left(\Cphi^m_l(t)-(\phi_0,y_l)\right),\\
     \Cu^m_k(t)+\lambda \sum_{l=1}^{m} \left((\vec{F}^m)^{-1} \vec{C}^m\right)_{kl} (\phi_0,y_l) &= -\lambda \sum_{l=1}^{m} \left((\vec{F}^m)^{-1}\vec{C}^m \right)_{kl}\left( \Cphi^m_l(t)-(\phi_0,y_l)\right).
    \end{aligned}
\end{equation*}
Taking convolution with $g_{1-\alpha}$ on both sides of the above two equations, multiplying by $y_k$ and $\vec{w}_k$ respectively, and taking summation over $k=1$ to $m$, we have 
\begin{equation*}
  \left(g_{1-\alpha}*\left(\mu^m-\mu_0^m\right)\right)(0)=0,\quad \left(g_{1-\alpha}*(\vec{u}^m-\vec{u}^m_0)\right)(0)=0,
\end{equation*}
where $\mu_0^m= \sum_{k,l=1}^m \left(c\mathbb{I}-\lambda^2(\vec{C}^m)^t(\vec{F}^m)^{-1} \vec{C}^m \right)_{kl} (\phi_0,y_l)y_k$ and $\vec{u}^m_0=-\lambda\sum_{l,k=1}^{m} \left((\vec{F}^m)^{-1} \vec{C}^m\right)_{kl}(\phi_0,y_l)\vec{w}_k$ and they satisfy
\begin{subequations}
 \begin{align}
    (\mu_0^m,y_k)&=c(\phi_0^m,y_k)+\lambda (\nabla\cdot\vec{u}_0^m,y_k),\label{Faedo--Galerkin:mu_0}\\
    2G\big(\symgrad(\vec{u}_0^m),\symgrad(\vec{w}_k)\big)&+ \frac{2G\nu}{1-2\nu}\big(\nabla\cdot \vec{u}_0^m ,\nabla \cdot \vec{w}_k\big) =- \lambda\big(\phi_0^m,\nabla\cdot\vec{w}_k\big).\label{Faedo--Galerkin:u_0}
    \end{align}
\end{subequations} 
Therefore, we conclude
\begin{align*}
	\phi^m\in \W^{\alpha}_{2,2}(0,T;\phi_0^m,\PY^m, \PY^m),\quad \mu^m \in \W^{\alpha}_{2,2}(0,T;\mu_0^m,\PY^m, \PY^m),\\ \vec{u}^m \in \W^{\alpha}_{2,2}(0,T;\vec{u}^m_0,\PW^m, \PW^m), \quad \psi^m,\chi^m\in \W^1_{2,2}(0,T,\PY^m,\PY^m).
\end{align*}
We obtain from \eqref{Faedo--Galerkin:mu} (and \eqref{Faedo--Galerkin:u}) and \eqref{Faedo--Galerkin:mu_0} (and \eqref{Faedo--Galerkin:u_0}) that $\mu^m$ (and $\vec{u}^m$) satisfy the following equations
\begin{subequations}
    \begin{align}
 \left(\FDmuM,y_k\right) &= c\left(\FDphiM,y_k\right) + \lambda \left(\FDdivuM,y_k\right), \label{Faedo--Galerkin:mu_derivative}   \\
 2G\left(\FDsymuM,\symgrad(\vec{w}_k)\right) &+ \frac{2G\nu}{1-2\nu}\left(\FDdivuM,\nabla \cdot \vec{w}_k\right)
 =- \lambda\left(\FDphiM,\nabla\cdot\vec{w}_k\right).\label{Faedo--Galerkin:u_derivative}
    \end{align}
\end{subequations}
Further we also get from the way discrete spaces are defined the following equation
\begin{equation*}
 \begin{aligned}
\frac{-1}{\lambda_k^y}(\mu_0^m,\Delta y_k)=(\mu_0^m,y_k)&=c(\phi_0^m, y_k)+\lambda (\nabla\cdot\vec{u}_0^m, y_k),
\end{aligned}   
\end{equation*}
and taking integration by parts we get 
\begin{equation}
 \begin{aligned}
(\nabla\mu_0^m,\nabla y_k)= \lambda_k^y c(\phi_0^m, y_k)+\lambda_k^y \lambda (\nabla\cdot\vec{u}_0^m, y_k).\label{Faedo--Galerkin:nabla_mu_0}
\end{aligned}   
\end{equation}

\subsection{Energy estimates}\label{subsec:energy_estimates}

\noindent \textbf{Estimates for $\mu^m$}.
Multiplying \eqref{Faedo--Galerkin:mu} with $\Cmu^m_k(t)$ and summing from $k=1$ to $m$, we have, using H\"{o}lder's inequality \eqref{inequality:Holder},
\begin{equation*}
\|\mu^m\|^2 =c(\phi^m,\mu^m) + \lambda(\nabla\cdot\vec{u}^m,\mu^m) \leq  \left(c\|\phi^m\|+\lambda\|\nabla\cdot\vec{u}^m\| \right)\|\mu^m\|,
\end{equation*}
and thus
\begin{equation}\label{est:mu}
\|\mu^m\| \leq  c\|\phi^m\|+\lambda\|\nabla\cdot\vec{u}^m\|\leq c\|\phi^m\|+\lambda\|\vec{u}^m\|_{\H^1(\Omega;\R^d)}. 
\end{equation}
Multiplying \eqref{Faedo--Galerkin:mu_0} with $(\mu_0^m,y_k)$, summing from $k=1$ to $m$ and using H\"{o}lder's inequality \eqref{inequality:Holder}, we have the estimate
\begin{equation}
 \|\mu_0^m\|\leq c\|\phi^m_0\|+\lambda\|\nabla\cdot\vec{u}^m_0\|\leq c\|\phi^m_0\|+\lambda\|\vec{u}^m_0\|_{\H^1(\Omega;\R^d)}.\label{est:mu_0}  
\end{equation}

\medskip
\noindent \textbf{Estimates for $\vec{u}^m$}. 
Multiplying \eqref{Faedo--Galerkin:u} with $\Cu^m_k(t)$, and summing from $k=1$ to $m$, we have 
\begin{equation*}
    \begin{aligned}
2G\|\symgrad(\vec{u}^m)\|^2+ \frac{2G\nu}{1-2\nu}\|\nabla\cdot \vec{u}^m \|^2 &=- \lambda\big(\phi^m,\nabla\cdot\vec{u}^m\big).
    \end{aligned}
\end{equation*}
Using $\epsilon$-Young's \eqref{inequality:Youngs} and Korn's inequality \eqref{inequality:Poincare_Korn_Sobolev}, we have 
\begin{equation}\label{est:u}
    \begin{aligned}
C\|\vec{u}^m\|^2_{\H^1(\Omega;\R^d)}\leq2G\|\symgrad(\vec{u}^m)\|^2+ \frac{G\nu}{1-2\nu}\|\nabla\cdot \vec{u}^m \|^2 &\leq \frac{\lambda^2(1-2\nu)}{4G\nu}\|\phi^m\|^2.
    \end{aligned}
\end{equation}
Multiplying \eqref{Faedo--Galerkin:u_0} with $(\vec{u}_0^m,\vec{w}_k)$, and summing from $k=1$ to $m$, we estimate as before to get
\begin{equation}\label{est:u_0}
    \begin{aligned}
C\|\vec{u}^m_0\|^2_{\H^1(\Omega;\R^d)}\leq 2G\|\symgrad(\vec{u}_0^m)\|^2+ \frac{G\nu}{1-2\nu}\|\nabla\cdot \vec{u}_0^m\|^2&\leq \frac{\lambda^2(1-2\nu)}{4G\nu}\|\phi^m_0\|^2.
    \end{aligned}
\end{equation}

\medskip
\noindent \textbf{Estimates for $\phi^m$}. Multiplying \eqref{Faedo--Galerkin:phi} with $\Cmu^m_k(t)$, \eqref{Faedo--Galerkin:mu} with $-\dt \left(g_{1-\alpha}*(\Cphi^m_k-(\phi_0,y_k))\right)(t)$, \eqref{Faedo--Galerkin:u} with $\dt\left( g_{1-\alpha}*\left(\Cu^m_k-(\vec{u}_0^m,\vec{w}_k)\right)\right)(t)+\Cu^m_k(t)$, and summing from $k=1$ to $m$, we have
\begin{subequations} \label{est:initial}
 \begin{align}
\left(\p_t^{\alpha}(\phi^m-\phi_0^m),\mu^m\right) &+ (M_\phi\nabla\mu^m,\nabla \mu^m)  =  N_\phi\big(f^m,\mu^m\big)- P_\phi\big(g^m,\mu^m\big),\label{est:initial:phi}\\
-(\mu^m,\p_t^{\alpha}(\phi^m-\phi_0^m)) &=-c(\phi^m,\p_t^{\alpha}(\phi^m-\phi_0^m)) - \lambda(\nabla\cdot\vec{u}^m,\p_t^{\alpha}(\phi^m-\phi_0^m)),\label{est:initial:mu}\\
2G\big(\symgrad(\vec{u}^m),\p_t^{\alpha}\left(\symgrad(\vec{u}^m)-\symgrad(\vec{u}^m_0)\right)\big)&+ \frac{2G\nu}{1-2\nu}\big(\nabla\cdot \vec{u}^m ,\p_t^{\alpha}\left(\nabla\cdot\vec{u}^m-\nabla\cdot\vec{u}^m_0\right)\big) =- \lambda\big(\phi^m,\p_t^{\alpha}\left(\nabla\cdot\vec{u}^m-\nabla\cdot\vec{u}^m_0\right)\big),\label{est:initial:u}\\
2G\big(\symgrad(\vec{u}^m),\symgrad(\vec{u}^m)\big)&+ \frac{2G\nu}{1-2\nu}\big(\nabla\cdot \vec{u}^m ,\nabla\cdot\vec{u}^m\big) =- \lambda\big(\phi^m,\nabla\cdot\vec{u}^m\big).\label{est:initial:u_1}
\end{align}
\end{subequations}
Upon adding \eqref{est:initial:phi}-\eqref{est:initial:u}, we obtain
\begin{equation*}
    \begin{aligned}
  c(\phi^m,\p_t^{\alpha}(\phi^m-\phi_0^m)) + (\p_{\phi}W(\phi^m,\symgrad(\vec{u}^m)),\p_t^{\alpha}(\phi^m-\phi_0^m))+ \big(\p_{\symgrad}W(\phi^m,\symgrad(\vec{u}^m)),\p_t^{\alpha}\left(\symgrad(\vec{u}^m)-\symgrad(\vec{u}^m_0)\right)\big)\\ + (M_\phi\nabla\mu^m,\nabla \mu^m) 
= N_\phi\big(f^m,\mu^m\big)- P_\phi\big(g^m,\mu^m\big),
    \end{aligned}
\end{equation*}
where we have, from Section \ref{sec:mathematical_modelling},
\begin{equation*}
\begin{alignedat}{10}
  W(\phi,\symgrad) &= \frac{1}{2} \symgrad:\vec{C}\symgrad + \symgrad:\lambda\phi\mathbb{I},\quad&&
\vec{C}\symgrad=2G\symgrad+\frac{2G\nu}{1-2\nu} \tr\symgrad \mathbb{I}, \\
\p_{\phi}W(\phi,\symgrad(\vec{u})) &= \symgrad(\vec{u}):\lambda\mathbb{I}=\lambda\nabla\cdot\vec{u}, \quad&&  \p_{\symgrad}W(\phi,\symgrad(\vec{u}))= \vec{C}\symgrad(\vec{u}) + \lambda\phi\mathbb{I}=2G\symgrad(\vec{u})+\frac{2G\nu}{1-2\nu} \nabla\cdot\vec{u}\mathbb{I}+\lambda\phi\mathbb{I}.
\end{alignedat}
\end{equation*}
We see that the convex functional $W(\phi^m,\symgrad(\vec{u}^m))$ satisfies the assumptions in Lemma \ref{Lem_basic_inequality_fractional_chain_rule} by noticing that 
\begin{equation}\label{form:int_W}
\int_\Omega W(\phi^m,\symgrad(\vec{u}^m))\dd \vec{x} = -G\|\symgrad(\vec{u}^m)\|^2-\frac{G\nu}{1-2\nu}\|\nabla\cdot\vec{u}^m\|^2,
\end{equation}
which is obtained using \eqref{est:initial:u_1}. We now apply Lemma \ref{Lem_basic_inequality_fractional_chain_rule} with $\Func(\varphi)=W(\phi^m,\symgrad(\vec{u}^m))$ and \eqref{basic_inequality_fractional}, and get the following estimate
\begin{equation*}
    \begin{aligned}
  \dt \left(g_{1-\alpha}*\left(\frac{c}{2}\|\phi^m\|^2+\int_\Omega W(\phi^m,\symgrad(\vec{u}^m))\dd \vec{x}\right)\right)(t)
  +\left(c(\phi^m,\phi^m-\phi_0^m)-\frac{c}{2}\|\phi^m\|^2\right)g_{1-\alpha}(t) \\ +\left(\big(\p_{\phi}W(\phi^m,\symgrad(\vec{u}^m)),\phi^m-\phi_0^m\big)+ \big(\p_{\symgrad}W(\phi^m,\symgrad(\vec{u}^m)),\symgrad(\vec{u}^m)-\symgrad(\vec{u}^m_0)\big)-\int_\Omega W(\phi^m,\symgrad(\vec{u}^m))\dd \vec{x}\right) g_{1-\alpha}(t) \\ + (M_\phi\nabla\mu^m,\nabla \mu^m) 
\leq N_\phi\big(f^m,\mu^m\big)- P_\phi\big(g^m,\mu^m\big).
    \end{aligned}
\end{equation*}
Using the convexity of the functionals $W(\phi^m,\symgrad(\vec{u}^m))$ and $\frac{c}{2}(\phi^m)^2$, \ref{Ass:mobility}, \eqref{est:mu}, \eqref{est:u} and \eqref{est:u_0}, we have
\begin{equation}\label{est:middle:phi}
    \begin{aligned}
  \dt \left(g_{1-\alpha}*\left(\frac{c}{2}\|\phi^m\|^2+\int_\Omega W(\phi^m,\symgrad(\vec{u}^m))\dd \vec{x}\right)\right)(t)
  + M_0\|\nabla\mu^m\|^2 
\leq\begin{multlined}[t]
C\left(\|\phi^m_0\|^2g_{1-\alpha}(t) + \|\phi^m\|^2\right) \\+  \frac{N_\phi}{2}\|f^m\|^2+ \frac{P_\phi}{2}\|g^m\|^2.
\end{multlined} 
    \end{aligned}
\end{equation}
All the terms in the above estimate belong to $\L_1(0,T)$, and so we \revision{convolve} with $g_{\alpha}$ to get a bound for $\phi^m$ in the space $\L_2(0,T;\L_2(\Omega))$. Using the fact that $g_{1-\alpha}*\left(\frac{c}{2}\|\phi^m\|^2+\int_\Omega W(\phi^m,\symgrad(\vec{u}^m))\dd \vec{x}\right)(0)=0$ and the auxiliary result $g_{\alpha}*g_{1-\alpha} =g_{1}$, see \cite[Theorem 2.2]{diethelm2010analysis}, we have 
\begin{align*}
g_{\alpha}*\dt \left(g_{1-\alpha}*\left(\frac{c}{2}\|\phi^m\|^2+\int_\Omega W(\phi^m,\symgrad(\vec{u}^m))\dd \vec{x}\right)\right)\\
= \dt\left( g_{\alpha}*g_{1-\alpha}*\left(\frac{c}{2}\|\phi^m\|^2+\int_\Omega W(\phi^m,\symgrad(\vec{u}^m))\dd \vec{x}\right)\right) = \frac{c}{2}\|\phi^m\|^2+\int_\Omega W(\phi^m,\symgrad(\vec{u}^m))\dd \vec{x}.
\end{align*}
\revision{Convolving} \eqref{est:middle:phi} with $g_{\alpha}$, we have for almost all $t\in[0,T]$,
\begin{equation*}
    \begin{aligned}
\frac{c}{2}\|\phi^m\|^2+\int_\Omega W(\phi^m,\symgrad(\vec{u}^m))\dd \vec{x}
  + M_0\left(g_{\alpha}*\|\nabla\mu^m\|^2 \right)(t)
\leq C\|\phi_0^m\|^2 +C\left(g_{\alpha}*\|\phi^m\|^2\right)(t)\\ +\frac{N_\phi}{2}\left(g_{\alpha}*\|f^m\|^2\right)(t)+ \frac{P_\phi}{2}\left(g_{\alpha}*\|g^m\|^2\right)(t).
    \end{aligned}
\end{equation*}
Further, using \eqref{form:int_W} and \eqref{est:u}, we have 
\begin{equation*}
    \begin{aligned}
\frac{1}{2}\left(c-\frac{\lambda^2(1-2\nu)}{2G\nu}\right)\|\phi^m\|^2 + M_0\left(g_{\alpha}*\|\nabla\mu^m\|^2\right)(t)
\leq 
 C\|\phi_0^m\|^2 +C\left(g_{\alpha}*\|\phi^m\|^2\right)(t) \\
 + \frac{N_\phi}{2}\left(g_{\alpha}*\|f^m\|^2\right)(t)+ \frac{P_\phi}{2}\left(g_{\alpha}*\|g^m\|^2\right)(t),
    \end{aligned}
\end{equation*}
where the constant in the first term is positive by \ref{Ass:paramters}. Using the generalised Gronwall--Bellman Lemma \ref{Lem_Gronwall_fractional}, integrating from $0$ to $T$,
using Young's inequality for convolution \eqref{inequality:Youngs} and the property of orthogonal projection, we obtain the upper bound
\begin{equation} \label{est:phi}
	\|\phi^m\|_{\L_2(0,T;\L_2(\Omega))}^2 \leq C\left( \|\phi_0\|^2+C_f+C_g\right).
\end{equation}
Moreover, integrating \eqref{est:middle:phi} from $0$ to $T$, using \eqref{est:phi} and the fact that $g_{1-\alpha}$ is positive yields
\begin{equation} \label{est:nabla_mu} 
\|\nabla \mu^m\|^2_{\L_2(0,T;\L_2(\Omega))} \leq C\left( \|\phi_0\|^2+C_f+C_g\right).
\end{equation}

\medskip
\noindent \textbf{Estimates for $\psi^m$.} We have stated in \eqref{Faedo--Galerkin:psi} the following Faedo--Galerkin equation for $\psi^m$,
\begin{align*}
	\big(\partial_t \psi^m,z_k\big) + \big(M_\psi\nabla \psi^m,\nabla z_k\big)&= \big(S_\psi,z_k\big) - N_\psi \big(f^m,z_k\big) + \left(\psi_b^m,z_k\right)_{\p\Omega\backslash\Dirpsi},
\end{align*}
and by multiplying this equation by $\Cpsi^m_k(t)$ and taking summation over $k=1$ to $m$ and using (A4), we arrive at
\begin{align*}
\big(\partial_t \psi^m,\psi^m\big) + M_0 \|\nabla \psi^m\|^2 &\leq  N_\psi \big(f^m,\psi^m\big) + \big(S_\psi,\psi^m\big) +  \left(\psi_b^m,\psi^m\right)_{\p\Omega\backslash\Dirpsi}.
\end{align*}
Using the inequalities \eqref{inequality:Holder} and \eqref{inequality:Youngs}, we estimate the right hand side and get the following upper bound
\begin{align*}
\frac{1}{2}\dt \|\psi^m\|^2+ M_0 \|\nabla \psi^m\|^2 &\leq \bigg(\frac{N_\psi}{2}+\frac{1}{2}\bigg)\|\psi^m\|^2+ \frac{N_\psi}{2}\|f^m\|^2 +\frac{1}{2}\|S_\psi\|^2 + \|\psi_b^m\|_{\p\Omega\backslash\Dirpsi} \|\Upsilon\psi^m\|_{\p\Omega\backslash\Dirpsi},
\end{align*}
where $\Upsilon:\H^1(\Omega)\rightarrow\L_2(\p\Omega\backslash\Dirpsi)$ is the trace operator. Since the trace operator is continuous, we have $\|\Upsilon\varphi\|_{\p\Omega\backslash\Dirpsi}\leq C\|\varphi\|_{\H^1(\Omega)}$ for every $\varphi\in\H^1(\Omega)$, see \cite[Section 5.5, Theorem 1]{evans2010partial}. Applying $\epsilon$-Young's inequality \eqref{inequality:Youngs}, we find
\begin{align*}
\frac{1}{2}\dt \|\psi^m\|^2+ \frac{M_0}{2} \|\nabla \psi^m\|^2 &\leq \bigg(\frac{N_\psi}{2}+1\bigg)\|\psi^m\|^2+ \frac{N_\psi}{2}\|f^m\|^2 +\frac{1}{2}\|S_\psi\|^2 + 
C\|\psi_b^m\|^2_{\p\Omega\backslash\Dirpsi}.
\end{align*}
Finally the Gronwall--Bellman Lemma \ref{Lem_Gronwall}, yields the estimate,
\begin{equation} \label{final:4}
	\|\psi^m\|_{\L_2(0,T;\L_2(\Omega))}^2+\|\nabla \psi^m\|_{\L_2(0,T;\L_2(\Omega))}^2 \leq \begin{multlined}[t] C \big(\|\psi_0\|^2 + C_f+\|S_{\psi}\|_{\L_2(0,T;\L_2(\Omega))}^2\\ +\|\psi_b\|_{\L_2(0,T;\L_2(\p\Omega\backslash\Dirpsi))}^2\big). \end{multlined}
\end{equation}

\medskip
\noindent \textbf{Estimates for $\chi^m$.} 
Multiplying \eqref{Faedo--Galerkin:chi} with $\Cchi^m_k(t)$ and taking summation over $k=1$ to $m$, using the typical inequalities and proceeding as in the estimates for $\psi$, we have the estimate
\begin{equation} \label{final:5}
\|\chi^m\|_{\L_2(0,T;\L_2(\Omega))}^2+\|\nabla \chi^m\|_{\L_2(0,T;\L_2(\Omega))}^2 \leq \begin{multlined}[t] C \big(\|\chi_0\|^2 +C_g  +\|S_{\chi}\|_{\L_2(0,T;\L_2(\Omega))}^2\\ +\|\chi_b\|_{\L_2(0,T;\L_2(\p\Omega\backslash\Dirpsi))}^2\big). \end{multlined}
\end{equation}
Summing the equations \eqref{est:mu}, \eqref{est:u}, \eqref{est:phi}--
\eqref{final:5}, we arrive at the energy estimate
\begin{equation} \label{final:estimate}
    \begin{aligned}
     \|\phi^m\|_{\L_2(0,T;\L_2(\Omega))}^2 &+ \|\mu^m\|^2_{\L_2(0,T;\H^1(\Omega))}+\|\vec{u}^m\|_{\L_2(0,T;\H^1(\Omega;\R^d))}^2+
    \| \psi^m\|_{\L_2(0,T;\H^1(\Omega))}^2 \\ 
    +\| \chi^m\|_{\L_2(0,T;\H^1(\Omega))}^2 &\leq \begin{multlined}[t]C \big(\textup{IC}+C_f+C_g+\|S_{\psi}\|_{\L_2(0,T;\L_2(\Omega))}^2 +\|S_{\chi}\|_{\L_2(0,T;\L_2(\Omega))}^2 \\+ \|\psi_b\|_{\L_2(0,T;\L_2(\p\Omega\backslash\Dirpsi))}^2 + \|\chi_b\|_{\L_2(0,T;\L_2(\p\Omega\backslash\Dirpsi))}^2\big).\end{multlined}
    \end{aligned}
\end{equation}

\medskip
\noindent \textbf{Estimates for the time derivatives.}
Since our equations in which we wish to pass to the limit have nonlinear functions in $\phi^m,\psi^m,\chi^m$, we need the strong convergence of these sequences. For this purpose we bound the time derivatives and use the compactness results \eqref{Lem_Aubin} and \eqref{Lem_FractionalAubin}.
\par We first obtain the estimate of time derivative of $\phi^m$ using the estimates from the time derivatives of $\mu^m$ and $\vec{u}^m$.
Multiplying \eqref{Faedo--Galerkin:mu_derivative} with $\dt g_{1-\alpha}*(\Cmu^m_k(t)-(\mu_0^m,y_k))$, summing from $k=1$ to $m$, and estimating using H\"older's inequality \eqref{inequality:Holder}, we have
\begin{equation}\label{est:time_derivative_mu}
\left\|\FDmuM\right\| \leq  c\left\|\FDphiM\right\|+\lambda\left\|\FDdivuM\right\|. 
\end{equation}
Multiplying \eqref{Faedo--Galerkin:u_derivative} with $\dt g_{1-\alpha}*(\Cu^m_k(t)-(\vec{u}_0^m,\vec{w}_k))$, summing from $k=1$ to $m$, we have
\begin{equation*}
    \begin{aligned}
       2G\left\|\FDsymuM\right\|^2+ \frac{2G\nu}{1-2\nu}\left\|\FDdivuM\right\|^2 =- \lambda\left(\FDphiM,\FDdivuM\right).
    \end{aligned}
\end{equation*}
Using H\"{o}lder's inequality \eqref{inequality:Holder} gives us
\revision{\begin{equation}\label{est:time_derivative_u}
    \begin{aligned}
   \frac{2G\nu}{1-2\nu}\left\|\FDdivuM\right\|^2 \leq  \lambda\left\|\FDphiM\right\|^2.
    \end{aligned}
\end{equation}}
Multiplying \eqref{Faedo--Galerkin:nabla_mu_0} with $(\mu_0^m,y_k)$ and summing from $k=1$ to $m$, we obtain
\begin{equation*}
\|\nabla\mu_0^m\|^2\leq \lambda_k^y c(\phi_0^m, \mu_0^m)+\lambda_k^y \lambda (\nabla\cdot\vec{u}_0^m,\mu_0^m).
\end{equation*}
Using H\"{o}lder's \eqref{inequality:Holder}, \eqref{est:mu_0} and \eqref{est:u_0}, we have
\begin{equation}
\|\nabla\mu_0^m\|^2\leq C\|\phi_0^m\|^2.\label{est:nabla_mu_0}
\end{equation}

We now use the above two estimates to obtain the estimate of time derivative of $\phi^m$. Multiplying \eqref{Faedo--Galerkin:phi} with $\dt g_{1-\alpha}*(\Cphi^m_k-(\phi_0,y_k))$, and \eqref{Faedo--Galerkin:mu_derivative} with $\dt g_{1-\alpha}*(\Cmu^m_k(t)-(\mu_0^m,y_k))$, and summing we get
\begin{equation*}
    \begin{aligned}
c\left\|\FDphiM\right\|^2 + \left(M_\phi\nabla\mu^m,\nabla\FDmuM\right)  &=  N_\phi\big(f^m,\FDmuM\big)- P_\phi\big(g^m,\FDmuM\big) \\
&- \lambda \left(\FDdivuM,\FDphiM\right).
    \end{aligned}
\end{equation*}
Using \eqref{basic_inequality_fractional} and H\"older's inequality \eqref{inequality:Holder}, we have
\begin{equation*}
    \begin{aligned}
c\left\|\p_t^{\alpha}(\phi^m-\phi_0^m)\right\|^2 + \frac{M_0}{2}\dt&\left( g_{1-\alpha}*\|\nabla\mu^m\|^2\right)(t)   \leq \left(N_\phi\|f^m\|+ P_\phi\|g^m\|\right)\left\|\FDmuM\right\|\\
&+ g_{1-\alpha}(t)\|\nabla\mu^m_0\| + \lambda \left\|\FDdivuM\right\|\left\|\FDphiM\right\|.
    \end{aligned}
\end{equation*}
Using \eqref{est:time_derivative_mu} and \eqref{est:time_derivative_u}, we have, for every $\epsilon_1,\epsilon_2>0$,
\begin{equation*}
    \begin{aligned}
c\left\|\FDphiM\right\|^2 + \frac{M_0}{2}\dt&\left( g_{1-\alpha}*\|\nabla\mu^m\|^2\right)(t)   \leq 
 \left(\frac{\epsilon_1+\epsilon_2}{4\epsilon_1\epsilon_2}\right) \left(N_\phi\|f^m\|^2+ P_\phi\|g^m\|^2\right)\\
&+ g_{1-\alpha}(t) \|\nabla\mu^m_0\|^2 + \left(\epsilon_1 c+(\epsilon_2 +1)\frac{\lambda^2(1-2\nu)}{2G\nu}\right)\left\|\FDphiM\right\|^2.
    \end{aligned}
\end{equation*}
Choosing $\epsilon_1$ and $\epsilon_2$ appropriately, using assumption \ref{Ass:paramters} and integrating from $0$ to $T$, using $g_{1-\alpha}$ is positive and \eqref{est:nabla_mu_0}, we get an upper bound
\begin{equation}\label{est:time_derivative_phi}
\frac{1}{2}\left(c-\frac{\lambda^2(1-2\nu)}{2G\nu}\right)\left\|\FDphiM\right\|^2_{\L_2(0,T;\L_2(\Omega))}  \leq
\begin{multlined}[t]
C(\|\phi^m_0\|^2+C_f+C_g).
\end{multlined} 
\end{equation}

\par We now obtain the estimates of time derivatives of $\psi^m$ and $\chi^m$. Let $\zeta_1\in \L_2(0,T;\H_{0,\Dirpsi}^1(\Omega))$, 
 such that 
$\Pi_{\PZ^m}\zeta_1 = \sum_{k=1}^m\zeta_{1,k}z_k$.
We use the boundedness of the projection and the invariance of the time derivatives under the adjoint operator of 
$\Pi_{\PZ^m}$, i.e., 
\begin{equation*}\label{identity:projection_under_time-derivative}
\begin{aligned}
\langle \partial_t \psi^m,\zeta_1\rangle = \langle \partial_t \psi^m,\Pi_{\PZ^m}\zeta_1\rangle,
	\end{aligned}
\end{equation*}
see \cite[Lemma V.1.6]{boyer2013navier}. Multiplying the Faedo--Galerkin equations 
\eqref{Faedo--Galerkin:psi} with $\zeta_{1,k}$ 
yields
\begin{equation}\label{est:time_derivative_psi}
\begin{aligned}
\int_0^T &\langle \partial_t \psi^m,\zeta_1\rangle \dd t = \begin{multlined}[t] -\int_0^T  \big(M_\psi\nabla \psi^m,\nabla \Pi_{\PZ^m}\zeta_1 \big)\dd t -\int_0^T N_\psi \big(f^m,\Pi_{\PZ^m}\zeta_1\big)\dd t \\+\int_0^T \big(S_\psi,\Pi_{\PZ^m}\zeta_1\big)\dd t + \int_0^T\big(\psi_b,\Pi_{\PZ^m}\zeta_1\big)_{\p\Omega\backslash\Dirpsi}\dd t,\end{multlined} \\
&\leq C\big(\|\psi_0\|+C_f+ \|S_\psi\|_{\L_2(0,T;\L_2(\Omega))}\big)\|\nabla \Pi_{\PZ^m}\zeta_1\|_{\L_2(0,T;\L_2(\Omega))},\\
&\leq C\big(\|\psi_0\|+C_f+ \|S_\psi\|_{\L_2(0,T;\L_2(\Omega))} + \|\psi_b\|_{\L_2(0,T;\L_2(\p\Omega\backslash\Dirpsi))} \big) \|\zeta_1\|_{\L_2(0,T;\H^1(\Omega))},
\end{aligned}
\end{equation}
and 
\begin{equation}\textbf{\label{est:time_derivative_chi}}
    \int_0^T\langle \partial_t \chi^m,\zeta_1\rangle \dd t\leq C(T,g,S_\chi,\chi_0,\chi_b) \|\zeta_1\|_{\L_2(0,T;\H^1(\Omega))}.
\end{equation}


\subsection{Existence of a weak solution}\label{subsec:existence_weak_solution}
We now prove that there is a subsequence of $\phi^m,\mu^m,\vec{u}^m,\psi^m,\chi^m$ which converges to the weak solution of our model \eqref{System} in the sense of Definition \ref{Def:WeakSolution}. We prove this by showing that the limit functions \revision{satisfy} the variational form \eqref{eqn:variational_form} and also \revision{satisfy} the initial conditions.

\noindent\textbf{Weak Convergence.}
The energy estimate \eqref{final:estimate} provides us the following 
\begin{equation}
\label{weak_convergence_bounds}
\begin{aligned}
\{\phi^m\} &\text{ bounded in }  \L_2(0,T;\L_2(\Omega)),\\
\{\mu^m\}, \{\psi^m\}, \{\chi^m\} &\text{ bounded in }  \L_2(0,T;\H^1(\Omega)),\\
\{\vec{u}^m\} &\text{ bounded in } \L_2(0,T;\H^1(\Omega;\R^d)).
\end{aligned}
\end{equation}
By the Banach--Alaoglu theorem, these bounded sequences have weakly convergent subsequences which we indicate with the same index. Hence, there exist functions $\phi,\mu,\psi,\chi:(0,T)\times\Omega\rightarrow \R$ and $\vec{u}:(0,T)\times\Omega\rightarrow \R^d$ such that as $m\rightarrow \infty$ we have the following weak convergences
\begin{equation} \label{Eq:WeakConvergence}
\begin{aligned}
\phi^m \rightharpoonup \phi &\quad \mbox{ in }\quad \L_2(0,T;\L_2(\Omega)),\\
	\mu^m\rightharpoonup \mu,\ \psi^m\rightharpoonup \psi,\ \chi^m\rightharpoonup \chi &\quad \mbox{ in }\quad \L_2(0,T;\H^1(\Omega)),\\
	\vec{u}^m\rightharpoonup \vec{u} &\quad \mbox{ in }\quad \L_2(0,T;\H^1(\Omega;\R^d).
\end{aligned}
\end{equation}

\noindent\textbf{Strong Convergence.}
From the inequalities \eqref{est:time_derivative_phi}-\eqref{est:time_derivative_chi}, we conclude that 
$$\begin{aligned}
	\{\phi^m\}&\text{ bounded in }  \W^\alpha_{2,2}(0,T;\phi_0,\L_2(\Omega),\L_2(\Omega)),\\
	\{\psi^m\},\{\chi^m\}&\text{ bounded in } \W^1_{2,2}(0,T;\H^1(\Omega),\H^{-1}(\Omega)). 
\end{aligned}$$
Using the Aubin--Lions compactness theorem and compactness results similar for fractional differential equations, see \eqref{Lem_Aubin} and \eqref{Lem_FractionalAubin}, we have
\begin{eqnarray*}
\W^{\alpha}_{2,2}(0,T;\phi_0,\L_2(\Omega),\L_2(\Omega))\doublehookrightarrow \L_2(0,T;\L_2(\Omega)),  \\
	\W^1_{2,2}(0,T;\H^1(\Omega),\H^{-1}(\Omega)) \doublehookrightarrow \L_2(0,T;\L_2(\Omega)),
\end{eqnarray*}
and therefore we have the strong convergences (as $m \to \infty$)
\begin{equation} \label{Eq:StrongConvergence}
\begin{aligned}
\phi^m \rightarrow \phi,\ \psi^m \rightarrow \psi,\ \chi^m \rightarrow \chi &\quad \mbox{ in } \L_2(0,T;\L_2(\Omega)).
\end{aligned}
\end{equation}


\noindent\textbf{Variational form.}
We now show the limit functions satisfy the variational form \eqref{eqn:variational_form}. Let $\eta\in C^\infty_0(0,T)$, 
multiplying the Faedo--Galerkin system \eqref{Faedo--Galerkin} by $\eta$ and integrating from $0$ to $T$, we have 
\begin{subequations}\label{limit_VF}
\begin{align}
\int_0^T&\big(\partial_t^{\alpha}(\phi^m-\phi^m_0),\eta(t)y_k\big)\dd t+\int_0^T  (M_\phi\nabla\mu^m,\eta(t)\nabla y_k) \dd t \nonumber\\
&=  N_\phi\int_0^T \big(f^m,\eta(t)y_k\big)\dd t- P_\phi\int_0^T \big(g^m,\eta(t)y_k\big)\dd t,\\
\int_0^T & (\mu^m,\eta(t)y_k) \dd t=c\int_0^T(\phi^m,\eta(t)y_k)\dd t + \lambda\int_0^T(\nabla\cdot\vec{u}^m,\eta(t)y_k)\dd t,\\
 2G\int_0^T&\big(\symgrad(\vec{u}^m),\eta(t)\symgrad(\vec{w}_k)\big)\dd t+ \frac{2G\nu}{1-2\nu} \int_0^T \big(\nabla\cdot \vec{u}^m ,\eta(t)\nabla\cdot\vec{w}_k\big)\dd t =-\int_0^T \lambda\big(\phi^m,\eta(t)\nabla\cdot\vec{w}_k\big)\dd t, \\
\int_0^T &\big(\partial_t \psi^m,\eta(t)z_k\big)\dd t +\int_0^T \big(M_\psi\nabla \psi^m,\eta(t)\nabla z_k\big)\dd t \nonumber\\
&= \int_0^T \big(S_\psi,\eta(t)z_k\big)\dd t - N_\psi \int_0^T \big(f^m,\eta(t)z_k\big)\dd t + \int_0^T\big(\psi_b^m,z_k\big)_{\p\Omega\backslash\Dirpsi}\dd t,\\
\int_0^T &\big(\partial_t \chi^m,\eta(t)z_k\big)\dd t +\int_0^T \big(M_\chi\nabla \chi^m,\eta(t)\nabla z_k\big)\dd t \nonumber\\
&=  \int_0^T \big(S_\chi,\eta(t)z_k\big)\dd t - N_\chi\int_0^T  \big(\chi^m,\eta(t)z_k\big) \dd t  - P_\chi\int_0^T  \big(g^m,\eta(t)z_k\big)\dd t + \int_0^T\big(\chi_b^m,z_k\big)_{\p\Omega\backslash\Dirpsi}\dd t.
\end{align}
\end{subequations}

\par The convergence of the linear terms follows directly from the definition of weak convergence. For instance, the functional 
\begin{equation*}
\begin{aligned}
\mu^m\mapsto\int_0^T (\nabla\mu^m,\eta(t)\nabla y_k) \dd t
&\leq \|\mu^m\|_{\L_2(0,T;\H^1(\Omega))}\|\eta\|_{\L_2(0,T)}\|\nabla y_k\|,\\
\end{aligned}
\end{equation*}
is linear and continuous on $\L_2(0,T;\H^1(\Omega))$ and therefore, we get from \eqref{Eq:WeakConvergence} that
$$\int_0^T (\nabla\mu^m,\eta(t)\nabla y_k) \dd t\rightarrow \int_0^T (\nabla\mu,\eta(t)\nabla y_k) \dd t,$$
for $m\rightarrow \infty$. The terms with time derivatives follow from integration by parts, change of integration and the definition of the weak convergence. The functionals
\begin{equation*}
\begin{aligned}
\phi^m&\mapsto\int_0^T\big(\p_t\left( g_{1-\alpha}*(\phi^m-\phi^m_0)\right)(t),\eta(t)y_k\big)\dd t = \begin{multlined}[t]-\int_0^T\left((\phi^m-\phi^m_0), \left(g_{1-\alpha}*'\p_t\eta\right)(t) y_k\right)\dd t,\\
\!\!\!\!\!\!\!\!\!\!\!\!\!\!\!\!\!\!\!\!\!\!\!\!\!\!\!\!\!\!\!\!\!\hspace{-1cm}\leq \|\phi^m-\phi^m_0\|_{\L_2(0,T;\L_2(\Omega))} \|g_{1-\alpha}\|_{\L_1(0,T)}\|\eta'\|_{\L_2(0,T)}\|y_k\|, \end{multlined}\\
\psi^m&\mapsto \int_0^T \big(\partial_t \psi^m,\eta(t)z_k\big)\dd t = -\int_0^T \big( \psi^m,\eta'(t)z_k\big)\dd t\leq \|\psi^m\|_{\L_2(0,T;\L_2(\Omega))}\|\eta'\|_{\L_2(0,T)}\|z_k\|,\\	
\chi^m&\mapsto\int_0^T \big(\partial_t \chi^m,\eta(t)z_k\big)\dd t = -\int_0^T \big( \chi^m,\eta'(t)w_k\big)\dd t\leq \|\chi^m\|_{\L_2(0,T;\L_2(\Omega))}\|\eta'\|_{\L_2(0,T)}\|z_k\|,
\end{aligned}
\end{equation*}
as we see are linear and continuous on $\L_2(0,T;\L_2(\Omega))$, and so by weak convergence and reapplying the integration by parts, we obtain the limit for the terms with time derivatives in \eqref{limit_VF}.

\par The strong convergence results in \eqref{Eq:StrongConvergence} give us the limits of the terms involving nonlinear functions. From the strong convergence, we have
	\begin{equation*}
		\phi^m\rightarrow \phi,\ \psi^m\rightarrow \psi \quad \mbox{ in }\L_2(0,T;\L_2(\Omega))\cong \L_2((0,T)\times \Omega),\quad \mbox{ as } m\rightarrow \infty.  
	\end{equation*}
This implies there exist subsequences such that they converge almost everywhere on $(0,T)\times \Omega$. Since almost everywhere convergence is preserved under composition of a continuous functional, we have
\begin{equation*}
f^m=f(\phi^m,\psi^m)\rightarrow f(\phi,\psi)\quad \mbox{ a.e. in }(0,T)\times \Omega,\quad \mbox{ as } m\rightarrow \infty.
\end{equation*}
Since $\{f^m\}$ is bounded, we have by the Lebesgue dominated convergence theorem 
\begin{equation*}
	f^m\eta(t)y_k\rightarrow f(\phi,\psi)\eta(t)y_k\quad \mbox{ in }\L_1(0,T\times\Omega),\quad \mbox{ as } m\rightarrow \infty.
\end{equation*}
Convergence of the terms involving $g^m$ follow analogously. 

By the above convergence results, we have that \eqref{limit_VF} holds true with the limit functions for all $\eta\in C_0^\infty(0,T)$. By the Lemma of du Bois-Reymond, we get that the limit functions $(\phi,\mu,\vec{u},\psi,\chi)$ satisfy
\begin{equation*}
\begin{aligned}
\left(\partial_t^{\alpha}(\phi(t)-\phi_0),y_k\right)+ (M_\phi\nabla\mu,\nabla y_k)  &= \begin{multlined}[t] N_\phi\big(f(\phi,\psi),y_k\big) - P_\phi\big(g(\phi,\chi),y_k\big),\end{multlined}\\
(\mu,y_k) &= c(\phi,y_k)+\lambda(\nabla\cdot\vec{u},y_k),\\
\left<\partial_t \psi,z_k\right> + \big(M_\psi \nabla \psi,\nabla z_k\big)&= \big(S_\psi,z_k\big) - N_\psi \big(f(\phi,\psi),z_k\big) + \big(\psi_b,z_k\big)_{\p\Omega\backslash\Dirpsi}, \\
\left<\partial_t \chi,z_k\right> + \big(M_\chi \nabla \psi,\nabla z_k\big)&= \big(S_\chi,z_k\big) - N_\chi(\chi,z_k) - P_\chi \big(g(\phi,\chi),z_k\big) + \big(\chi_b,z_k\big)_{\p\Omega\backslash\Dirpsi},\\
-\lambda\big(\phi,\nabla\cdot\vec{w}_k\big)&=2G\big(\symgrad(\vec{u}),\symgrad(\vec{w}_k)\big) + \frac{2G\nu}{1-2\nu}\big(\nabla\cdot \vec{u} ,\nabla\cdot\vec{w}_k\big),
\end{aligned}
\end{equation*}
for almost all $t\in(0,T)$ and for all $k\geq 1$. Using the density of $\cup_{m\in\mathbb{N}}\PY^m$, $\cup_{m\in\mathbb{N}}\PZ^m$, $\cup_{m\in\mathbb{N}}\PW^m$ in $\H^1(\Omega)$, $\H_{0,\Dirpsi}^1(\Omega)$, $\H_{0,\Diru}^1(\Omega;\R^d)$ respectively, we obtain a solution  $(\phi,\mu,\vec{u},\psi,\chi)$ to the system \eqref{System} in the sense of Definition \ref{Def:WeakSolution}, provided they satisfy the initial conditions.\\

\noindent \textbf{Initial conditions.}
We now prove that the limit functions satisfy the initial conditions. From the continuous embedding results \eqref{Lem_InterpolationEmbedding} and \eqref{Lem_ContinousEmbedding_fractional}, we have
\begin{eqnarray*}
	\W^1_{2,2}(0,T;\H^1(\Omega),\H^{-1}(\Omega)) \hookrightarrow \C([0,T];\L_2(\Omega)),\\
	\phi\in \W^\alpha_{2,2}(0,T;\phi_0,\L_2(\Omega),\L_2(\Omega)) \implies \left(g_{1-\alpha}*(\phi-\phi_0)\right)(t)\in \C([0,T];\L_2(\Omega)),
\end{eqnarray*} 
yields $\psi,\chi\in\C([0,T];\L_2(\Omega))$ and $\left(g_{1-\alpha}*(\phi-\phi_0)\right)(t)\in \C([0,T];\L_2(\Omega))$.
 Let $\eta\in \C^1([0,T];\R)$ with $\eta(T)=0$ and $\eta(0)=1$. Then for all $\xi_1\in\H^1(\Omega), \xi_4\in\H^1_{0,\Dirpsi}(\Omega)$, we have
\begin{align*}
\int_0^T\left(\partial_t \left(g_{1-\alpha}*(\phi-\phi_0)\right)(t),\eta(t)\xi_1\right)\dd t&= \begin{multlined}[t]-\int_0^T\big((g_{1-\alpha}*(\phi-\phi_0))(t),\eta'(t)\xi_1\big)\dd t\\
-\big((g_{1-\alpha}*(\phi-\phi_0))(0),\xi_1\big), \end{multlined} \\
\int_0^T \left<\partial_t \psi,\eta(t)\xi_4\right>\dd t&=-\int_0^T \big(\psi(t),\eta'(t)\xi_4\big)\dd t-\big(\psi(0),\xi_4\big).
\end{align*}
Taking the limit $m\rightarrow \infty$, we have
\begin{align*}
-&\int_0^T\big(\left(g_{1-\alpha}*(\phi^m-\phi^m(0))\right)(t),\eta'(t)\xi_1\big)\dd t-\big((g_{1-\alpha}*(\phi^m-\phi^m(0)))(0),\xi_1\big)\\
&\rightarrow -\int_0^T\big(\left(g_{1-\alpha}*(\phi-\phi_0)\right)(t),\eta'(t)\xi_1\big)\dd t-\big(0,\xi_1\big),\\
-&\int_0^T \big(\psi^m(t),\eta'(t)\xi_4\big)\dd t-\big(\psi^m(0),\xi_4\big)\rightarrow -\int_0^T \big(\psi(t),\eta'(t)\xi_4\big)\dd t-\big(\psi_0,\xi_4\big).
\end{align*}
Comparing, we have $\left(g_{1-\alpha}*(\phi-\phi_0)\right)(0)=0$ and $\psi(0)=\psi_0$. Analogously, we get $\chi(0)=\chi_0$.
\begin{remark}
The result $\left(g_{1-\alpha}*(\phi-\phi_0)\right)(0)=0$ does not imply $\phi(0)=\phi_0$. However, the function $\phi_0$ plays the role of an initial data for $\phi$ in a weak sense. Suppose, $\phi$ and $\partial_t \left(g_{1-\alpha}*(\phi-\phi_0)\right)(t)$ are in $\C([0,T];\L_2(\Omega))$ and $\left(g_{1-\alpha}*(\phi-\phi_0)\right)(0)=0$, then we have 
$\phi\in \C([0,T];\L_2(\Omega))$ and $\phi(0)=\phi_0$, see \cite{zacher2009weak}.
\end{remark}

\medskip
\noindent\textbf{Estimates for the weak solution.}
We know that norms are weakly (also weakly-$*$) lower semicontinuous and using the weak convergences in \eqref{Eq:WeakConvergence}
\begin{equation*}
\begin{aligned}
&\|\phi\|_{\L_2(0,T;\L_2(\Omega))}^2+\|\mu\|^2_{\L_2(0,T;\H^1(\Omega))}+\|\vec{u}\|_{\L_2(0,T;\H^1(\Omega;\R^d))}^2+\| \psi\|_{\L_2(0,T;\H^1(\Omega))}^2+\| \chi\|_{\L_2(0,T;\H^1(\Omega))}^2\\
&\leq\liminf_{m\rightarrow\infty}\begin{multlined}[t]\left( \|\phi^m\|_{\L_2(0,T;\L_2(\Omega))}^2\right.+\|\mu^m\|^2_{\L_2(0,T;\H^1(\Omega))}+\|\vec{u}^m\|_{\L_2(0,T;\H^1(\Omega;\R^d))}^2\\+\| \psi^m\|_{\L_2(0,T;\H^1(\Omega))}^2+\left.\| \chi^m\|_{\L_2(0,T;\H^1(\Omega))}^2\right),\end{multlined}\\
&\leq \begin{multlined}[t]C \big(\textup{IC}+C_f+C_g+\|S_{\psi}\|_{\L_2(0,T;\L_2(\Omega))}^2+\|S_{\chi}\|_{\L_2(0,T;\L_2(\Omega))}^2+\|\psi_b\|^2_{\L_2(0,T;\L_2(\p\Omega\backslash\Dirpsi))}\\ +\|\chi_b\|^2_{\L_2(0,T;\L_2(\p\Omega\backslash\Dirpsi))}\big),\end{multlined}
\end{aligned}
\end{equation*}
where the final bound is obtained from \eqref{final:estimate}, and IC is as defined in the theorem statement.

\subsection{Uniqueness}\label{subsec:uniqueness}
Now we prove the uniqueness of the weak solution under the assumption of Lipschitz continuity of the nonlinear functions. We assume that there exist two weak solutions $(\phi_1,\mu_1,\vec{u}_1,\psi_1,\chi_1)$ and $(\phi_2,\mu_2,\vec{u}_2$, $\psi_2,\chi_2)$ to the system and prove that these two solutions have to be identical. Introducing the following notations:
\begin{gather*}
\tilde{\phi} :=\phi_1-\phi_2,\ \ \ \tilde{\mu} :=\mu_1-\mu_2,\ \ \ \tilde{\vec{u}}:=\vec{u}_1-\vec{u}_2,\\
\tilde{\psi} :=\psi_1-\psi_2,\ \ \ \tilde{\chi}:=\chi_1-\chi_2,\\
f_1-f_2:=f(\phi_1,\psi_1)-f(\phi_2,\psi_2),\ \ \ g_1-g_2 :=g(\phi_1,\chi_1)-g(\phi_2,\chi_2),
\end{gather*}
%
we see that $(\tilde{\phi},\tilde{\mu},\tilde{\vec{u}},\tilde{\psi},\tilde{\chi})$ satisfies
\begin{subequations}
\begin{align}
\big(\p_t^{\alpha}\tilde{\phi},\xi_1\big)+ (M_\phi\nabla\tilde{\mu},\nabla \xi_1)  &= N_\phi\big(f_1-f_2,\xi_1\big)-P_\phi\big(g_1-g_2,\xi_1\big),\label{Unique_VF_1}\\
\big(\tilde{\mu},\xi_2\big)&=c\big(\tilde{\phi},\xi_2\big)+\lambda\big(\nabla\cdot\tilde{\vec{u}},\xi_2\big),\label{Unique_VF_5}\\
2G\big(\symgrad(\tilde{\vec{u}}),\symgrad(\vec{\xi}_3)\big)+ \frac{2G\nu}{1-2\nu}\big(\nabla\cdot \tilde{\vec{u}}, \nabla\cdot\vec{\xi}_3\big)&=-\lambda\big(\tilde{\phi},\nabla\cdot\vec{\xi}_3\big),\label{Unique_VF_4}\\
\left< \partial_t \tilde{\psi},\xi_4\right> + \big(M_\psi\nabla \tilde{\psi},\nabla \xi_4\big)&= - N_\psi \big(f_1-f_2,\xi_4\big),\label{Unique_VF_2}\\
\left< \partial_t \tilde{\chi},\xi_4\right> + \big( M_\chi\nabla \tilde{\chi},\nabla \xi_4\big)&=  -N_\chi \big(\tilde{\chi},\xi_4\big)- P_\chi \big(g_1-g_2,\xi_4\big).\label{Unique_VF_3}
\end{align}
\end{subequations}
By choosing the test functions $\tilde{\mu}(t),\p_t^\alpha\tilde{\phi}(t)+\tilde{\mu}(t),\p_t^\alpha\tilde{\vec{u}}(t)+\tilde{\vec{u}},\tilde{\psi}(t),\tilde{\chi}(t)$ for equations \eqref{Unique_VF_1}-\eqref{Unique_VF_3} respectively and proceeding as in the existence part, we have
\begin{subequations}
\begin{align}
\frac{c}{2}\|\tilde{\phi}(t)\|^2
  + M_0\left(g_{\alpha}*\|\nabla\tilde{\mu}\|^2\right)(t) &\leq C\left(g_{\alpha}*\|\tilde{\phi}\|^2\right)(t) +  \frac{N_\phi}{2}\left(g_{\alpha}*\|f_1-f_2\|^2\right)(t)\nonumber\\&+ \frac{P_\phi}{2}\left(g_{\alpha}*\|g_1-g_2\|^2\right)(t),\label{final:uni_phi}\\
\|\tilde{\mu}\| &\leq  c\|\tilde{\phi}\|+\lambda\|\nabla\cdot\tilde{\vec{u}}\|,\label{final:uni_mu}\\
2G\|\symgrad(\tilde{\vec{u}})\|^2 + \frac{G\nu}{1-2\nu}\|\nabla\cdot \tilde{\vec{u}}\|^2 &\leq \frac{\lambda^2(1-2\nu)}{4G\nu}\|\tilde{\phi}\|,\label{final:uni_u}\\
\frac{1}{2}\|\tilde{\psi}(t)\|^2+ M_0 \int_0^t \|\nabla \tilde{\psi}(s)\|^2\, \textup{d}s &\leq \left(\frac{N_\psi}{2}+1\right)\int_0^t \|\tilde{\psi}(s)\|^2\, \textup{d}s+ \frac{N_\psi}{2}\int_0^t \|f_1-f_2\|^2\, \textup{d}s,\label{final:uni_psi}\\
\frac{1}{2}\|\tilde{\chi}(t)\|^2+ M_0 \int_0^t \|\nabla \tilde{\chi}(s)\|^2\, \textup{d}s &\leq \left(N_\chi+\frac{P_\chi}{2}+\frac{1}{2}\right)\int_0^t \|\tilde{\chi}(s)\|^2\, \textup{d}s+ \frac{P_\chi}{2}\int_0^t \|g_1-g_2\|^2\, \textup{d}s.\label{final:uni_chi}
\end{align}
\end{subequations}
We observe that 
\begin{equation*}
\int_0^t \|\varphi(s)\|^2\, \textup{d}s \leq t^{1-\alpha}\Gamma(\alpha)(g_{\alpha}*\|\varphi\|^2)(t).
\end{equation*}
Adding \eqref{final:uni_phi}, \eqref{final:uni_psi} and \eqref{final:uni_chi}, using the Lipschitz condition on the nonlinear functions and the above estimate, we have
\begin{equation*}\|\tilde{\phi}(t)\|^2 + \|\tilde{\psi}(t)\|^2+ \|\tilde{\chi}(t)\|^2\leq C\left(g_{\alpha}*\left(\|\tilde{\phi}\|^2+\|\tilde{\psi}\|^2+\|\tilde{\chi}\|^2\right)\right)(t). \end{equation*}
Applying the generalized Gronwall-Bellman Lemma \ref{Lem_Gronwall_fractional}, we have $\|\tilde{\phi}(t)\|=\|\tilde{\psi}(t)\|=\|\tilde{\chi}(t)\|=0$, almost everywhere in $[0,T]$. Further, we have from \eqref{final:uni_u} using Korn's inequality \eqref{inequality:Poincare_Korn_Sobolev}, and from \eqref{final:uni_mu} that that $\|\tilde{\vec{u}}\|_{\H^1(\Omega)}=\|\tilde{\mu}\|=0$, almost everywhere in $[0,T]$. \qed
\section{Numerical discretisation}\label{sec:numerical_discretization}

The system~\eqref{System} can be written in the form
\begin{align}\label{eq:odes}
	\p_t^{\vec{\alpha}}(\vec{X}-\vec{X}_0) = \vec{F}(t,\vec{X}(t)),
	\qquad	
	\vec{X} = \left(\phi,\psi,\chi\right),
	\qquad
	\vec{\alpha} = \left(\alpha, 1, 1\right),
\end{align}
with 
\begin{equation}
  	\p_t^{\vec{\alpha}}(\vec{X}-\vec{X}_0) = \left[\begin{array}{c}	
  	\p_t^{\alpha}(\phi-\phi_0) \\
  	\p_t\psi\\
  	\p_t\chi \\
  	\end{array}\right],\quad \vec{F}(t,\vec{X}(t)) = \left[\begin{array}{l}	
		\nabla\cdot (M_\phi(\phi,\psi,\chi) \nabla \mu) + N_\phi f(\phi,\psi) - P_\phi g(\phi,\chi)\\
		\nabla\cdot (M_\psi(\phi,\psi,\chi) \nabla \psi) + S_\psi - N_\psi f(\phi,\psi)\\
		\nabla\cdot (M_\chi(\phi,\psi,\chi) \nabla \chi) - N_\chi \chi + S_\chi - P_\chi g(\phi,\chi)
	\end{array}\right]
\end{equation}
and $\mu=\mu(\phi)$ implicitly defined via \eqref{eqn:mueq}-\eqref{eqn:ueq}. In \eqref{eq:odes}, $\vec{X}_0=(\phi_0,\psi_0,\chi_0)$ is the initial condition.

In our simulations, time discretisation is performed using a first order quadrature scheme, of which we now recall the main features. 


For ease of presentation, we focus on a single scalar equation. Convolution quadrature schemes approximate the Riemann--Liouville derivative $\p_t^{\alpha}\varphi$  \eqref{Def:frac_Caputo} of some function $\varphi=\varphi(t)$ by a discrete convolution, see for instance \cite{Lubich86,Lubich88} for seminal papers and \cite{Zeng13,Jin16} for application to partial differential equations. Let $t_n=n T/N_t$, for $n\in\{0,1,\ldots,N_t\}$, be a subdivision of $[0,T]$ in $N_t$ equispaced time intervals of size $\Delta t := T/N_t$. Assuming $\varphi(0)=0$, we approximate the Riemann-Liouville time derivative of a function $\varphi$ by
\begin{equation}\label{eq:cq}
    \p_t^{\alpha}\varphi \approx \overline{\p}_t^{\alpha}\varphi := \frac{1}{(\Delta t)^{\alpha}}\sum_{j=0}^{N_t} b_j \varphi_{n-j},
\end{equation}
where $\varphi_{n-j}$ is the approximation to $\varphi(t_{n-j})$. The quadrature weights $(b_j)_{j\geq 1}$ are the coefficients in the power series expansion of $\omega^{\alpha}(\zeta)$, with $\omega(\zeta)=\frac{\sigma(1/ \zeta)}{\rho(1/ \zeta)}$ the generating function of a linear multistep method \cite{Lubich88}. When $\varphi(0)=\varphi_0\neq 0$, we have a discretisation to the Caputo derivative by applying \eqref{eq:cq} to $\varphi-\varphi_0$. From \eqref{eq:cq} we see that the memory effect of the fractional time derivative translates, numerically, to the fact that the value of the solution at some time step depends on its values at all previous time steps. In the backward differentiation formula of first order, also known as Gr\"unwald--Letnikov approximation \cite[Section 2.1.2]{Dumitru12}, the quadrature weights are defined recursively by
\begin{equation}\label{eq:cqweights}
    b_0=1,\quad b_j=-\frac{\alpha-j+1}{j}b_{j-1}\quad \text{for }j\geq 1.
\end{equation}
For $\alpha=1$, $b_j=0$ for $j\geq 2$, \eqref{eq:cq} coincides with the implicit Euler method. We refer to \cite[Section 4]{Jin16} for a detailed summary on the derivation of higher order convolution quadrature schemes together with their convergence properties.

Applying the scheme \eqref{eq:cq}--\eqref{eq:cqweights} to \eqref{eq:odes} and denoting by $\phi_n\approx \phi(t_n)$, $\psi_n\approx \psi(t_n)$ $\chi_n\approx \chi(t_n)$ the approximate solutions at time $t_n$, $n=1,\ldots,N_t$, we arrive at the following system of equations:
\begin{subequations}\label{DiscreteSystem}
\revision{\begin{align}
\sum_{j=0}^n b_j (\phi_{n-j}-\phi_{0})&=(\Delta t)^\alpha \nabla\cdot \left(M_\phi \nabla \mu_n\right) + (\Delta t)^\alpha N_\phi f(\phi_{n},\psi_{n}) - (\Delta t)^\alpha P_\phi g(\phi_{n},\chi_{n})\\
\psi_n-\psi_{n-1}&=(\Delta t)\nabla\cdot\left(M_\psi \nabla \psi_n\right) - (\Delta t)N_\psi f(\phi_{n},\psi_{n}) + (\Delta t)S_\psi\\
\chi_n-\chi_{n-1}&=(\Delta t)\nabla\cdot\left(M_\chi \nabla \chi_n\right) - (\Delta t) N_\chi \chi_n - (\Delta t) P_\chi g(\phi_{n},\chi_{n}) + S_\chi,
\end{align}}
where
\begin{align}
    \mu_n&=c\phi_n+\lambda \nabla\cdot \vec{u}_n,\\
    0&= \nabla \cdot \left( 2G \symgrad(\vec{u}_n) + \frac{2G\nu}{1-2\nu}\text{tr}(\symgrad(\vec{u}_n))\mathbb{I} + \lambda \phi_n\mathbb{I}\right).
\end{align}
\end{subequations}

We obtain an algebraic system using a Galerkin approach with linear finite elements. Namely, let $\mathcal{T}^h$ be a quasiuniform family of triangulations of $\Omega$, $h$ denoting the mesh width. For simplicity, we assume that $\cup_{T\in \mathcal{T}^h}\overline{T}=\overline{\Omega}$, which holds in all our numerical experiments. The piecewise linear finite element space is defined as
$$\mathcal{V}^h =\{\varphi\in \C(\overline{\Omega}): \varphi|_{T}\in P_1(T), \forall T\in \mathcal{T}^h\}\subset\H^1(\Omega),$$
where $P_1(T)$ denotes the set of all affine functions on $T$. As in the previous section, we assume that $\tilde{\psi_b}\equiv 0$ and $\tilde{\chi_b}\equiv 0$. We formulate the discrete problem as follows: at the $n$-th time step, find 
$$\phi_n,\mu_n \in \mathcal{V}^h ,\,\,  \vec{u}_n\in \left(\mathcal{V}^h\cap \H^1_{0,\Diru}(\Omega)\right)^d,\,\, \psi_n,\chi_n\in \mathcal{V}^h\cap \H^1_{0,\Dirpsi}(\Omega)$$
such that, for all test functions $\varphi_1,\varphi_4\in \mathcal{V}^h$, $\varphi_2,\varphi_3 \in \mathcal{V}^h\cap \H^1_{0,\Dirpsi}(\Omega)$ and $\vec{\varphi}_5\in \left(\mathcal{V}^h\cap \H^1_{0,\Diru}(\Omega)\right)^d$,
\revision{\begin{subequations}
\begin{align}
(\psi_n,\varphi_2)+ (\Delta t)(M_\psi\nabla \psi_n,\nabla \varphi_2)= &(\psi_{n-1},\varphi_2) - (\Delta t) N_\psi (f(\phi_{n},\psi_{n}),\varphi_2)\nonumber \\ &+ (\Delta t) (S_\psi,\varphi_2) + (\Delta t)(\psi_b,\varphi_2)_{\p\Omega\backslash\Dirpsi}\label{DiscreteSystempsi}\\
(\chi_n,\varphi_3)+ (\Delta t)(M_\chi\nabla \chi_n,\nabla \varphi_3) + (\Delta t) N_\chi (\chi_n,\varphi_3)= &(\chi_{n-1},\varphi_3) - (\Delta t)P_\chi (g(\phi_{n},\chi_{n}),\varphi_3) + (\Delta t)(S_\chi,\varphi_3) \nonumber \\ &+ (\Delta t)(\chi_b,\varphi_2)_{\p\Omega\backslash\Dirpsi}\label{DiscreteSystemchi}\\
b_0(\phi_{n},\varphi_1)+(\Delta t)^\alpha (M_\phi \nabla \mu_n,\nabla \varphi_1)= &- \sum_{j=1}^{n-1} b_j ((\phi_{n-j}-\phi_{0}),\varphi_1) + (\Delta t)^\alpha (b_0\phi_0,\varphi_1)\nonumber\\
& \hspace{-1.3cm} + (\Delta t)^\alpha N_\phi (f(\phi_{n},\psi_{n}),\varphi_1) - (\Delta t)^\alpha P_\phi (g(\phi_{n},\chi_{n}),\varphi_1)\label{DiscreteSystemphi}\\
(\mu_n,\varphi_4)-c(\phi_n,\varphi_4)-\lambda (\nabla\cdot \vec{u}_n,\varphi_4)=& \,0\label{DiscreteSystemmu}\\
2G (\symgrad(\vec{u}_n),\symgrad(\vec{\varphi}_5)) + \frac{2G\nu}{1-2\nu}(\nabla\cdot\vec{u}_n,\nabla\cdot \vec{\varphi}_5) = &- \lambda (\phi_n,\nabla\cdot \vec{\varphi}_5)\label{DiscreteSystemu}
\end{align}
\end{subequations}}
(here we have rearranged the order of the equations for easier explanation in the upcoming remarks). This is a non-linear, coupled algebraic system with unknowns $\phi_n,\mu_n,\vec{u}_n,\psi_n,\chi_n$. At each time step, we solve this system with a fixed point iteration. In our experiments, we set the termination criteria for the latter to be a maximum of $50$ iterations and a maximum tolerance for the relative error between two iterates of TOL=$10^{-6}$. For all experiments in the next section, this tolerance was always reached within less than $10$ fixed point iterations.

The procedure described in this section has been implemented in FEniCS \cite{fenics}, using version 2019.1.0. of the DOLFIN library \cite{dolfin}, to obtain the numerical results shown in the next section.

\begin{remark}
Equations \eqref{DiscreteSystempsi}-\eqref{DiscreteSystemchi} are decoupled from \eqref{DiscreteSystemmu}-\eqref{DiscreteSystemu} and they are coupled with \eqref{DiscreteSystemphi} just through the non-linear terms on the right hand side involving the functions $f$ and $g$, respectively. This can be exploited to improve efficiency when solving the non-linear system: in the spirit of a Gauss-Seidel method, at each fixed point iteration one can first update $\psi_n$ and $\chi_n$ with \eqref{DiscreteSystempsi}-\eqref{DiscreteSystemchi} using the previous iterate of $\phi_n$ on the right hand side, and then insert these updated values in the right hand side of \eqref{DiscreteSystemphi} and solve the last three coupled equations to update $\phi_n,\mu_n$ and $\vec{u}_n$.
\end{remark}

\begin{remark}
When using first order finite elements, further efficiency can be gained by using mass lumping \cite[Sect.11.4]{QuarteroniValli} in \eqref{DiscreteSystemmu}. More precisely, if in the first two terms in \eqref{DiscreteSystemmu} we approximate the mass matrix by its lumped version (which is diagonal), the latter can be inverted cheaply and we can express explicitly the coefficients vector of $\mu_n$ (with respect to the finite element basis functions) in terms of the coefficient vectors for $\phi_n$ and $\vec{u}_n$. Using the resulting expression for the coefficient vector of $\mu_n$ in  \eqref{DiscreteSystemphi} allows to eliminate \eqref{DiscreteSystemmu} from the system of equations and thus to reduce the system's size.
\end{remark}


\section{Numerical simulations}\label{sec:numerical_simulation}
\par In this section, we present the numerical approximations of the variables $\phi,\mu,\vec{u},\psi,\chi$ in the model (\ref{System}) with a two-dimensional domain $\Omega=(0,1)^2$.
In Section \ref{ssec:numexp_fractional}, we study the effects of introducing the fractional time derivative in the reaction--diffusion model, neglecting mechanical effects and in absence of treatment. We next introduce, in Section \ref{ssec:numexp_notreat}, the coupling with mechanical forces, still in absence of treatment. Finally, in Section \ref{ssec:numexp_treat}, we consider the effect of \revision{chemotherapeutic agents}, including a periodic source for the treatment. 

\par Where not otherwise stated, we choose the parameters to have the dimensionless values listed in Table \ref{eq:param_val}. We have $M_{\phi}\ll M_\psi,M_\chi$ because the tumour diffuses at a much lower speed compared to how fast the nutrient and chemotherapeutic agents diffuse. We set $N_\phi\ll N_\psi$ because the nutrient rate of decrease is much faster than the proliferation rate of the tumour due to the nutrient consumption (usually a tumour can at most double its size in one day) and, for similar reasons, we have $P_\phi\ll P_\chi$. The half maximum value $K_{\psi}$ has been chosen to be lower than the maximum concentration reached by the nutrient, in order to observe both effects of low-density limited and maximum capacity-limited growth of the tumour. Accordingly, $K_{\chi}$ has been chosen to be lower than the maximum concentration reached by the chemotherapeutic agents. The degradation rate $N_\chi$ is large compared to the other coefficients in the reaction terms of the equation for the chemotherapeutic substances because the latter degrade quite fast (their half-time is usually a couple of hours). The parameters $c,\lambda,G$ and $\nu<0.5$ have been selected to be in the ranges considered in \cite{lima2016selection}.
 
 \begin{table}[H]
 \setlength{\tabcolsep}{5pt}
 \small
     \centering
     \begin{tabular}{|l|c|c|c|c|c|c|c|c|c|c|c|c|c|c|}
     \hline
     \multirow{2}{3.5em}{Equation} & \multicolumn{14}{c|}{Parameter values}\\
     \cline{2-15}
         &  $M_\phi$ & $N_\phi$ & $K_\psi$ & $P_\phi$ & $c$ & $\lambda$ & $G$ & $\nu$ & $M_\psi$ & $N_\psi$ & $M_\chi$ & $N_\chi$ & $P_{\chi}$ & $K_{\chi}$\\
          \hline
      \eqref{eqn:phieq}     &  0.0001 & 0.6 & 2 & 1.1 & &&&&&&&&&\\
       \eqref{eqn:mueq}    & &&&& 0.4 & 0.002 &&&&&&&&\\
      \eqref{eqn:ueq}     & &&&&  & & 0.4615 & 0.3 &&&&&&\\
     \eqref{eqn:psieq}      & &&&&  & & & & 1 & 40 &&&&\\
      \eqref{eqn:chieq}     & &&&&  & & & &  &  & 1 & 3 & 30 & 0.6\\
          \hline
     \end{tabular}
     \caption{Parameter values used in the simulations of Section \ref{sec:numerical_simulation}, unless otherwise stated.}\label{eq:param_val}
 \end{table}

In the simulations, we have set $\Delta t=1/15$ for the time stepping, and we have discretised the domain $\Omega$ using a regular, triangular mesh with mesh size $h=\sqrt{2}/150$ for all simulations but those in Figures \ref{Fig_RD_radius} and \ref{Fig_RD_interface}, where we have used a finer mesh with $h=\sqrt{2}/200$ to have a more precise computation of the radius of the tumour.

We refer to the quantities $\int_{\Omega}\phi(\vec{x},t)\textrm{d}\vec{x}$, $\int_{\Omega}\psi(\vec{x},t)\textrm{d}\vec{x}$ and $\int_{\Omega}\chi(\vec{x},t)\textrm{d}\vec{x}$ as the tumour, nutrient and chemotherapy mass, respectively. Since $\phi$ is a volume fraction, the mass of the tumour is technically given by $\int_{\Omega}\rho\phi(\vec{x},t)\textrm{d}\vec{x}$, but, since we assume $\rho$ to be constant, $\int_{\Omega}\phi(\vec{x},t)\textrm{d}\vec{x}$ is the mass up to rescaling.


\subsection{Reaction-diffusion system without treatment}\label{ssec:numexp_fractional}
\par The goal of this subsection is to show the basic effects of introducing a fractional time derivative and the new modelling possibilities that it offers. 
For this, we consider the reaction--diffusion model \eqref{eqn:phieq}, \eqref{eqn:mueq} and \eqref{eqn:psieq} with the parameters $\lambda=0, P_\phi=0$, in which case we have two variables $\phi$ and $\psi$. For both $\phi$ and $\psi$ we impose homogeneous Neumann boundary conditions, and a constant nutrient is supplied over the whole domain by setting $S_{\psi}\equiv$0.5. 

We first consider the evolution of a circular tumour and then of a tumour concentration having, initially, two disconnected components.

For the circular tumour, we consider the initial conditions
\begin{equation}\label{eq:ic_circularplateau}
    \phi_0(\vec{x}) = \begin{cases}
    1 & \mbox{if}\quad  \lVert \vec{x}-\vec{c}\rVert \leq b,\\
    \text{exp}\left(1- \frac{(a-b)^2}{(a-b)^2- ( \lVert \vec{x}-\vec{c}\rVert-b)^2}\right)& \mbox{if}\quad b<  \lVert \vec{x}-\vec{c}\rVert\leq a,\\
    0&\mbox{otherwise},
    \end{cases}
\end{equation}
and $\psi_0\equiv 0$. We first set $\vec{c}=(0.5,0.5)$, $a=0.22$ and $b=0.05$. A cross section of \eqref{eq:ic_circularplateau} along the $x$-axis is depicted in the left plot of Figure \ref{Fig_RD_interface}, dashed line.

Experiments with \textsl{in vitro} cell cultures as well as \textsl{in vivo} data \cite{jiang2014anomalous} have shown that the growth of the tumour size over time can have different power law exponents depending on the type of cells and the surrounding environment. For this reason, we have tracked the radius of the tumour over time for different values of $\alpha$, whose results are shown in Figure \ref{Fig_RD_radius}, left plot. Here we have defined the radius of the tumour as 
\begin{equation*}
    R(t) = \text{argmax}_{\lVert \vec{x}-\vec{c}\rVert, \vec{x}\in\Omega}\; \left\{\phi(\vec{x},t)\geq R_{thresh}\right\},
\end{equation*}
with $R_{thresh}$ a threshold value that we have set to $10^{-3}$. The left plot in Figure \ref{Fig_RD_radius} clearly shows that, by varying $\alpha$, the radius grows with different power laws. This means that, if one is only interested in predicting the tumour size at a certain time, then a reaction-diffusion model with properly tuned diffusion and reaction coefficients would be sufficient. However, if one is interested in the dynamics,  then the model with fractional exponent can express behaviours that cannot be modelled with an integer order model with time-constant coefficients. To show this, we have taken $\alpha=1$ and tuned $M_{\phi}$ and $N_{\phi}$ in order to obtain the same radius at $t=30$ as the one obtained with $\alpha=0.5$ and parameters as in Table \ref{eq:param_val}. The results are shown in the right plot in Figure \ref{Fig_RD_radius}, where, for $\alpha=1$, we have taken $M_\phi=0.0001/1.803$ and $N_\phi=0.6/1.803$. We see that, with the new values for $M_\phi$ and $N_\phi$, the model with integer order time derivative can predict the same radius for the tumour at $t=30$ as $\alpha=0.5$ and the parameters in Table \ref{eq:param_val}. However, the dynamics is quite different in the two cases.

\begin{figure}[!htb]\centering
	\includegraphics[width=0.45\textwidth]{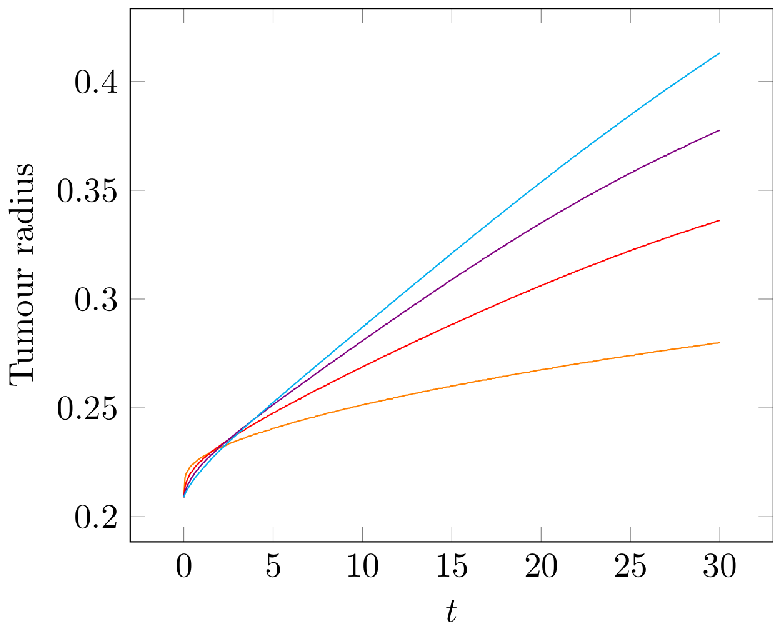}\qquad
	\includegraphics[width=0.45\textwidth]{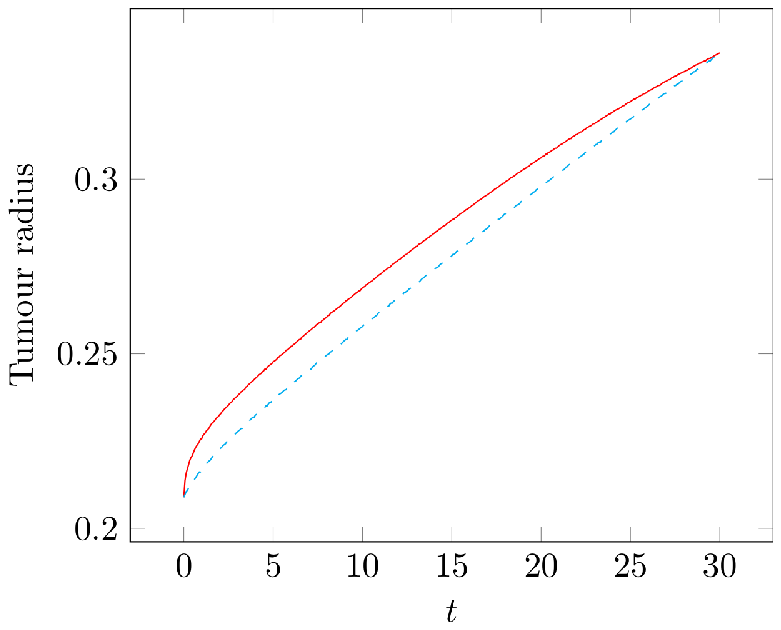}
	\caption[Tumour evolution when starting with circular initial condition]{\revision{Reaction-diffusion model and circular initial condition \eqref{eq:ic_circularplateau} with $a=0.22$ and $b=0.05$. Left: evolution of radius of the tumour over time with parameters as in Table \ref{eq:param_val} and $\alpha=0.25$(\textcolor{orange}{-----}), $\alpha=0.5$ (\textcolor{red}{-----}),  $\alpha=0.75$ (\textcolor{violet}{-----}) and $\alpha=1$ (\textcolor{cyan}{-----}).  Right: evolution of the radius over time for parameters as in Table \ref{eq:param_val} and $\alpha=0.5$ (\textcolor{red}{-----}) - same line as in the left plot - and $M_\phi=0.0001/1.803$, $N_\phi=0.6/1.803$ and $\alpha=1$ (\textcolor{cyan}{-\hspace{0.05cm}-\hspace{0.05cm}-})}}
	\label{Fig_RD_radius}
\end{figure}

The left plot in Figure \ref{Fig_RD_interface} shows a cross section along the $x$ axis of the density of the tumour at final time $T=30$. For reference, the dashed line is the initial condition. We observe that the spreading of the tumour is very sensitive to $\alpha$, and when $\alpha$ is large, the spreading is faster. When $\alpha$ is smaller, not only the tumour grows more slowly, but the interface between the tumour and the surrounding tissue is less sharp. The initial condition \eqref{eq:ic_circularplateau} has a plateau around the centre of the domain where the tumour density is $1$. A question that can arise is what happens when we start with an initial condition which has no plateau. To test this, we have run a second experiment with $b=0$ and $a=0.2$. The cross section of the tumour density at $T=30$ is shown in the right plot of Figure \ref{Fig_RD_interface}, where again the dashed line refers to the initial condition. Here we can observe that, for $\alpha<1$, the tumour grows over time without forming a plateau in the centre, which happens for $\alpha=1$. Such different behaviour is not surprising if one notices that varying the fractional exponent is not a simple re-scaling of the time variable and it introduces instead different nonlinear behaviours in the model.

\begin{figure}[!htb]\centering
\includegraphics[width=0.45\textwidth]{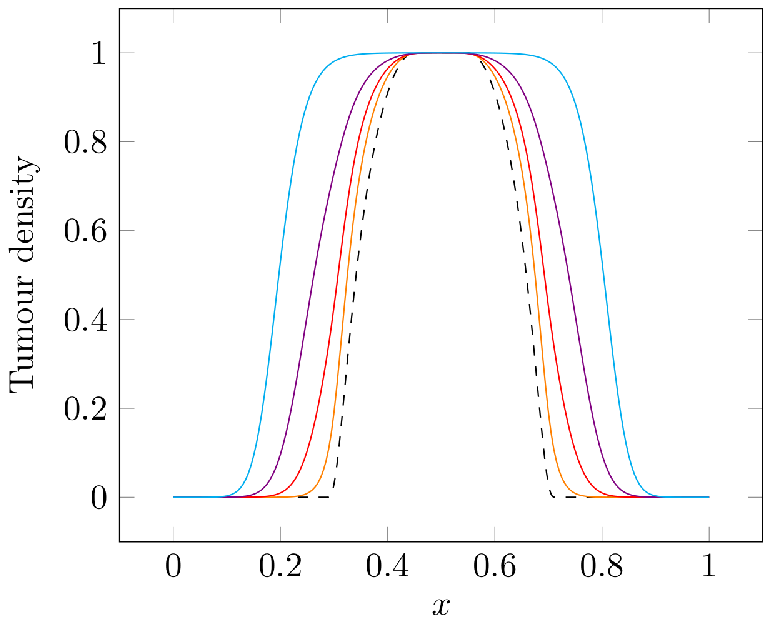}\qquad
\includegraphics[width=0.45\textwidth]{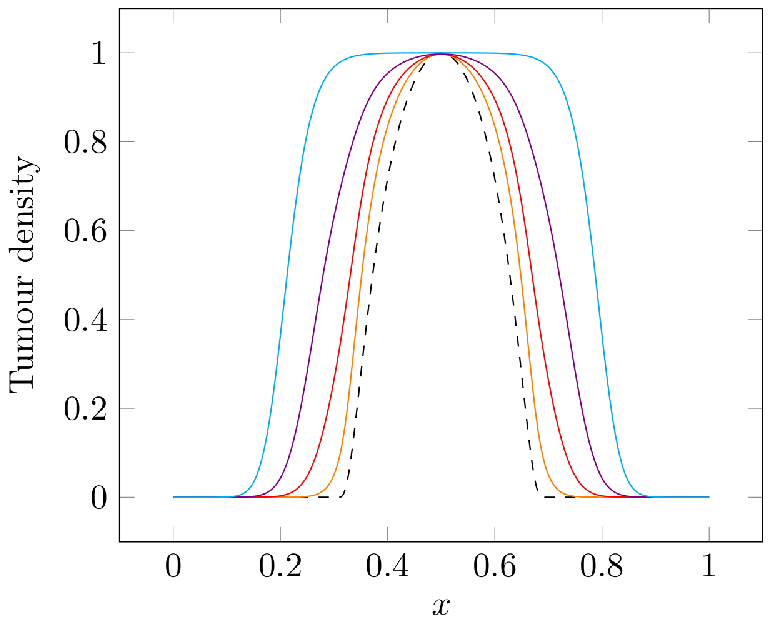}
\caption[Tumour evolution when starting with circular initial condition]{\revision{Reaction-diffusion model and circular initial condition. Cross section of the tumour density along the $x$-axis at $T=30$ for $\alpha=0.25$ (\textcolor{orange}{-----}), $\alpha=0.5$ (\textcolor{red}{-----}), $\alpha=0.75$ (\textcolor{violet}{-----}) and $\alpha=1$ (\textcolor{cyan}{-----}). Left: results for initial condition \eqref{eq:ic_circularplateau} with $a=0.22$ and $b=0.05$. Right: results for initial condition \eqref{eq:ic_circularplateau} with $a=0.2$ and $b=0$. In both plots, the dashed line depicts the initial condition.}}
\label{Fig_RD_interface}
\end{figure}


\par We now consider an initial condition with two disconnected components, namely, two initially separated elliptical tumour masses:
\revision{\begin{equation}\label{eq:ic_twoellipses}
\phi_0(\vec{x}) = \begin{cases}
\text{exp}\left(1- \frac{a^2}{a^2-\lVert A(\vec{x}-\vec{c}_1)\rVert^2}\right) & \mbox{if}\quad \lVert A(\vec{x}-\vec{c}_1)\rVert \leq a, \\[0.05em]
\text{exp}\left(1- \frac{a^2}{a^2-\lVert A(\vec{x}-\vec{c}_2)\rVert^2}\right) & \mbox{if}\quad \lVert A(\vec{x}-\vec{c}_2)\rVert \leq a, \\
0 & \mbox{otherwise},
\end{cases}
\end{equation}}
with $A=\begin{pmatrix} 1&0\\ 0 & \gamma \end{pmatrix}$, $\gamma=\sqrt{5}$, $\vec{c}_1=(0.5,0.6)$, $\vec{c}_2=(0.5,0.4)$ and $a=0.2$. This initial condition is depicted in the first column of Figure \ref{Fig_RD_coalescence}. As before, we take $\psi_0\equiv 0$. The evolutions of the tumour for four values of $\alpha$ are depicted in Figure  \ref{Fig_RD_coalescence}. There, we can see that different values of the fractional exponent affect the coalescence speed of the two ellipses: the smaller the $\alpha$, the lower the speed at which they merge. Moreover, as in the circular tumour case, we see that, the larger the $\alpha$, the sharper the interface between the tumour and the healthy tissue (compare for instance $\alpha=1$ at $t=9$ with $\alpha=0.75$ at $t=15$ and $t=20$). Similar observations were made in \cite{Liu18} in a fractional phase-field model for porous media applications. 

\begin{figure}[!htb]
	\phantom{x}\hspace{-0.4cm}\begin{tabular}{ccccc}
		& \phantom{x}\hspace{0.4cm}$t=0$ & \phantom{x}\hspace{0.4cm}$t=9$ & \phantom{x}\hspace{0.4cm}$t=15$ & \phantom{x}\hspace{0.4cm}$t=20$ \\[0.8em]
		$\alpha=0.25$& \parbox[c]{5em}{\includegraphics[scale=0.08]{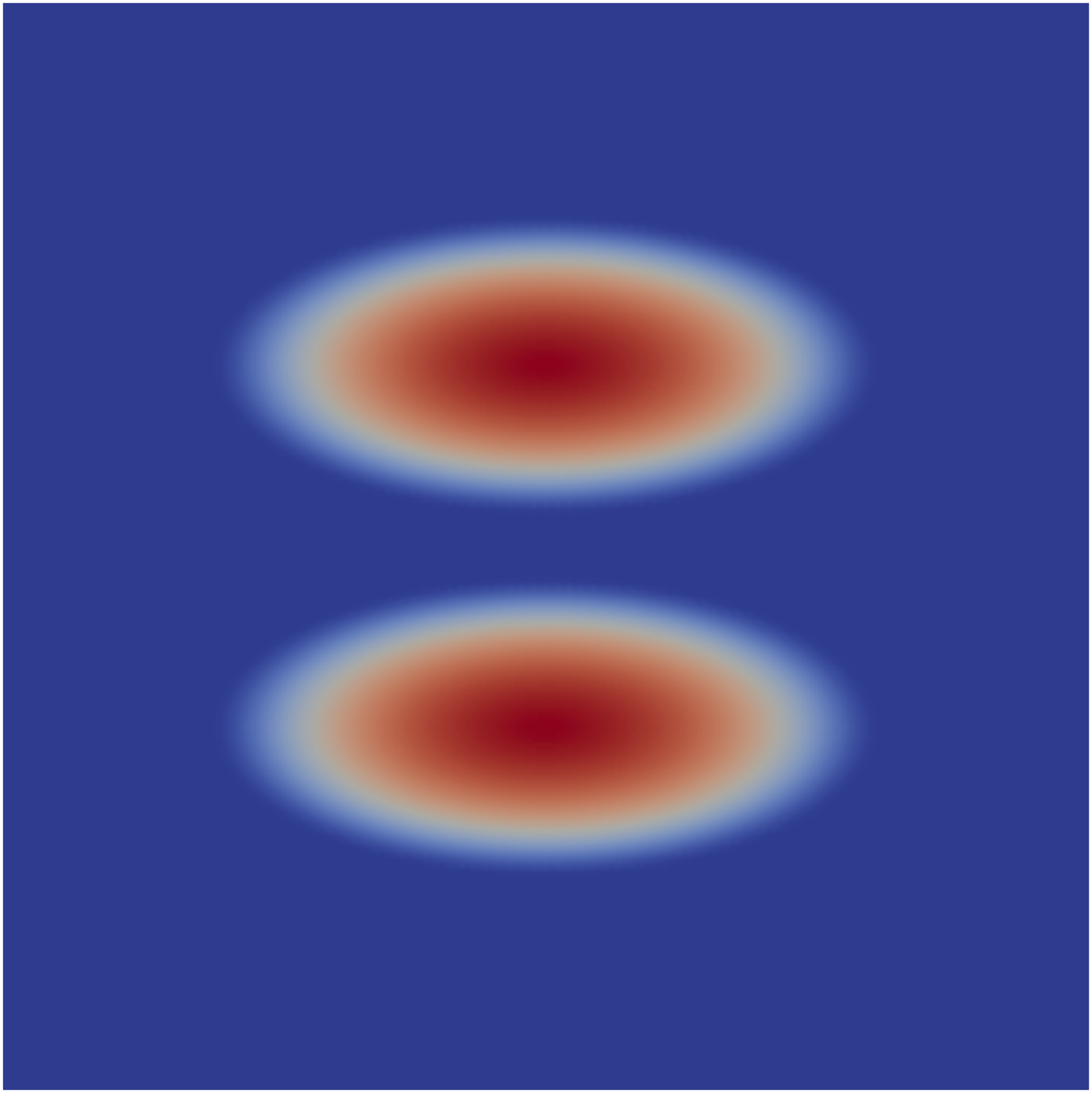}}
		& \parbox[c]{5em}{\includegraphics[scale=0.08]{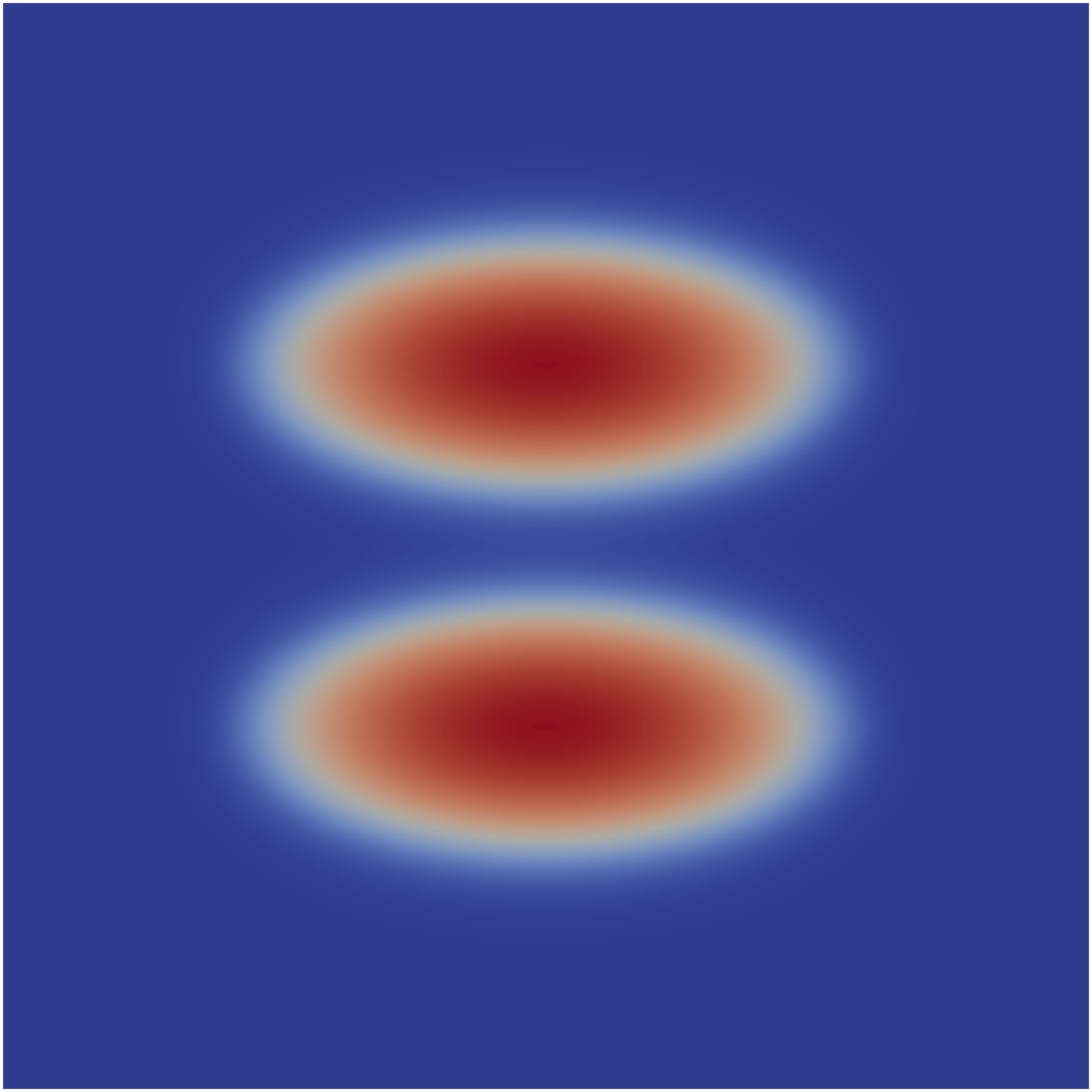}}
		& \parbox[c]{5em}{\includegraphics[scale=0.08]{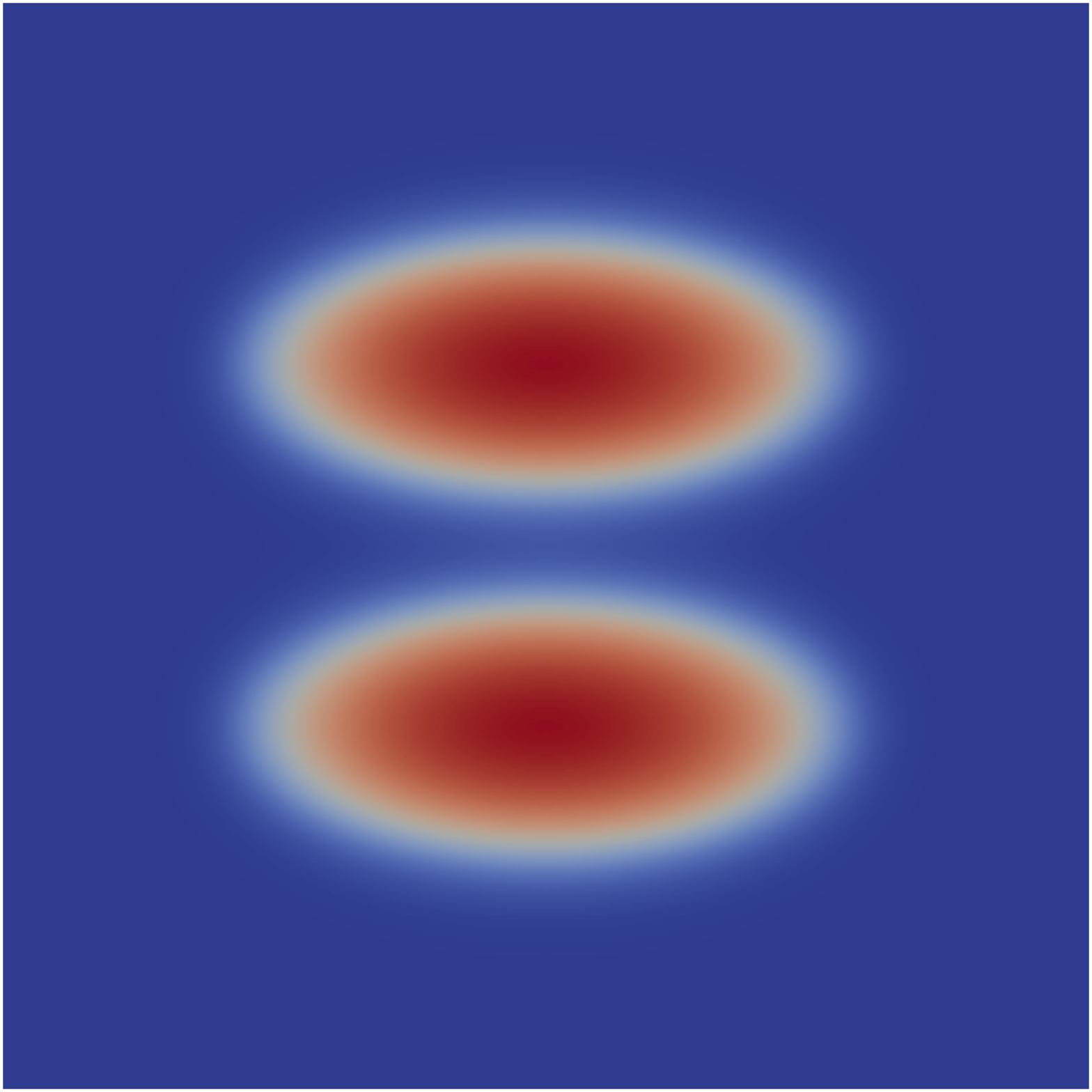}}
		& \parbox[c]{5em}{\includegraphics[scale=0.08]{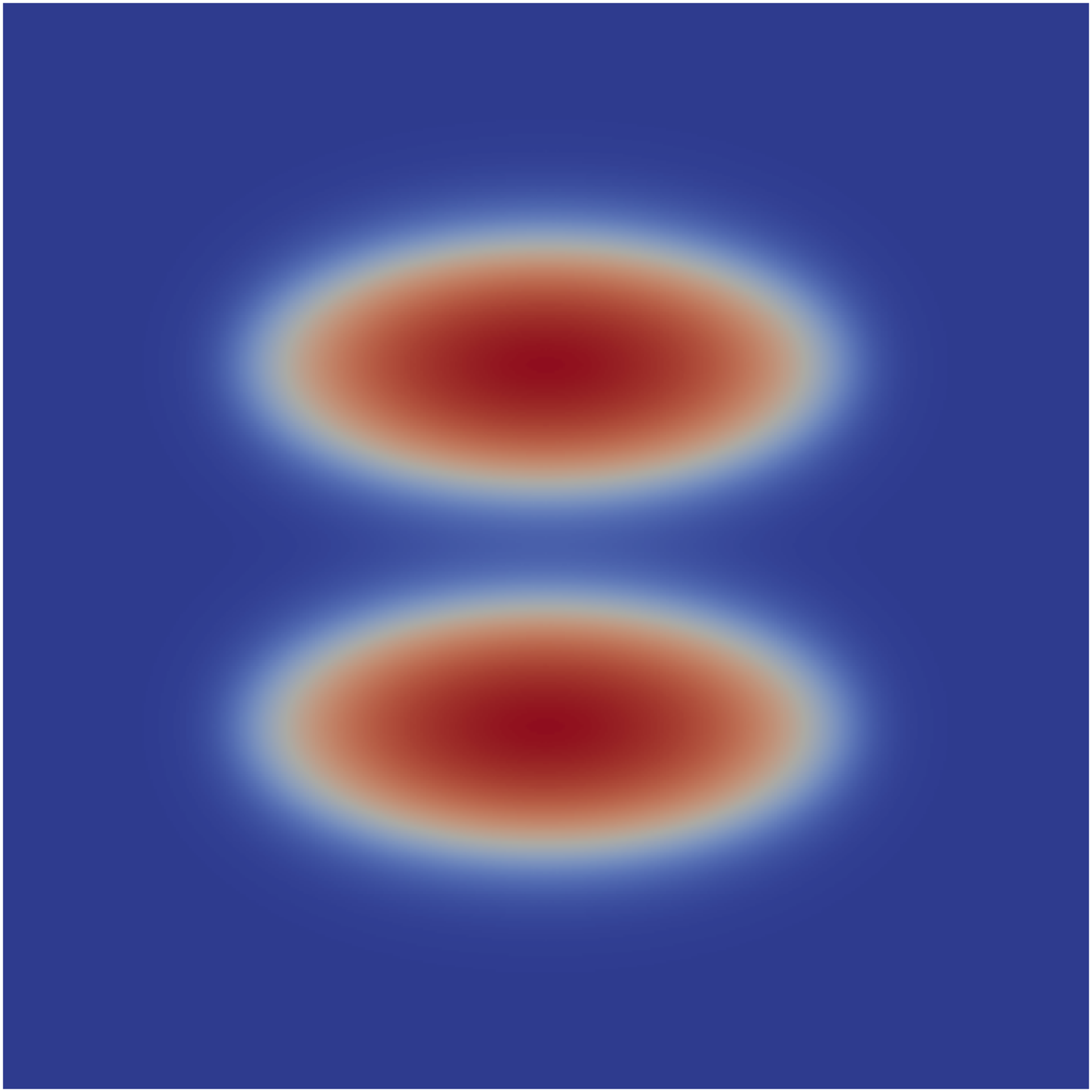}}\\[4em]
		$\alpha=0.5$ & \parbox[c]{5em}{\includegraphics[scale=0.08]{initialcondition_zoomed}}
		& \parbox[c]{5em}{\includegraphics[scale=0.08]{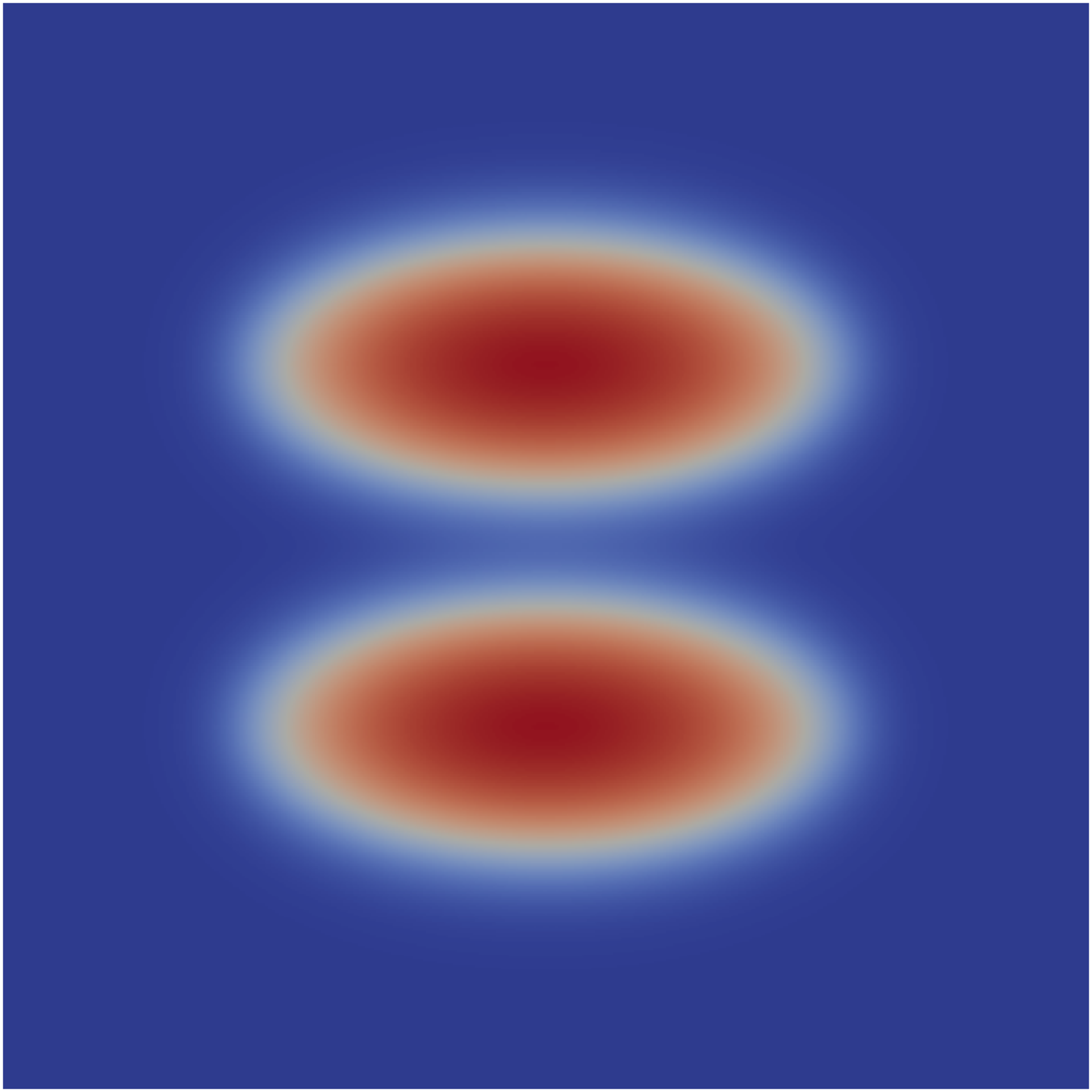}}
		& \parbox[c]{5em}{\includegraphics[scale=0.08]{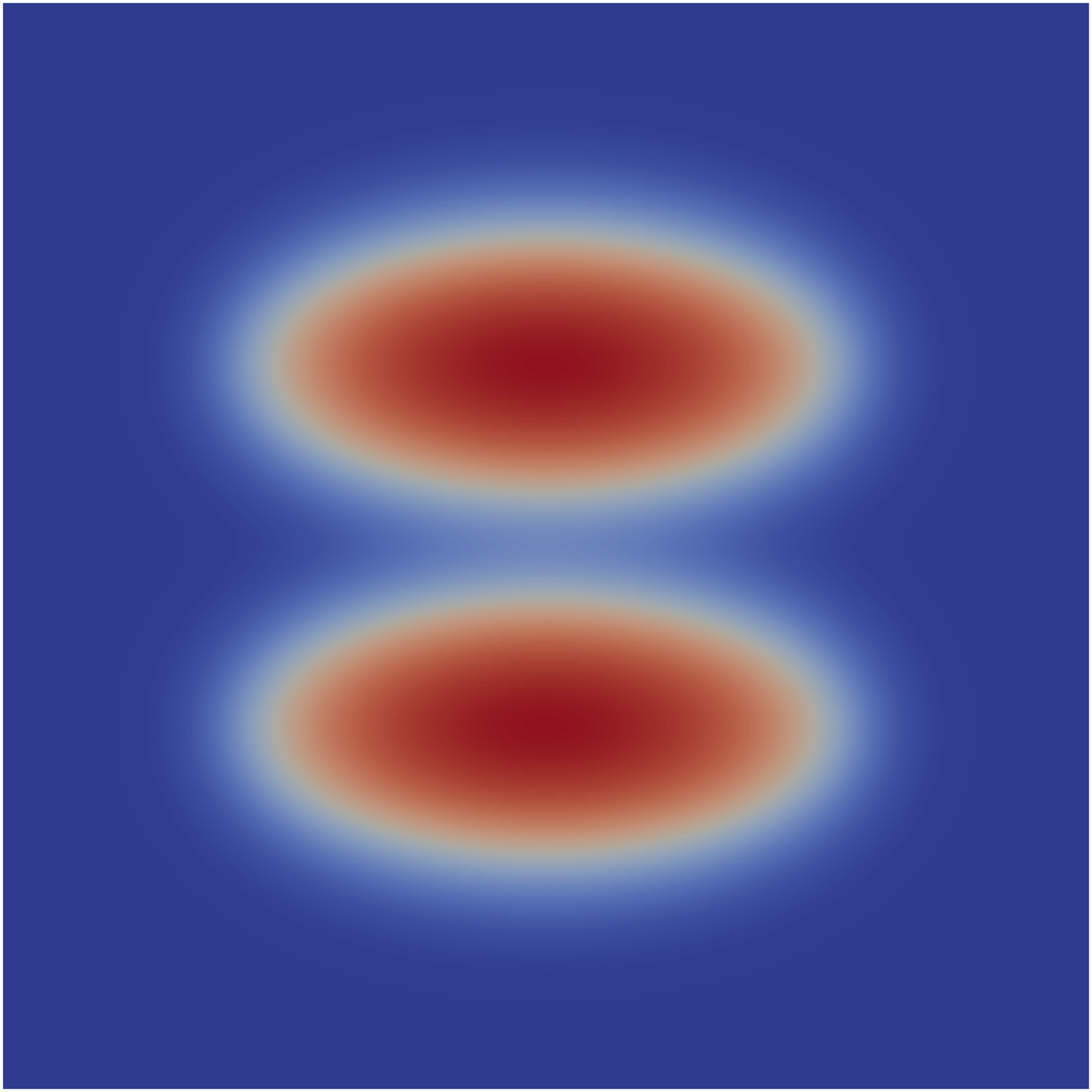}}
		& \parbox[c]{5em}{\includegraphics[scale=0.08]{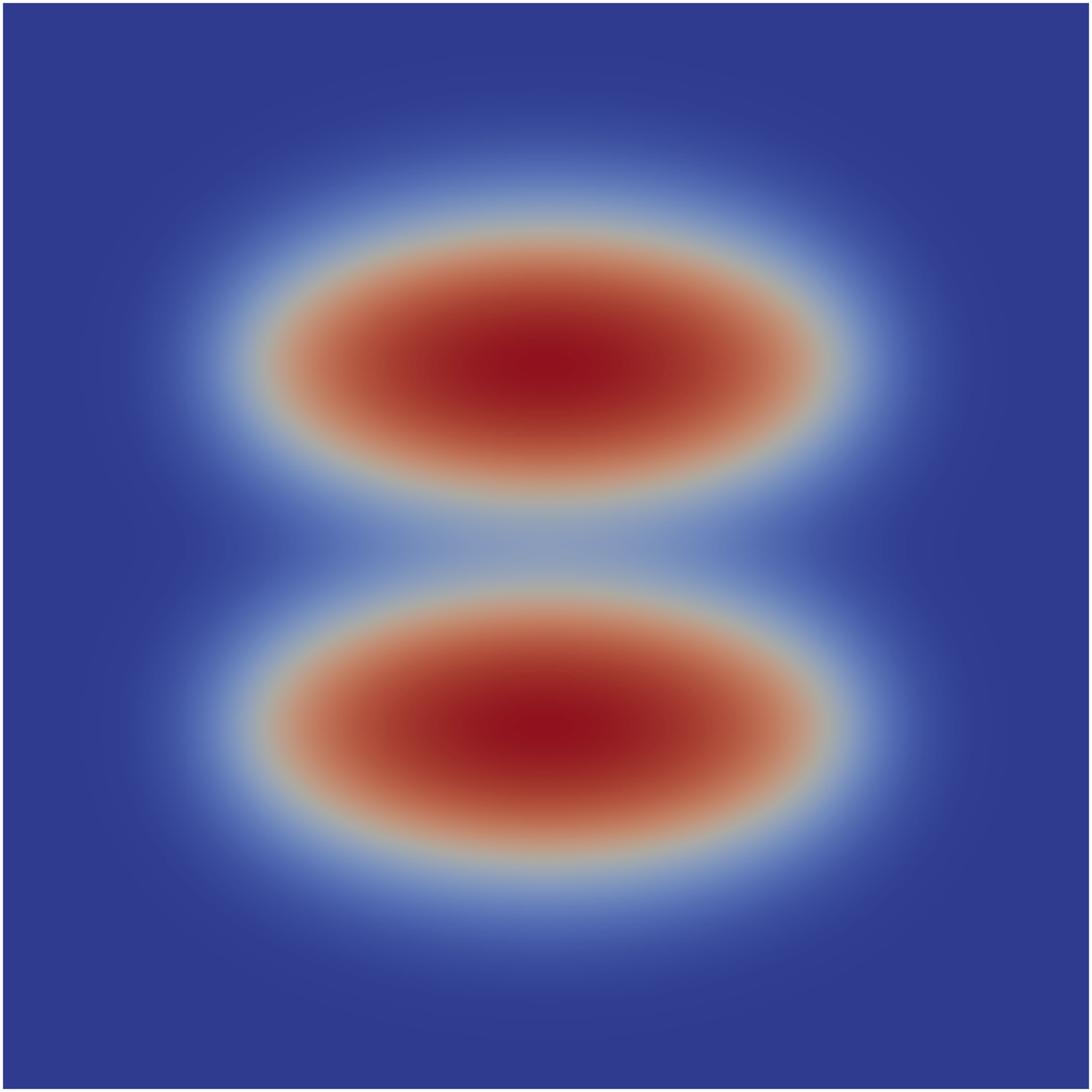}}\\[4em]
		$\alpha=0.75$ & \parbox[c]{5em}{\includegraphics[scale=0.08]{initialcondition_zoomed}}
		& \parbox[c]{5em}{\includegraphics[scale=0.08]{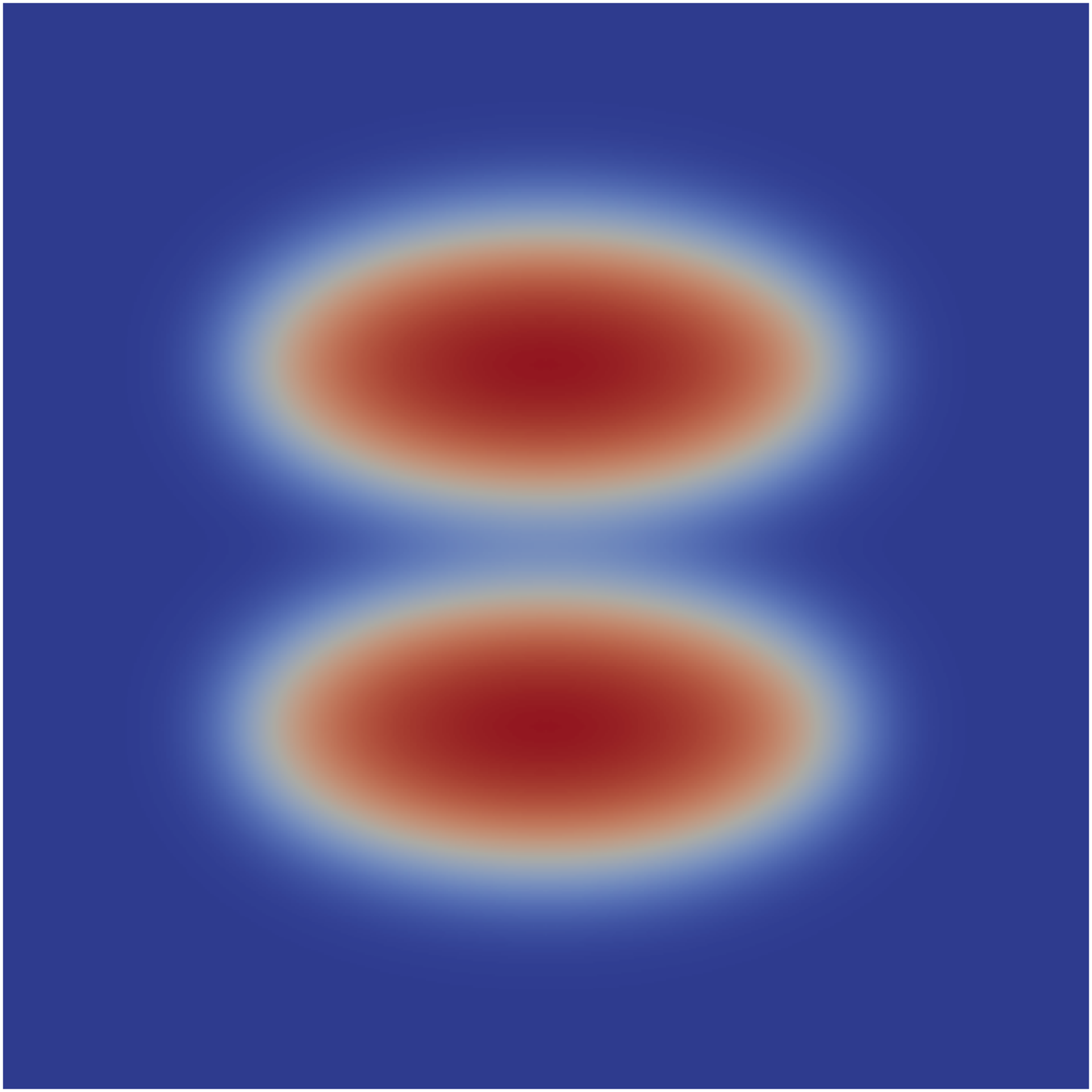}}
		& \parbox[c]{5em}{\includegraphics[scale=0.08]{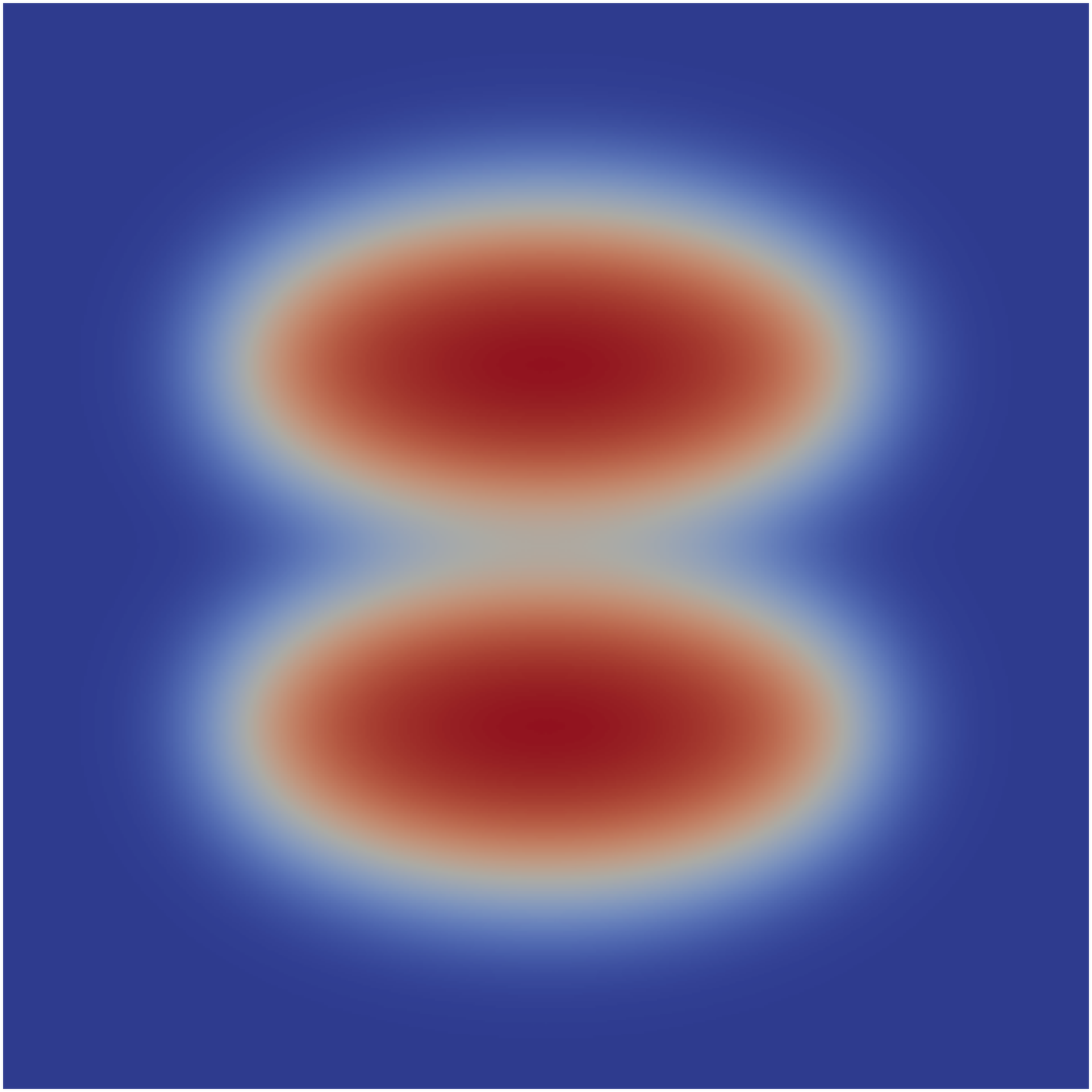}}
		& \parbox[c]{5em}{\includegraphics[scale=0.08]{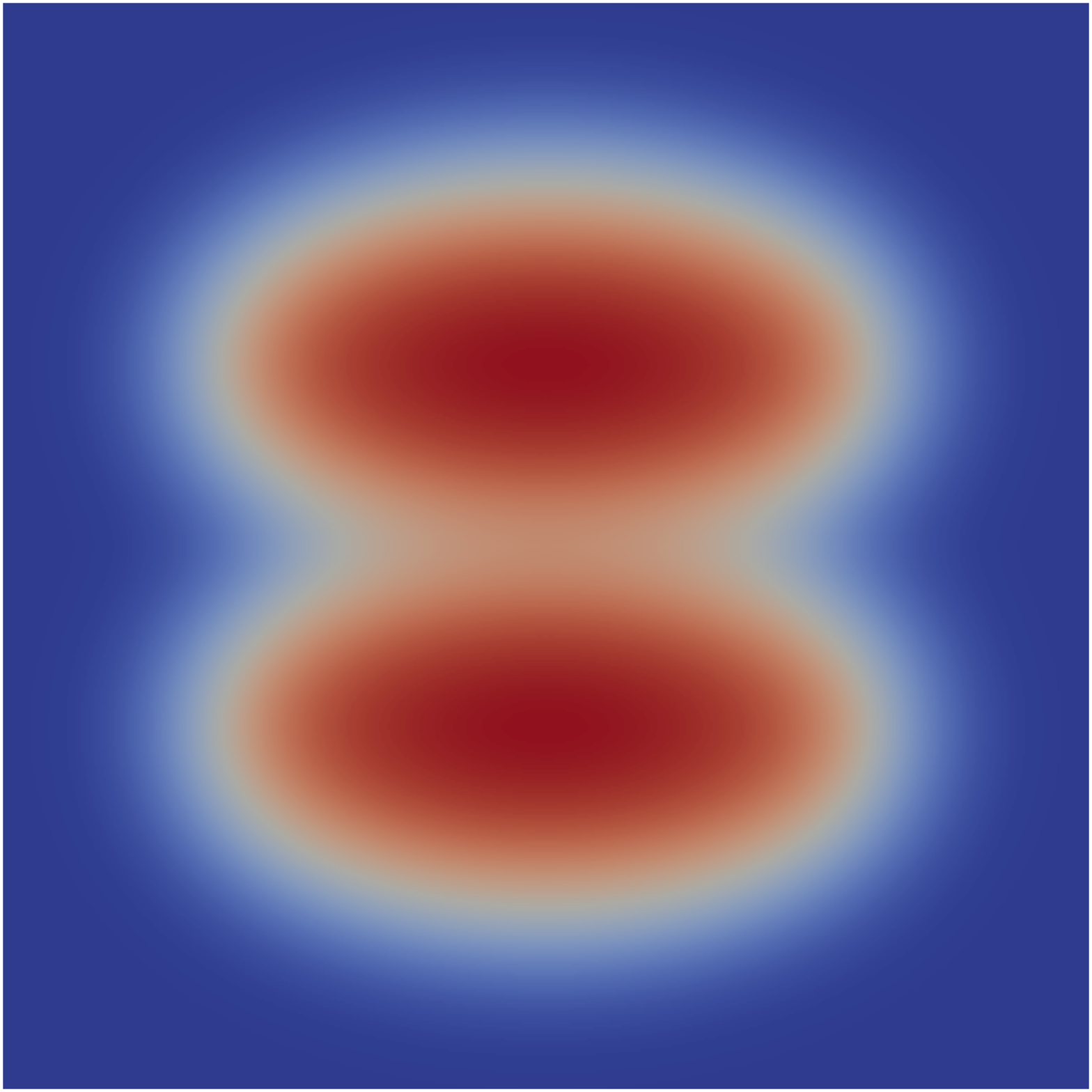}}\\[4em]
		$\alpha=1$ & \parbox[c]{5em}{\includegraphics[scale=0.08]{initialcondition_zoomed}}
		& \parbox[c]{5em}{\includegraphics[scale=0.08]{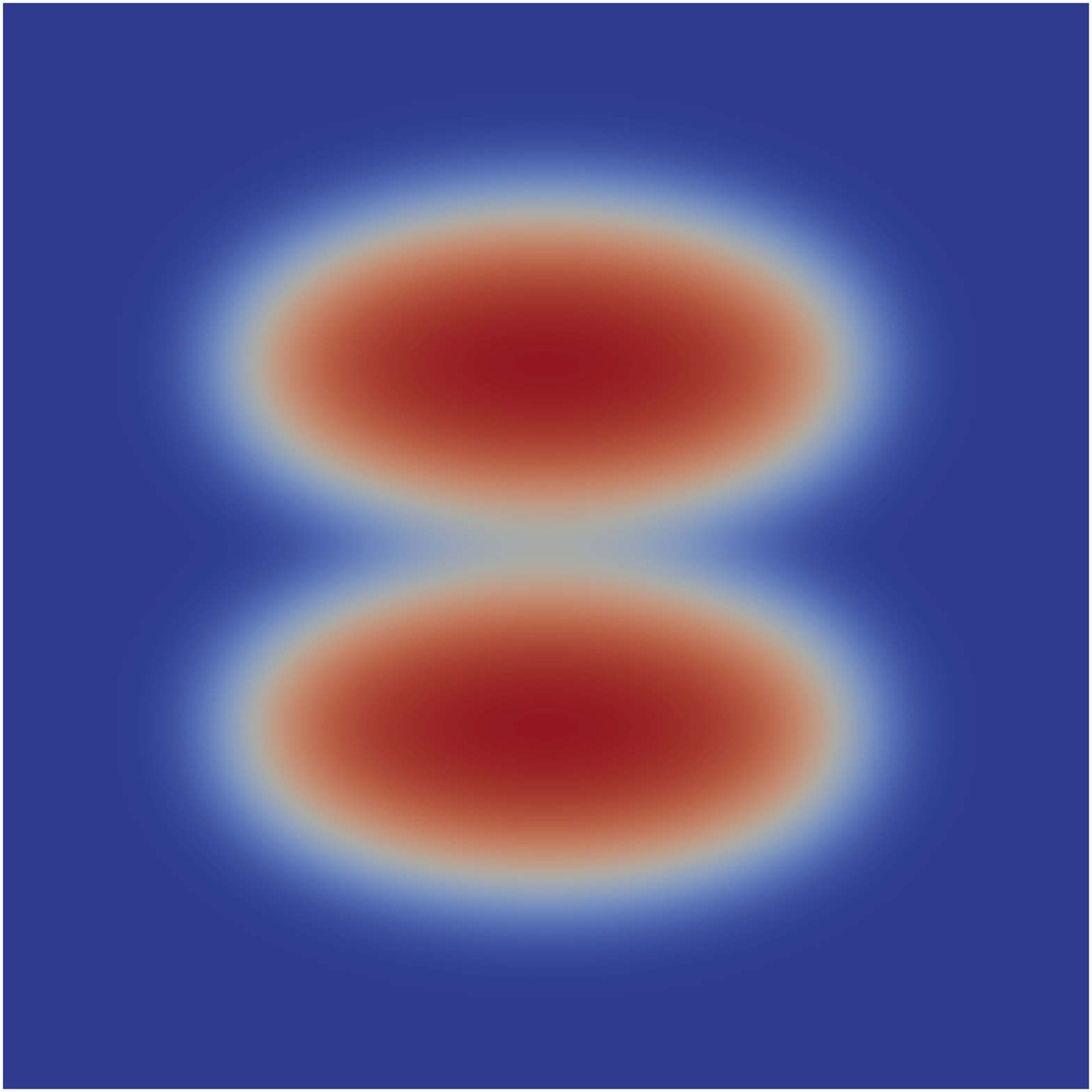}}
		& \parbox[c]{5em}{\includegraphics[scale=0.08]{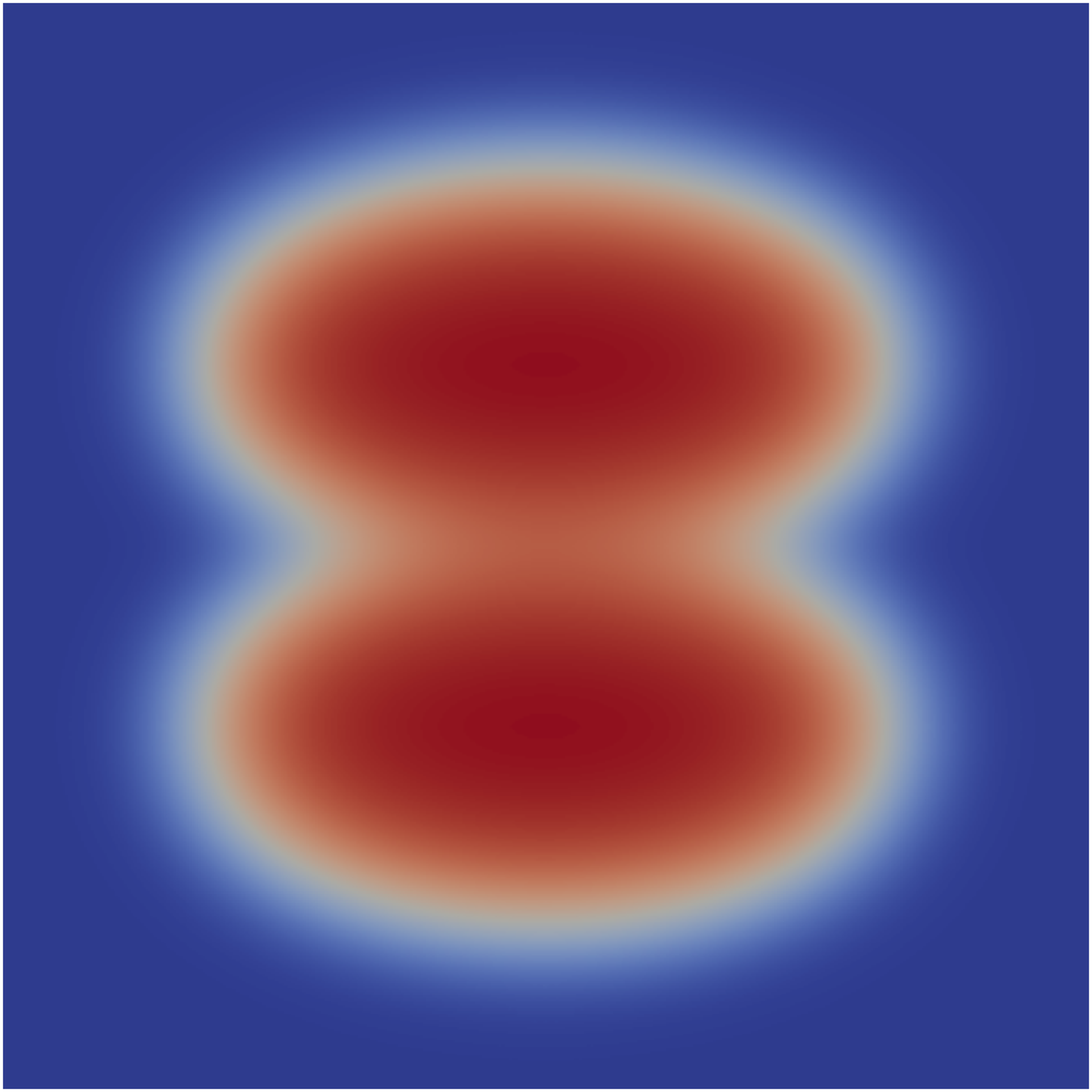}}
		& \parbox[c]{5em}{\includegraphics[scale=0.08]{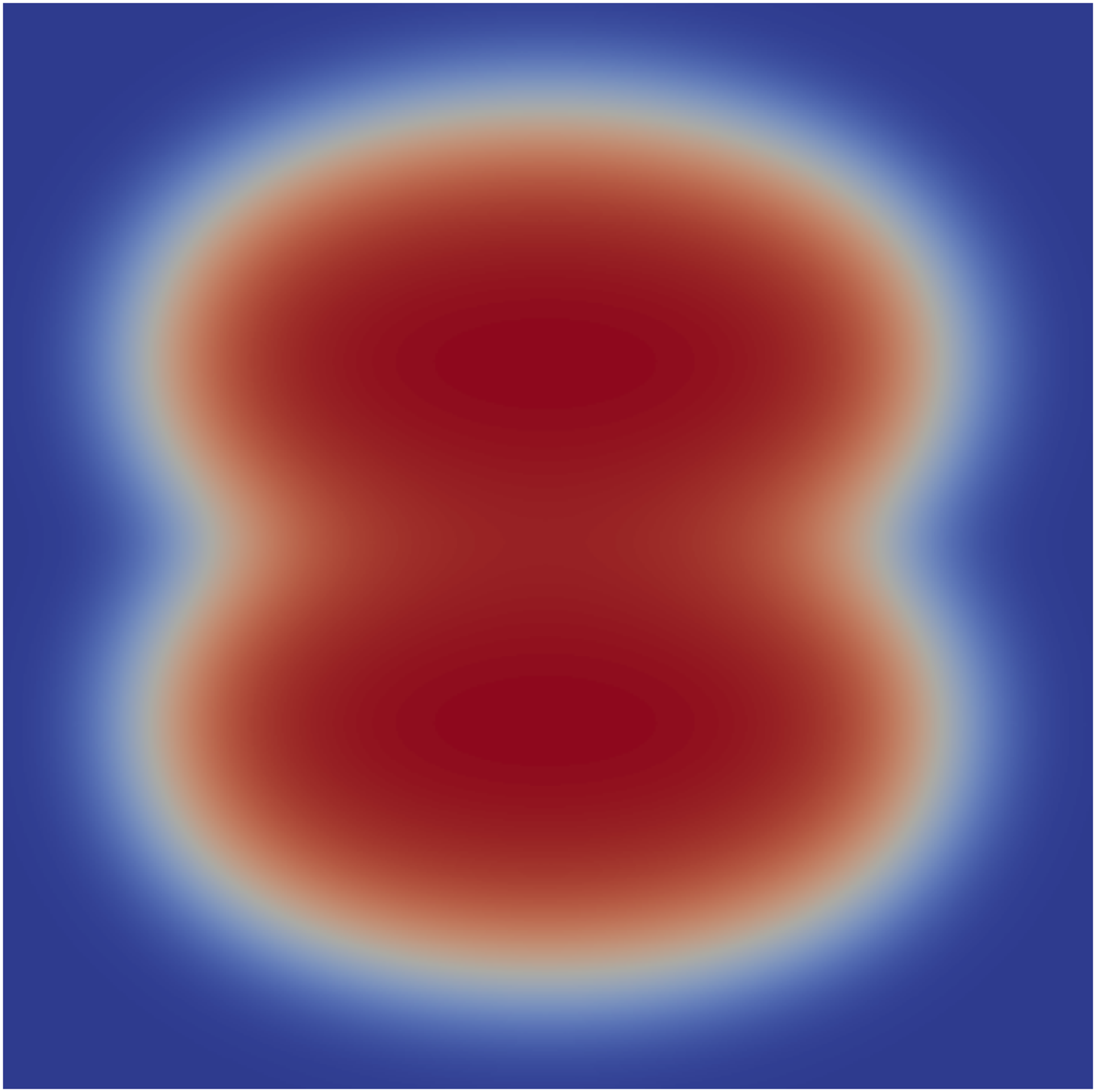}}\\
	\end{tabular} 
	\caption{Reaction-diffusion model with two ellipses as initial condition. Zoom in $[0.2,0.8]^2$ of tumour volume fraction for different values of $\alpha$ at different times using a constant nutrient source $S_\psi\equiv 0.5$. The range for the colorbar has been fixed to be $[0,1]$ for all plots.}
	\label{Fig_RD_coalescence}
\end{figure}

In the experiments shown so far, we have used a time-constant supply of nutrient, meaning a strictly increasing tumour mass over time. To observe the effects of varying $\alpha$ in a more dynamic setting, we still consider the initial condition \eqref{eq:ic_twoellipses} but now a periodic source of nutrient, i.e. we set:
\begin{equation}
    S_\psi(t)= \begin{cases}
    0.5 & \mbox{ if } \quad 1< t\leq 3\mbox{ or }5< t\leq 7\mbox{ or }9< t\leq 10,\\
    0 & \mbox{ otherwise}.
    \end{cases}\label{eq:BC_chi}
\end{equation}
The mass of the tumour and nutrient up to $T=10$ are depicted in Figure \ref{Fig_RD_mass_tumour_nutrient}, left and right plot, respectively. Regarding the tumour evolution, we observe conservation of mass up to $t=1$, because of no nutrient supply and homogeneous Neumann boundary conditions. For $1<t\leq 3$, we can see that, in the beginning, the tumour grows faster over time when $\alpha$ is smaller, and it grows faster for $\alpha$ larger as time passes. We notice that, when the nutrient is again not provided, for $3<t\leq 5$ (and for $7<t\leq 9$), the tumour keeps growing nevertheless, because there is still some nutrient in the domain, and it grows faster for larger $\alpha$. Regarding the nutrient, we see that, for $1<t\leq 3$, the nutrient consumption is approximately the same for all values of $\alpha$, while, for longer times, the larger the $\alpha$, the larger the nutrient uptake. This can be expected from the fact, that, for later times, the tumour mass is larger for $\alpha$ large and therefore it consumes more nutrient.

\begin{figure}[H]\centering
	\includegraphics[width=0.45\textwidth]{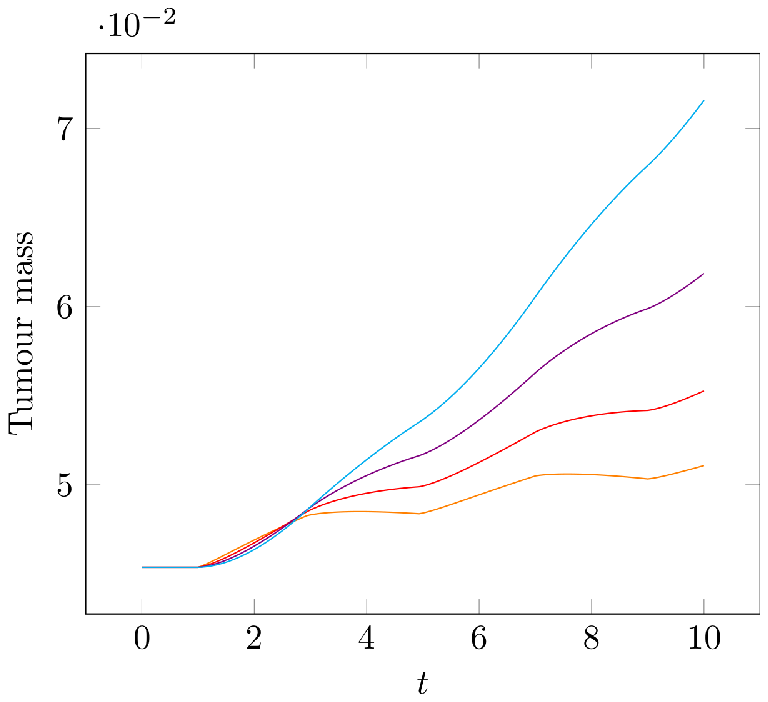}\qquad
	\includegraphics[width=0.468\textwidth]{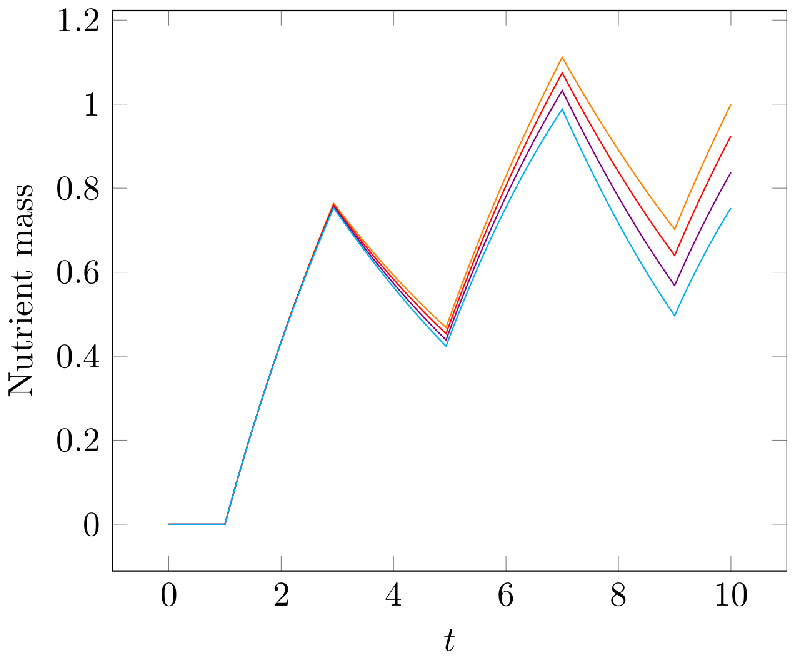}
	\caption[Initial condition with two ellipses.]{\revision{Reaction-diffusion model with two ellipses as initial condition. Tumour mass ($\int_\Omega\phi\,\textup{d}\vec{x}$) and nutrient mass ($\int_\Omega\psi\,\textup{d}\vec{x}$) over time obtained by giving a periodic nutrient source \eqref{eq:BC_chi}, for different values of the fractional exponent: \textcolor{orange}{-----} $\alpha=0.25$, \textcolor{red}{-----} $\alpha=0.5$, \textcolor{violet}{-----} $\alpha=0.75$, \textcolor{cyan}{-----} $\alpha=1$.}}
	\label{Fig_RD_mass_tumour_nutrient}
\end{figure}

\pagebreak
\subsection{Reaction-diffusion system with mechanical coupling in absence of treatment}\label{ssec:numexp_notreat}
The goal of this section is to show results of the reaction-diffusion model with a mechanical coupling as given in \eqref{eqn:ueq}, and \eqref{eqn:psieq} with $P_\phi=0$, in which case we have four variables. We use homogeneous Neumann boundary conditions for $\phi,\mu,\psi$, for $\vec{u}$ we use homogeneous Dirichlet boundary condition on the left boundary and homogeneous Neumann elsewhere. We give a constant nutrient source $S_\psi\equiv 0.5$. We consider both constant coefficients $M_{\phi}$, $M_{\psi}$ as from Table \ref{eq:param_val} and spatially-varying ones, given by $\tilde{M}_{\phi}=M_{\phi}\text{exp}(5(y-0.5))$, $\tilde{M}_{\psi}=M_{\psi}\text{exp}(5(y-0.5))$, with again  $M_{\phi}$, $M_{\psi}$ as from Table \ref{eq:param_val}. We note that both constants and non-constant coefficients take the same value at the centre of the domain, where we locate the irregularly shaped initial tumour mass
\begin{equation}\label{eq:icirregular}
\begin{cases}
\text{exp}\left(1- \frac{1}{1-f(\vec{x})}\right) & \mbox{if}\quad f(\vec{x})<1, -0.45<x<0.2, -0.4<y<0.35, \\
0 & \mbox{otherwise},
\end{cases}
\end{equation}
where $f(\vec{x})= \sin(6x+2y+1)(7x-0.2)^2+\sin(-8x+10y+1.1)(9x-0.1)^2$. This initial condition is depicted in the left plot of Figure \ref{Fig:irrshape}. As in the previous experiments, $\psi_0\equiv 0$. In this section, we compare the results when using $\alpha=0.25$ and $\alpha=1$.

Figure \ref{Fig_RD_MD_mass} shows the evolution of tumour mass and of the total displacement $\int_{\Omega} |\mathbf{u}|$\textrm{d}\vec{x} over time, when using constant and non-constant coefficients. In both cases, we observe that, apart from the very beginning, the tumour grows faster for $\alpha=1$, and consequently, the displacement of the tumour is larger in this case. It is then for $\alpha=1$ that, for later times, we can observe some difference between the case of constant and non-constant coefficients. The fact that the tumour grows more when $\alpha=1$ can also be seen in the cross sections along the $y$-axis at time $T=10$ in Figure \ref{Fig_crosssectionsmech} (tumour density in the left plot and modulus of the displacement in the right plot), where the dotted lines denote the corresponding initial conditions: we note that for $\alpha=0.25$ the shapes of the solutions at $T=10$ are closer to the initial conditions than for $\alpha=1$, for both constant and spatially-varying coefficients. Furthermore, in Figure \ref{Fig_crosssectionsmech} we see that the spatial variability of the coefficients (in the $y$-direction) translates in a more pronounced asymmetry of the solution with respect to the $y$-axis. Regarding the displacement, the asymmetry is more evident when $\alpha=1$, because there the magnitude of the displacement is larger compared to when $\alpha=0.25$.


\begin{figure}[H]\centering
\includegraphics[width=0.45\textwidth]{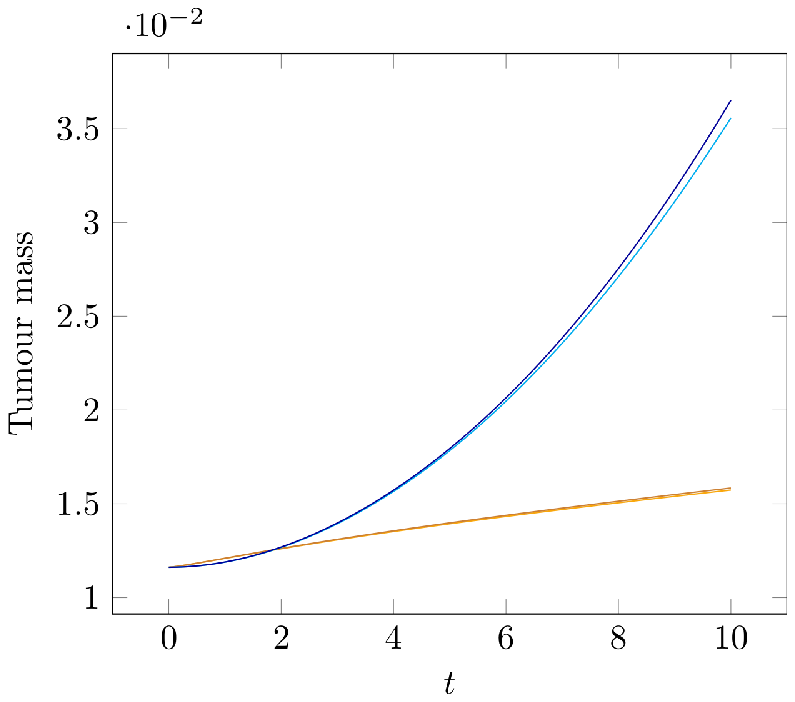}\qquad
\includegraphics[width=0.45\textwidth]{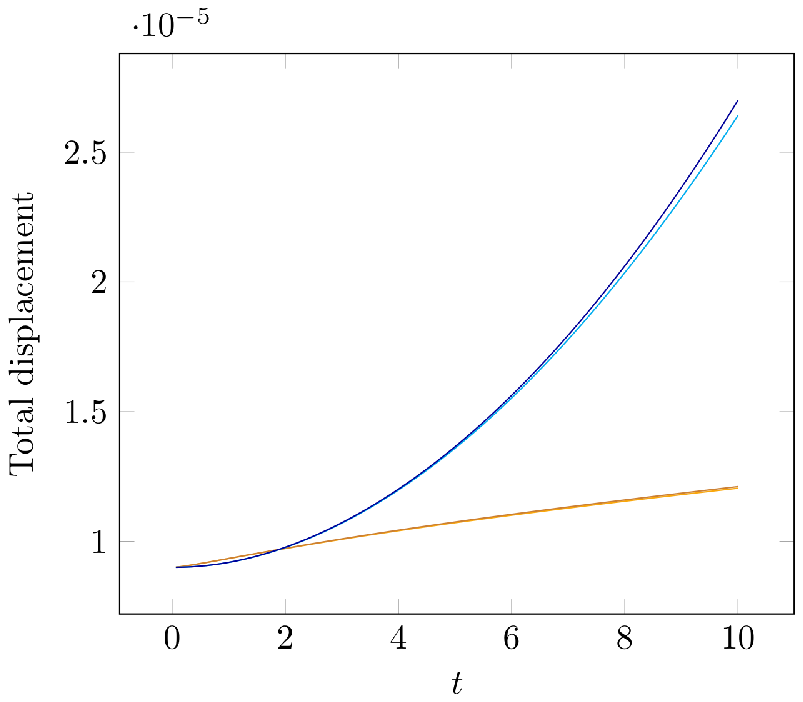}
\caption[Reaction-diffusion model with mechanical coupling ]{\revision{Reaction-diffusion model with mechanical coupling and initial condition \eqref{eq:icirregular}. Evolution of tumour mass ($\int_\Omega\phi\,\textup{d}\vec{x}$) and of total displacement ($\int_\Omega |\vec{u}|\,\textup{d}\vec{x}$) over time: \textcolor{yellow!30!orange}{-----} $\alpha=0.25$ and constant coefficients, \textcolor{orange!20!brown}{-----} $\alpha=0.25$ and spatially varying coefficients, \textcolor{cyan}{-----} $\alpha=1$ and constant coefficients, \textcolor{blue!60!black}{-----} $\alpha=1$ and spatially varying coefficients.}}
\label{Fig_RD_MD_mass}
\end{figure}

\begin{figure}[H]\centering
\includegraphics[height=4.15cm]{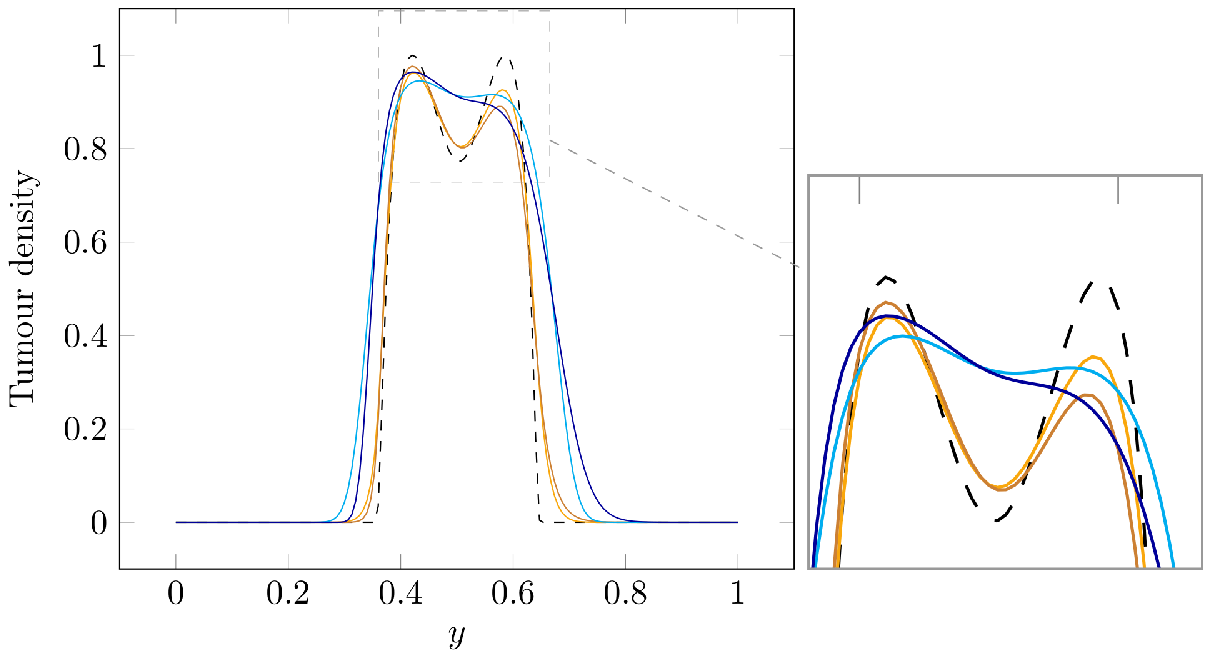} \quad 
\includegraphics[height=4.55cm]{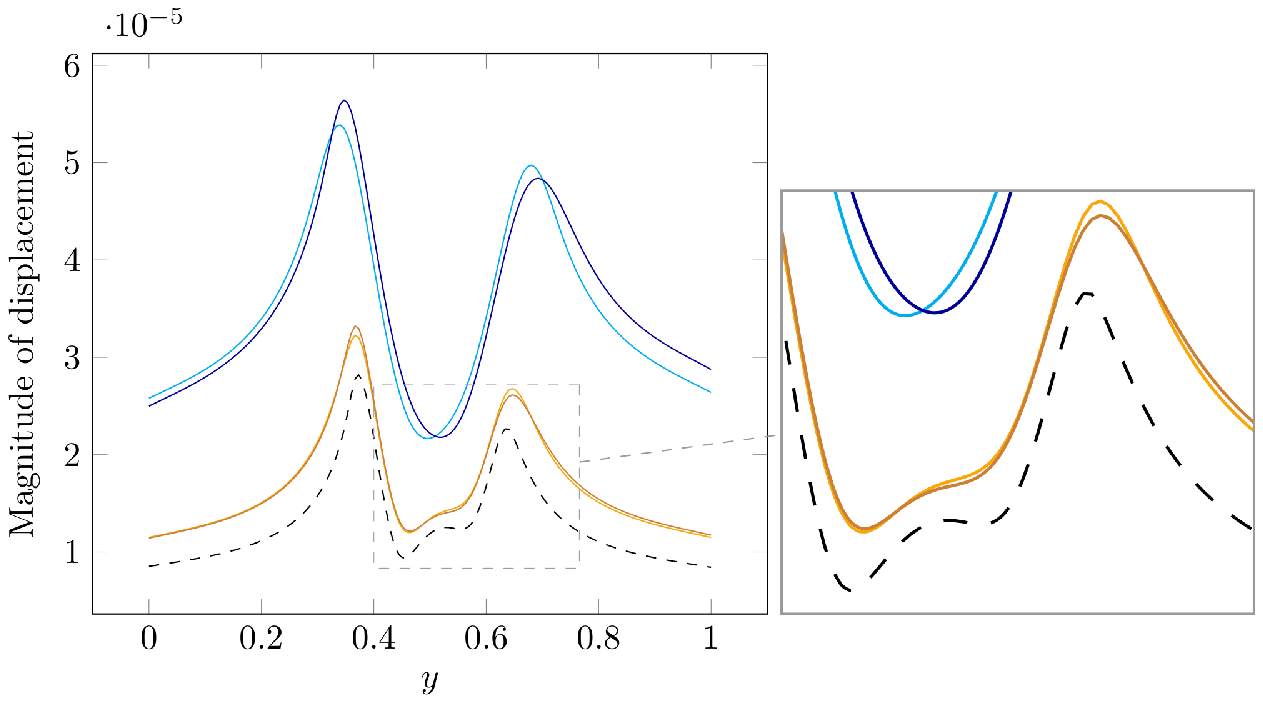}
\caption[RDM]{\revision{Reaction-diffusion model with mechanical coupling and initial condition \eqref{eq:icirregular}. Cross sections along the $y-$axis at $T=10$: \textcolor{yellow!30!orange}{-----} $\alpha=0.25$, constant coefficients, \textcolor{orange!20!brown}{-----} $\alpha=0.25$, spatially varying coefficients \textcolor{cyan}{-----} $\alpha=1$, constant coefficients, \textcolor{blue!60!black}{-----} $\alpha=1$, spatially varying coefficients. Left: cross section of the tumour volume fraction. Right: cross section of the modulus of the displacement. The dashed lines denote the corresponding initial conditions.}}
\label{Fig_crosssectionsmech}
\end{figure}

\subsection{Reaction-diffusion system with mechanical coupling and chemotherapy}\label{ssec:numexp_treat}
In this section, we show simulations of a more realistic situation and include the treatment of cancer by giving chemotherapeutic agents. In all, we have five unknowns, solving \eqref{eqn:phieq}-\eqref{eqn:chieq}. In the previous sections, nutrient supply with a source term for the nutrient could be thought as a situation close to an \textsl{in vitro} setting, where nutrients are added directly in the wells. Here, we assume nutrients and chemotherapeutic agents to be  supplied through some blood vessels which are around the tumour area, and so we take zero source functions $S_\psi=S_\chi\equiv 0$ and non-homogeneous Dirichlet boundary conditions for $\psi$ and $\chi$, over the whole boundary. We take $\tilde{\psi}_b\equiv 2$ as Dirichlet boundary condition for the nutrient and 
\begin{equation*}
       \tilde{\chi}_b(t)= \begin{cases}
    1 & \mbox{ if } \quad t\leq 2 \mbox{ or }6< t\leq 8\mbox{ or }12< t\leq 14,\\
    0 & \mbox{ otherwise,}
    \end{cases}
\end{equation*}
for the chemotherapeutic agents, which are usually administered in cycles.

The boundary conditions for $\phi,\mu$ and $\vec{u}$ are as in the previous section. We plot the mass of the tumour and chemotherapy for different values of $\alpha$. The initial condition for the tumour is the one with irregular shape as in the previous section and in the left plot of Figure \ref{Fig:irrshape}. We take $\chi_0\equiv 0$ for the chemotherapeutic agents. For the nutrient, \revision{we take an initial condition with values close to the concentration of the nutrient at equilibrium}, namely, we take $\psi_0(\vec{x})=2-0.5x(1-x)y(1-y)$.

\begin{figure}[!htb]
	\centering
	\includegraphics[height=0.3\textwidth]{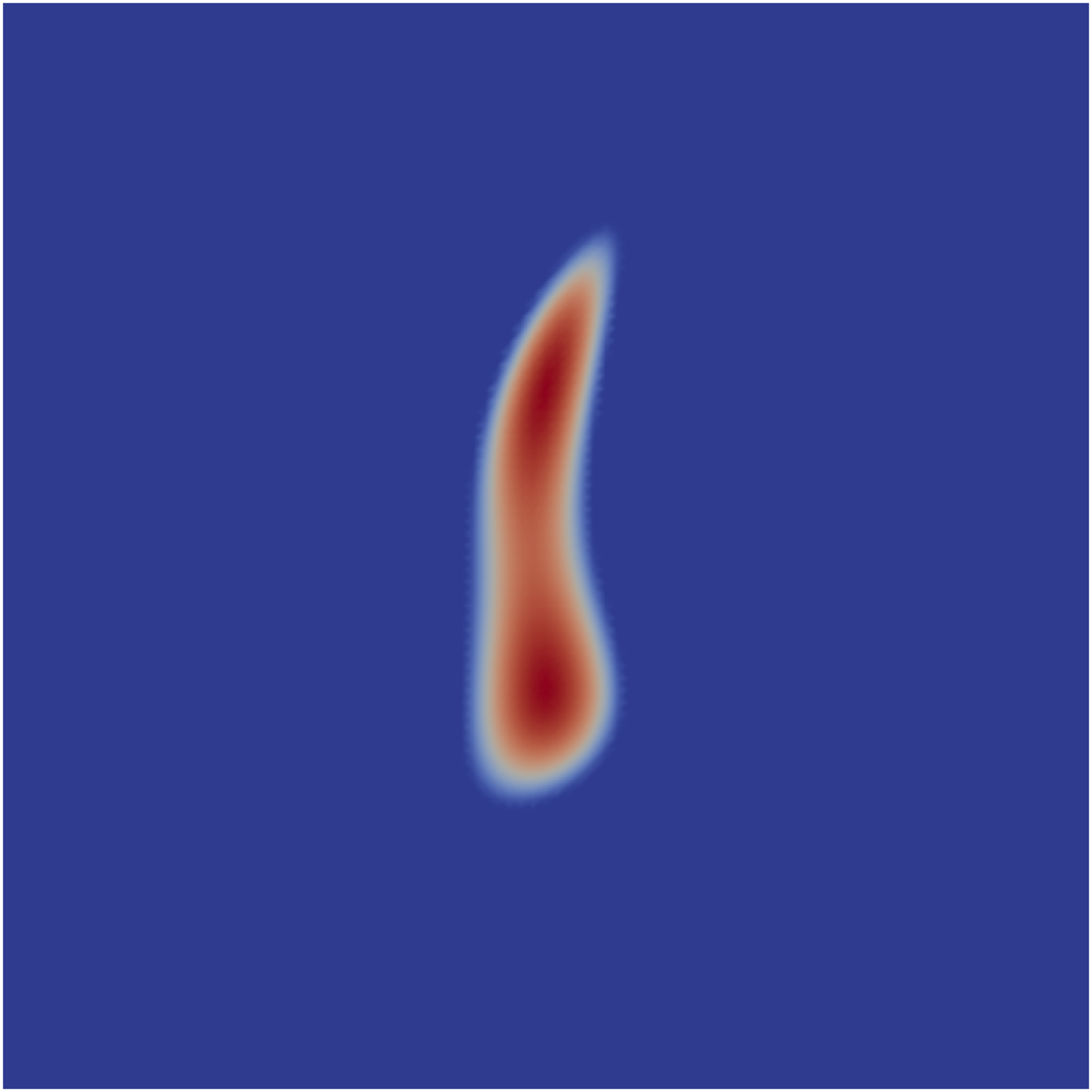}
	\includegraphics[height=0.3\textwidth]{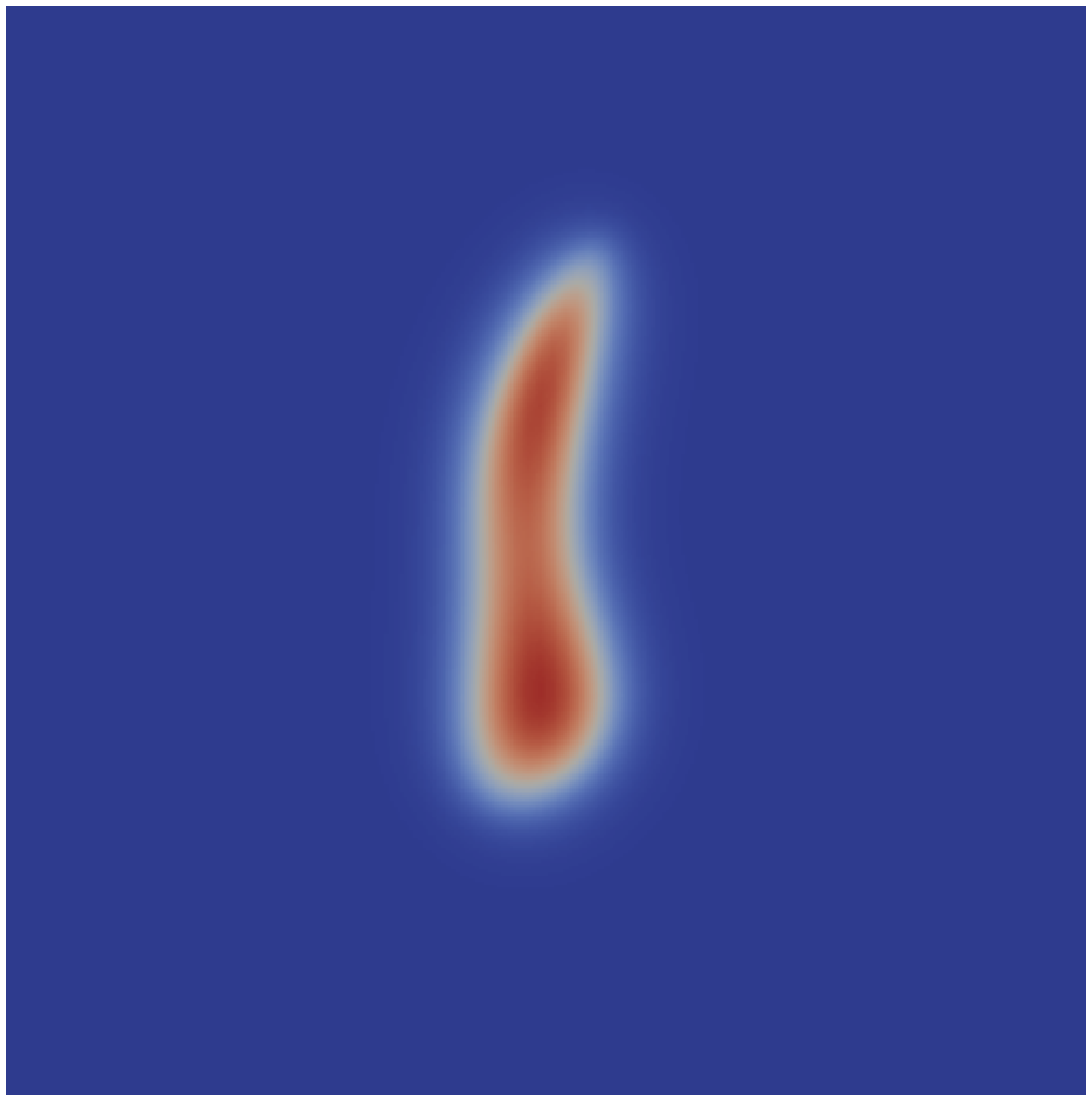}
	\includegraphics[height=0.3\textwidth]{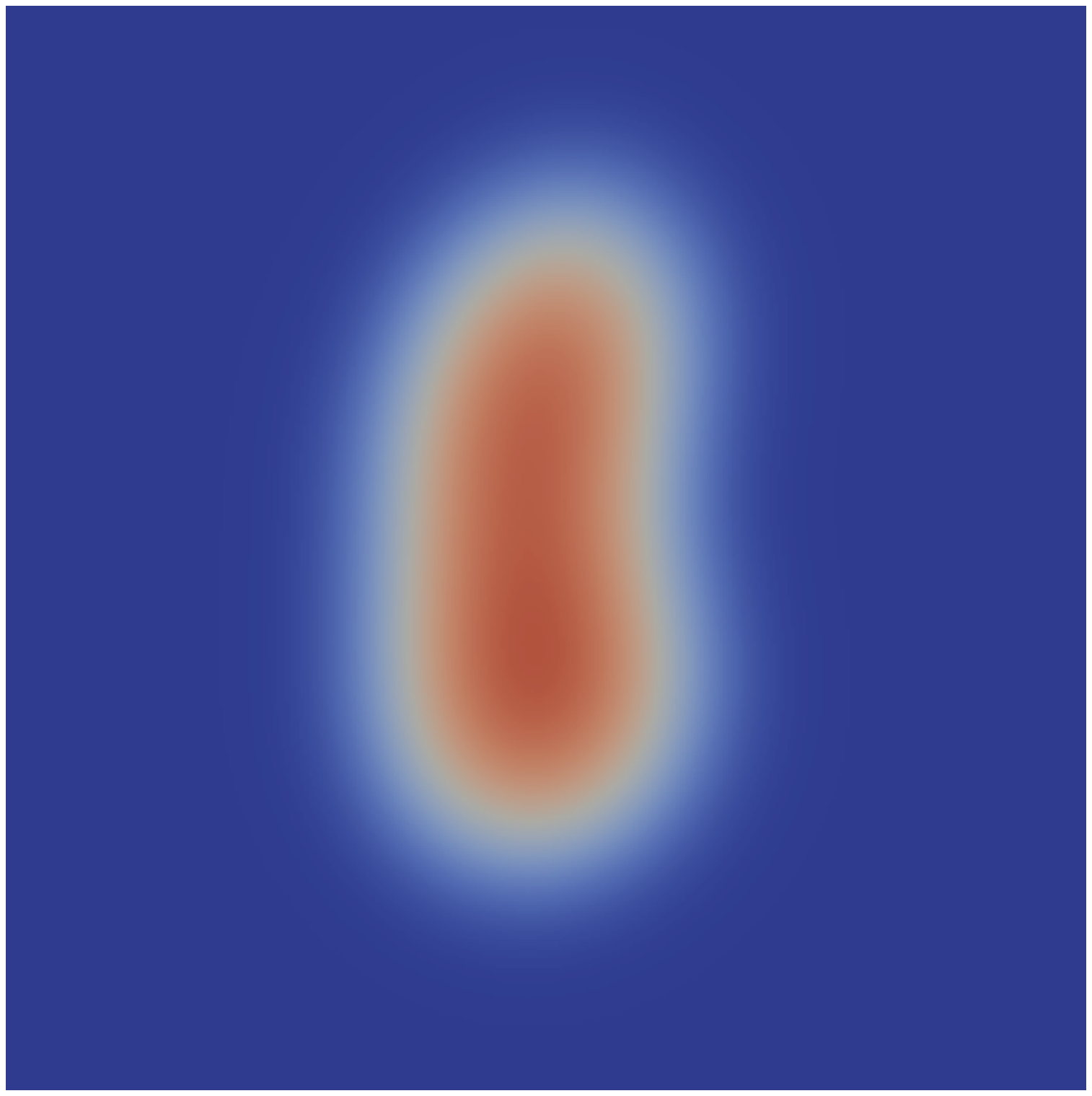}
	\caption{Left: zoom in $[0.2,0.8]^2$ of initial condition for the tumour with irregular shape. Centre and right: reaction-diffusion model with mechanical coupling and treatment, zoom in $[0.2,0.8]^2$ of tumour density at $T=20$ for $\alpha=0.25$ (centre) and $\alpha=1$ (right). The range for the colorbar has been fixed to be $[0,1]$ for all plots.}
	\label{Fig:irrshape}
\end{figure}

\pagebreak

The tumour densities at $T=20$ for $\alpha=0.25$ and $\alpha=1$ are shown in the centre and right plot of Figure \ref{Fig:irrshape}, respectively (the left plot showing the initial condition). Figure \ref{Fig_RD_MD_NU_CTD_tumour_chemo} shows the evolution of the tumour mass (left) and of the mass of chemotherapeutic agents for different values of the fractional exponent. From the left plot, we see that the response of the tumour to the therapy in the model depends sensitively on $\alpha$: the smaller the $\alpha$, the more nonlinear the responses to applying or removing the supply of chemotherapeutic substances. One could also think about using a piecewise constant $\alpha$, one for when chemotherapy is supplied, one when it is not, in order to model a different response of the tumour in these two scenarios. The mass of chemotherapy over time is instead very similar for all values of $\alpha$. In particular, when administration of the chemotherapeutic agents is interrupted, their concentration drops quickly to $0$ because of the degradation term $(-N_{\chi}\chi)$ in \eqref{eqn:chieq}.

\begin{figure}[!htb]\centering
\includegraphics[height=0.4\textwidth]{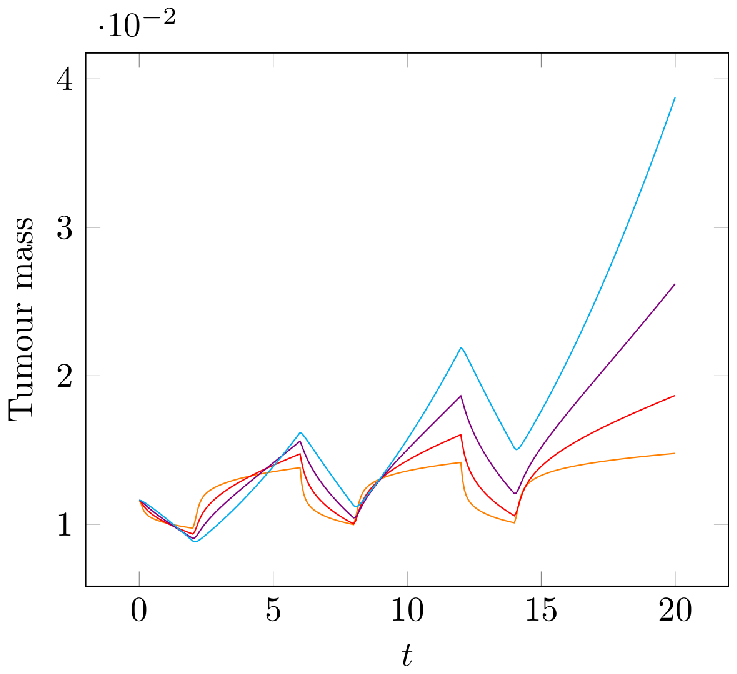}\qquad
\includegraphics[height=0.375\textwidth]{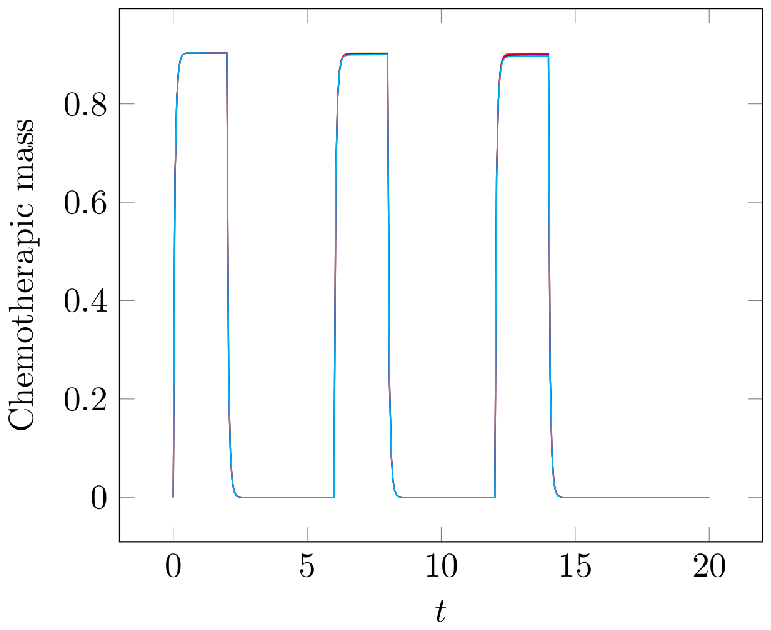}
\caption[RDMDchemof]{\revision{Reaction-diffusion model with mechanical coupling and treatment, and initial condition \eqref{eq:icirregular}. Tumour mass ($\int_\Omega\phi\,\textup{d}\vec{x}$) and chemotherapy mass ($\int_\Omega\chi\,\textup{d}\vec{x}$) over time: \textcolor{orange}{-----} $\alpha=0.25$, \textcolor{red}{-----} $\alpha=0.5$, \textcolor{violet}{-----} $\alpha=0.75$, \textcolor{cyan}{-----} $\alpha=1$.}}\label{Fig_RD_MD_NU_CTD_tumour_chemo}
\end{figure}

\section{Conclusions}
We have presented a new model for tumour growth with fractional time derivatives, including mechanical effects and treatment by chemotherapy. Existence and uniqueness of a weak solution to the coupled, nonlinear model are obtained by a Galerkin method. Numerical experiments, based on low order finite elements in space and convolution quadrature in time, show that the order of the fractional time derivative influences strongly the evolution. Using the fractional order as an additional parameter results in a larger model class. This can be of future interest for calibration of the model parameters by experimental data.

\section*{Funding}
Deutsche Forschungsgemeinschaft (DFG) through TUM International Graduate School of Science and Engineering (GSC 81); Laura Bassi Postdoctoral Fellowship (Technical University of Munich; to M.L.R.); and DFG (WO-671 11-1 to M.F., L.S. and B.W.).

\newpage

\bibliography{references.bib}
\bibliographystyle{siam} 

\newpage 

\appendix
\section{Existence Result for Nonlinear Finite Dimensional System}\label{appendix:existence_finitedim}
Consider the multi-order fractional differential system of the form
\begin{eqnarray}\label{finite_dimension:de}
\begin{aligned}
\frac{d}{dt}\left(g_{1-\alpha_k}* (X_k(t)-X_{k,0})\right)(t) &= F_k(t,X_1(t),\ldots,X_m(t)), &k=1,\ldots,m,\\
\left(g_{1-\alpha_k}*(X_k-X_{k,0})\right)(0)&=0, &k=1,\ldots,m,
\end{aligned}
\end{eqnarray}
where $0<\alpha_k\leq 1$, $X_k:[0,T]\rightarrow \mathbb{R}$, $F_k:[0,T]\times \R^m\rightarrow\R$ is such that $F_k(\cdot,X_1,\ldots,X_m)\in\L_2(0,T)$ and it is Lipschitz in the other variables. Existence of a solution to a similar system with continuous function $F_k$ is given in \cite[Lemma 5.3]{diethelm2010analysis}, here we prove the result in the vector form. In the vector notation, the system \eqref{finite_dimension:de} can be written as
\begin{equation}
    \begin{aligned}
    D^{\vec{\alpha}}(\vec{X}-\vec{X}_0) &= \vec{F}(t,\vec{X}(t)),\\
   \left(\vec{k}*(\vec{X}-\vec{X}_0)\right)(0) &= \vec{0},
    \end{aligned}\label{FD:vec_DE}
\end{equation}
where 
\begin{equation*}
    \begin{aligned}
    D^{\vec{\alpha}}(\vec{X}-\vec{X}_0)= \frac{\dd}{\dd t}\left( \vec{k}*(\vec{X}-\vec{X}_0)\right)(t),\
    \vec{X}(t) =  \left( {\begin{array}{c}
   X_1(t) \\
   \vdots\\
   X_m(t) \\
  \end{array} } \right),\ 
  \vec{X}_0 =  \left( {\begin{array}{c}
   X_{1,0} \\
   \vdots\\
   X_{m,0} \\
  \end{array} } \right),\\
  \vec{k}(t)= \left( {\begin{array}{cccc}
   g_{1-\alpha_1} & 0 & \ldots & 0 \\
   \vdots&&&\\
   0&\ldots&0& g_{1-\alpha_m}\\
  \end{array} } \right),\
      \vec{F}(t,\vec{X}(t)) =  \left( {\begin{array}{c}
   F_1(t,X_1(t),\ldots,X_m(t)) \\
   \vdots\\
   F_m(t,X_1(t),\ldots,X_m(t)) \\
  \end{array} } \right).\\
    \end{aligned}
\end{equation*}
\begin{lemma}\label{FD:equivalence_DE_IE}
Let $\vec{F}(\cdot,\vec{X})\in\L_2(0,T;\R^m)$ for any $\vec{X}\in\R^m$ and  $\vec{X}(\cdot)\in\L_{2}(0,T;\R^m)$. Then $\vec{X}(t)$ satisfies \eqref{FD:vec_DE} if, and only if, $\vec{X}(t)$ satisfies the following Volterra integral equation
\begin{eqnarray}\label{FD:vec_IE}
\vec{X}(t)  = \vec{X}_0 + \int_0^t \vec{l}(t-s)\vec{F}(s,\vec{X}(s))\mathrm{d}s,
\end{eqnarray}
where 
\begin{equation*}
    \begin{aligned}
  \vec{l}(t)= \left( {\begin{array}{cccc}
   g_{\alpha_1} & 0 & \ldots & 0 \\
   \vdots&&&\\
   0&\ldots&0& g_{\alpha_m}\\
  \end{array} } \right).
    \end{aligned}
\end{equation*}
\end{lemma}
\begin{proof}
First we prove necessity. Let $\vec{X}(\cdot)\in\L_{2}(0,T;\R^m)$ satisfy \eqref{FD:vec_DE}. With the initial condition $\left(\vec{k}*(\vec{X}-\vec{X}_0)\right)(0)=0$ and the result $\vec{l}*\vec{k}=\vec{1}$, we have
\begin{eqnarray}\label{FD:libnitz_rule}
\vec{l}*\frac{\dd}{\dd t} (\vec{k}*(\vec{X}-\vec{X}_0))(t) = \frac{\dd}{\dd t} (\vec{l}*\vec{k}*(\vec{X}-\vec{X}_0))(t).
\end{eqnarray} 
Taking a convolution with $\vec{l}$ on both sides of \eqref{FD:vec_DE}, using \eqref{FD:libnitz_rule}, we obtain \eqref{FD:vec_IE}, and hence the necessity is proved.

Now we prove the sufficiency. Let $\vec{X}(\cdot)\in\L_{2}(0,T;\R^m)$ satisfy \eqref{FD:vec_IE}. Taking a convolution with $\vec{k}$ and differentiating on both sides of \eqref{FD:vec_IE}, using $\vec{l}*\vec{k}=\vec{1}$, we arrive at \eqref{FD:vec_DE}. Further from the continuity of $\left(\vec{1}*\vec{F}(t,\vec{X})\right)(t)$, we have $\left(\vec{1}*\vec{F}(t,\vec{X}(t))\right)(0)=0$, which implies $\left(\vec{k}*(\vec{X}-\vec{X}_0)\right)(0)=\vec{0}$, and this proves the sufficiency part.
\end{proof}

\begin{lemma}{(Banach Fixed point theorem)}[{\cite[Theorem 1.9]{kilbas2006theory}}] Let $(\Ba,d)$ be a nonempty complete metric space, let $0\leq \omega<1$, and let $\Lambda:\Ba\rightarrow \Ba$ be a map such that, for every $\varphi_1, \varphi_2\in \Ba$, the relation
$$d(\Lambda\varphi_1,\Lambda\varphi_2)\leq \omega d(\varphi_1,\varphi_2)$$
holds. Then the operator $\Lambda$ has a unique fixed point $\varphi^*\in \Ba$. Furthermore, if $\{\Lambda^k\}_{k\in\N}$ is the sequence of operators defined by
$$\Lambda^1=\Lambda,\quad \Lambda^k=\Lambda\Lambda^{k-1},\, \forall k\in \N\backslash \{1\},$$
then, for any $\varphi_0\in U$, the sequence $\{\Lambda^k \varphi_0\}_{k=1}^{\infty}$ converges to the above fixed point $\varphi^*$.
        \label{Lem_banach_fixed_point_theorem}
\end{lemma}

\begin{theorem}
The initial value problem given by the system of multi-order fractional differential equations along with the initial condition (\ref{finite_dimension:de}) has a uniquely determined solution on the interval $[0,T]$.
\end{theorem}
\begin{proof}
To prove the existence for the nonlinear differential equation \eqref{FD:vec_DE} it is enough to show the existence to its equivalent integral equation \eqref{FD:vec_IE} as shown in Lemma \ref{FD:equivalence_DE_IE}. The nonlinear integral equation is converted to a linear integral equation and Banach fixed point theorem is used to show the existence of the unique solution to \eqref{FD:vec_IE}.
\par For a particular $\vec{Y}\in\L_2(0,T;\R^m)$, we obtain the corresponding linear equation to (\ref{FD:vec_IE}) as 
\begin{eqnarray}\label{FD:linear_system}
\vec{X}(t)  = \vec{X}_0 + \int_0^t \vec{l}(t-s)\vec{F}(s,\vec{Y}(s))\mathrm{d}s. 
\end{eqnarray}
Define the operator $\Lambda$ on $\Ba:=\L_2(0,T_h;\R^m)$ for some $T_h>0$ as
\begin{equation*}
     \Lambda\vec{Y}(t):= \vec{X}_0 + \int_0^t \vec{l}(t-s)\vec{F}(s,\vec{Y}(s))\mathrm{d}s. 
\end{equation*}
Using Young's inequality for convolution \eqref{inequality:Youngs}, we have
\begin{equation}\label{FD:est:map}
\begin{aligned}
    \|\Lambda\vec{Y}\|^2_{\L_2(0,T_h;\R^m)} &\leq C\left(\|\vec{X}_0\|_{\R^m} + \|\vec{l}\|_{\L_1(0,T_h;\R^m)}\|\vec{F}(t,\vec{Y})\|_{\L_2(0,T_h;\R^m)}\right),\\
    &\leq C\left(\|\vec{X}_0\|_{\R^m} + \|\vec{l}\|_{\L_1(0,T;\R^m)}\left(L_{\vec{F}}\|\vec{Y}\|_{\L_2(0,T_h;\R^m)}+\|\vec{F}(t,\vec{0})\|_{\L_2(0,T;\R^m)}\right)\right),\\
    &\leq C \|\vec{Y}\|_{\L_2(0,T_h;\R^m)}.
\end{aligned}
\end{equation}
This means that $\Lambda$ maps $\Ba$ into itself. Further we get
\begin{equation*}
    \begin{aligned}
     \|\Lambda\vec{Y}-\Lambda\vec{Z}\|_{\L_2(0,T_h;\R^m)} &\leq \|\left(\vec{l}*(\vec{F}(t,\vec{Y})-\vec{F}(t,\vec{Z}))\right)(t)\|_{\L_2(0,T_h;\R^m)}, \\
   &\leq L_{\vec{F}}  \|\vec{l}\|_{\L_1(0,T_h;\R^m)} \|\vec{Y}-\vec{Z}\|_{\L_2(0,T_h;\R^m)}.
    \end{aligned}
\end{equation*}
We choose $T_h>0$ such that $L_{\vec{F}} \|\vec{l}\|_{\L_1(0,T_h;\R^m)}<1$, which means $\Lambda$ is a contraction.
By Lemma \ref{Lem_banach_fixed_point_theorem} there exists a unique solution $\vec{X}\in\L_2(0,T_h;\R^m)$ to (\ref{FD:linear_system}) on the interval $[0,T_h]$. 
Further we see that for any $\tau\in(0,T)$, we have by proceeding as before in \eqref{FD:est:map}
\begin{equation*}
    \begin{aligned}
    \|\vec{X}\|_{\L_2(0,\tau;\R^m)}\leq C, 
    \end{aligned}
\end{equation*}
for some constant independent of $\tau$. Therefore we obtain $\vec{X}\in\L_2(0,T;\R^m)$. Then $\vec{F}\in\L_2(0,T;\R^m)$, and the initial condition implies $\vec{X}\in\W^{\alpha}_{2,2}(0,T;\vec{X}_0,\R^m,\R^m)$.
\end{proof}

\end{document}